\documentclass[12pt,reqno]{amsart}
\usepackage{listings}
\usepackage{xr}
\usepackage{hyperref}
\usepackage{latexsym,amssymb,braket}
\usepackage{mathrsfs}
\usepackage{enumerate}
\usepackage{enumitem}
\usepackage{bm}
\usepackage[all]{xy}
\allowdisplaybreaks[1]
\usepackage{longtable}
\usepackage{amsmath}
\usepackage{float}
\usepackage{bbm}
\usepackage{mathtools}
\usepackage{comment}
\usepackage{hyperref}

\usepackage{listings}
\usepackage{listings,jvlisting}
\usepackage{listings, xcolor}
\usepackage{rotating}
\newtheorem{thm}{Theorem}[section]
\newtheorem{lem}[thm]{Lemma}

\newtheorem{prob}[thm]{Problem}

\theoremstyle{definition}

\newtheorem{defn}[thm]{Definition}
\newtheorem{rem}[thm]{Remark}
\newtheorem{exam}[thm]{Example}

\numberwithin{equation}{section}

\newcommand{\C}{\mathbb{C}}
\newcommand{\T}{\mathbb{T}}

\DeclareMathOperator{\End}{End}

\DeclareMathOperator{\Irr}{Irr}

\DeclareMathOperator{\Mor}{Mor}

\DeclareMathOperator{\Rep}{Rep}

\newcommand{\Cstar}{\mathrm{C}^*\!}
\newcommand{\Gizumi}{G_{\textrm{IK}}}

\newcommand{\columnvec}[2]{
	\begin{pmatrix}
		#1 \\ #2
	\end{pmatrix}
}

\newcommand{\dualcolumnvecvec}[4]{
	\widehat{
		\begin{pmatrix}
			#1 \\ #2 \\ #3 \\ #4
		\end{pmatrix}
	}
}

\makeatletter
\@namedef{subjclassname@2020}{%
	\textup{2020} Mathematics Subject Classification}
\makeatother
\begin{document}

\title{Complete Isocategorical Classification of Groups of Order $64$ via GAP}

\author[S. Sato]{Shoki Sato}
\address{
Department of Mathematics, Waseda University,
Tokyo \mbox{169-8050},
JAPAN}
\email{barley-tea-3416@moegi.waseda.jp}


\maketitle

\begin{abstract}
	The classification of finite groups under monoidal equivalence is a fundamental topic in the study of finite quantum groups. 
	While a complete classification has been established for all groups of order strictly less than 64, the case for order 64 has remained limited to the construction of specific examples. 
	In this study, we achieve the complete classification for groups of order 64 by developing an original computational approach using GAP. 
	We describe our methodology and demonstrate that there exist exactly two pairs of non-isomorphic isocategorical groups of this order.
\end{abstract}

\section{Introduction}

The classification of finite groups under monoidal equivalence is a fundamental problem in the study of finite quantum groups and fusion categories. 
Two groups $G$ and $H$ are said to be monoidally equivalent (or isocategorical) if their respective representation categories $\Rep(G)$ and $\Rep(H)$ are equivalent as tensor categories.

To date, a complete classification based on this equivalence has been established for all groups of order strictly less than 64 \cite{GoyvaertsMeir2014}. 
However, the case of order $64 = 2^6$ has remained a significant open challenge. 
The difficulty of this specific order arises from the structural requirements for a group $G$ to possess a non-trivial monoidally equivalent partner. 
It is a well-established result that if $\Rep(G) \simeq \Rep(H)$ as tensor categories, then $H$ must be isomorphic to a "twisted" group $G^\omega$ constructed via a $G$-invariant non-degenerate 2-cocycle $\omega \in Z^2(\widehat{N}, \mathbb{T})$ on a normal abelian subgroup $N \trianglelefteq G$ \cite{etingof2000isocategoricalgroups,Wassermann1989}.

Furthermore, order $64$ is recognized as the minimal order for which non-isomorphic, yet monoidally equivalent, pairs of groups exist \cite{etingof2000isocategoricalgroups}. 
The most prominent examples are the Izumi--Kosaki pairs, which were originally identified through the study of subfactors \cite{IzumiKosaki2002}. 
Despite the theoretical importance of these examples, the sheer number of groups of order $64$, totaling $267$ distinct isomorphism classes, and the complexity of their internal subgroup structures have prevented an exhaustive classification until now.

In this paper, we resolve this problem by combining elementary algebraic arguments with a systematic computational approach. 
We have developed an original suite of functions in GAP (Groups, Algorithms, and Programming) to exhaustively scan all \(267\) groups of order \(64\). 
By evaluating $G$-invariance and non-degeneracy conditions across candidate normal subgroups, we successfully achieve the complete classification of groups of order \(64\) under monoidal equivalence.

\section{Preliminaries}
\subsection{Finite quantum group}
Let \(A\) be a finite dimensional \(\Cstar\)-algebra and \(\Delta\colon A\to A\otimes A\) a unital faithful \(*\)-homomorphism. 
A pair \((A,\Delta)\) is called \textit{a finite quantum group} if \(\Delta\) satisfies the cancellation properties \(\Delta(A)(1_A\otimes A)=A\otimes A\) and \(\Delta(A)(A\otimes 1_A)=A\otimes A\). 
It is denoted by \(G\coloneqq(A,\Delta)\) and \(A\coloneqq C(G)\) as is standard.
Let \(\iota\colon C(G)\to C(G)\) be the identity map.
The two maps \(S\colon C(G)\to C(G)\), \(\varepsilon\colon C(G)\to\C\) are respectively called the  \textit{antipode}, and the \textit{counit} if they satisfy the equalities
\begin{align*}
(\varepsilon\otimes\iota)\Delta(a)=(\iota\otimes\varepsilon)\Delta(a)=a \\
m(S\otimes\iota)\Delta(a)=m(\iota\otimes S)\Delta(a)=\varepsilon(a)1
\end{align*}
for all \(a\in C(G)\), where \(m\colon C(G)\otimes C(G)\to C(G)\) is the multiplication map.
It is well known that any finite quantum group possesses a unique antipode and counit. 

\begin{exam}
	Let \(G\) be a finite group and \(C(G)\) denote the unital \(\Cstar\)-algebra of \(\C\)-valued functions on \(G\) equipped with pointwise operations. We define the comultiplication \(\Delta\colon C(G)\to C(G)\otimes C(G)\simeq C(G\times G)\) by \(\Delta(f)(g,h)\coloneqq f(gh)\) for \(f\in C(G)\) and \(g,h\in G\). It is straightforward to verify that  \((C(G),\Delta)\) is a finite quantum group.
\end{exam}

\begin{exam}\label{abelian finite group}
	Let \(G\) be a finite group and \(\C[G]\) the group algebra of \(G\). 
	We define the comultiplication  \(\Delta\colon \mathbb{C}[G]\to\mathbb{C}[G]\otimes\mathbb{C}[G]\) by \(\Delta(g)\coloneqq g\otimes g\) for \(g\in G\). 
	Then \((\mathbb{C}[G],\Delta)\) is a finite quantum group. When \(G\) is abelian, the Foulier transform induces an isomorphism \(\mathbb{C}[G]\simeq C(\widehat{G})\) as unital \(\Cstar\)-algebras, where \(\widehat{G}\) denotes the Pontryagin dual of \(G\). 
\end{exam}

A finite quantum group \((A,\Delta)\) is called \textit{cocommutative} if \(\Delta\) satisfies the equation \(\Delta=\Delta^{\mathrm{op}}\), where \(\Delta^{\mathrm{op}}\coloneqq\Sigma\Delta\) and \(\Sigma\colon A\otimes A\to A\otimes A\) is the flip map. It is known that \((A,\Delta)\) is cocommutative if and only if 
it is isomorphic, as a finite quantum group, to the group algebra $(\mathbb{C}[G], \Delta)$ for some finite group $G$. 
As seen in Example \ref{abelian finite group}, if $G$ is an abelian group, then the function algebra $(C(G), \Delta)$ is cocommutative.

Let \(G\) be a finite quantum group and \(\widehat{A}\coloneqq C(G)^*\). 
For any $\varphi, \psi \in \widehat{A}$, we define the product $\varphi \cdot \psi$, the involution $\varphi^*$, and the dual comultiplication $\widehat{\Delta}\colon \widehat{A} \to \widehat{A} \otimes \widehat{A}$ by
\begin{align*}
(\varphi\cdot\psi)(a)\coloneqq(\varphi\otimes\psi)(\Delta(a)), \\
\varphi^*(a)\coloneqq \overline{\varphi}(S(a)), \\
\widehat{\Delta}(\varphi)(a\otimes b)\coloneqq\varphi(ab).
\end{align*}
Then \((\widehat{A},\widehat{\Delta})\) forms a finite quantum group, referred to as the \textit{dual quantum group of $G$} and denoted by $\widehat{G} = (\widehat{A}, \widehat{\Delta})$. 
In the case where $G$ is a finite group, there exists a canonical isomorphism
\begin{equation}\label{iso1}
C(\widehat{G})\to\mathbb{C}[G],\ \varphi\to\sum_{g\in G}\varphi(g)g.
\end{equation}

\subsection{Representation category}
Let \(G\) be a finite quantum group and \(\mathcal{H}\) a finite dimensional Hilbert space.
A unitary operator $U \in B(\mathcal{H}) \otimes C(G)$ is called \textit{a unitary representation of $G$} if \(U\) satisfies the equality
\[(\iota\otimes\Delta)(U)=U_{12}U_{13},\]
where $U_{12} \coloneqq U \otimes 1_{C(G)}$ and $U_{13} \coloneqq (\text{id} \otimes \Sigma)(U \otimes 1_{C(G)})$.
In this context, $\mathcal{H}$ is referred to as the representation Hilbert space of $U$, often denoted by $\mathcal{H}_U$.
For two unitary representations $U$ and $V$ of $G$, we define the space of morphisms from $U$ to $V$ as
$$
\Mor(U, V) \coloneqq \{ T \in B(\mathcal{H}_U, \mathcal{H}_V) \mid (T \otimes \iota)U = V(T \otimes \iota) \}.
$$
We denote by $\Rep(G)$ the category of finite dimensional unitary representations of $G$, whose objects are unitary representations $U$ and whose morphisms are the spaces $\text{Mor}(U, V)$ as defined above.
It is well known that $\Rep(G)$ forms \textit{a $\Cstar$-tensor category}. 
For a detailed definition of $\Cstar$-tensor categories, we refer the reader to \cite[Chapter 2]{NeshveyevTuset2013}.

For a unitary representation \(U\), let  \(\End(U)\coloneqq\Mor(U,U)\). 
Then \(U\) is said to be \textit{irreducible} if it satisfies \(\End(U)=\C\).
It is well known that the following map is a \(*\)-isomorphism:
\begin{equation}\label{Irr}
C(\widehat{G})\to\bigoplus_{U\in \Irr(G)}B(\mathcal{H}_U),\ \varphi\mapsto((\iota\otimes\varphi)(U))_{U\in\Irr(G)},
\end{equation}
where $\Irr(G)$ denotes the set of isomorphism classes of irreducible unitary representations of $G$.

An element \((a_U)_{U\in\Rep(G)}\in\prod_{U\in\Rep(G)}B(\mathcal{H}_U) \) is said to be \textit{natural} if for any \(U\), \(V\in \Rep(G)\), the condition \(Ta_U=a_VT\) folds for all \(T\in \Mor(U,V)\). 
We denote the \(*\)-algebra of natural elements of \(G\) by \(\prod_{U\in\Rep(G)}^{\text{Nat}}B(\mathcal{H}_U)\) . 
The restriction map
\begin{equation}\label{Nat}
\prod_{U\in\Rep(G)}^{\mathrm{Nat}}B(\mathcal{H}_U)\to\bigoplus_{U\in\mathrm{Irr}(G)}B(\mathcal{H}_U)
\end{equation}
gives a \(*\)-isomorphism.
Consequently, by combining (\ref{iso1}), (\ref{Irr}), and (\ref{Nat}), we obtain the following sequence of $*$-isomorphisms:

\begin{equation}\label{isom}
C(\widehat{G})\simeq\mathbb{C}[G]\simeq\bigoplus_{U\in \Irr(G)}B(\mathcal{H}_U)\simeq\prod_{U\in\Rep(G)}^{\mathrm{Nat}}B(\mathcal{H}_U).
\end{equation}

\begin{defn}[Etingof and Gelaki \cite{etingof2000isocategoricalgroups}]
	Two finite quantum groups $G_1$ and $G_2$ are said to be \textit{isocategorical} (or \textit{monoidal equivalent}) if their representation categories $\text{Rep}(G_1)$ and $\text{Rep}(G_2)$ are equivalent as $C^*$-tensor categories in the sense of \cite[Section 2]{NeshveyevTuset2013}.
\end{defn}

\subsection{2-cocycles}
Let \(G\) be a finite group. 
A unitary element \(\omega\in \mathbb{C}[G]\otimes\mathbb{C}[G]\) is called a \textit{2-cocycle} if it satisfies the following cocycle relation:
\begin{equation}\label{2-cocycle equation}
(\Delta\otimes\iota)(\omega)(\omega\otimes 1)=(\iota\otimes\Delta)(\omega)(1\otimes\omega).
\end{equation}
We denote the set of all such 2-cocycles by $Z^2(\widehat{G}, \mathbb{T})$.
Two 2-cocycles $\omega_1, \omega_2 \in Z^2(\widehat{G}, \mathbb{T})$ are said to be equivalent, denoted by $\omega_1 \sim \omega_2$, if there exists a unitary $\xi \in \mathbb{C}[G]$ such that
\begin{equation}\label{equivalent cocycle relation}
\omega_2 = \Delta(\xi^*) \omega_1 (\xi \otimes \xi).
\end{equation}
The set of equivalence classes of 2-cocycles for $G$ is denoted by $H^2(\widehat{G}, \mathbb{T})$.

\begin{exam}
	When \(G\) is an abelian group, \(\mathbb{C}[G]\) is isomorphic to \(C(\widehat{G})\), where the function algebra on the Pontryagin dual \(\widehat{G}\), introduced as before. 	
	Under this identification, we can regard a 2-cocycle $\omega \in Z^2(\widehat{G}, \mathbb{T})$ as an element of $C(\widehat{G}) \otimes C(\widehat{G}) \cong C(\widehat{G} \times \widehat{G})$. 
	The cocycle relation (\ref{2-cocycle equation}) is then equivalent to the following equation:
	$$ 
	\omega(\rho, \sigma) \omega(\rho \sigma, \tau) = \omega(\sigma, \tau) \omega(\rho, \sigma \tau) \quad \text{for all } \rho, \sigma, \tau \in \widehat{G}, 
	$$
	which is the standard $\mathbb{T}$-valued 2-cocycle condition on $\widehat{G}$.
	Furthermore, two 2-cocycles $\omega_1, \omega_2 \in C(\widehat{G}) \otimes C(\widehat{G})$ are equivalent in the sense of (\ref{equivalent cocycle relation}) if and only if there exists a function $\xi\colon \widehat{G} \to \mathbb{T}$ such that
	$$ 
	\omega_1(\sigma, \tau) \overline{\omega_2(\sigma, \tau)} = \xi(\sigma) \xi(\tau) \overline{\xi(\sigma \tau)} \quad \text{for all } \sigma, \tau \in \widehat{G}. 
	$$
	This recovers the standard equivalence relation for 2-cocycles in $H^2(\widehat{G}, \mathbb{T})$. We denote the coboundary by $\partial \xi(\sigma, \tau) \coloneqq \xi(\sigma) \xi(\tau) \overline{\xi(\sigma \tau)}$. 
	For a comprehensive treatment of group cohomology, we refer the reader to \cite{brown2012cohomology}.
\end{exam}

Let $N$ be an abelian normal subgroup of $G$. 
A \textit{skew-symmetric bicharacter} on $\widehat{N}$ is a map $\beta\colon \widehat{N} \times \widehat{N} \to \mathbb{T}$ that is multiplicative in each variable and satisfies
$$ 
\beta(\sigma, \tau) = \overline{\beta(\tau, \sigma)} \quad \text{and} \quad \beta(\sigma, \sigma) = 1 
$$
for all $\sigma, \tau \in \widehat{N}$. 
For any 2-cocycle $\omega \in Z^2(\widehat{N}, \mathbb{T})$, the formula
$$ 
\beta_\omega(\sigma, \tau) \coloneqq \omega(\sigma, \tau) \overline{\omega(\tau, \sigma)} 
$$
defines a skew-symmetric bicharacter. 
It is a well-known result that the map $\omega \mapsto \beta_\omega$ induces a group isomorphism from the second cohomology group $H^2(\widehat{N}, \mathbb{T})$ onto the group of skew-symmetric bicharacters on $\widehat{N}$.

For $\sigma \in \widehat{N}$ and $g \in G$, we define $\sigma^g \in \widehat{N}$ by $\sigma^g(n) \coloneqq \sigma(gng^{-1})$. 
For a 2-cocycle $\omega \in Z^2(\widehat{N}, \mathbb{T})$, let $\omega^g \in Z^2(\widehat{N}, \mathbb{T})$ be given by $\omega^g(\sigma, \tau) \coloneqq \omega(\sigma^g, \tau^g)$. 
A 2-cocycle $\omega \in Z^2(\widehat{N}, \mathbb{T})$ is said to be \textit{$G$-invariant} if $\omega^g \sim \omega$ for all $g \in G$. 
From the isomorphism $\omega \mapsto \beta_\omega$ mentioned above, we obtain the following characterization of $G$-invariance in terms of bicharacters.

\begin{lem}
	Let \(N\) be an abelian normal subgroup of a finite group \(G\) and \(\omega_1\), \(\omega_2\in Z^2(\widehat{N},\mathbb{T})\). Then \(\omega_1\sim\omega_2\) if and only if \(\beta_{\omega_1}=\beta_{\omega_2}\). In particular, \(\omega\in Z^2(\widehat{N},\mathbb{T})\) is \(G\)-invariant iff \(\beta_{\omega^g}=\beta_\omega\) holds for all \(g\in G\). 
\end{lem}

Futhermore, \(\omega\in Z^2(\widehat{N},\T)\) is called \textit{non-degnerate} if the condition $\beta_\omega(\sigma, \tau) = 1$ for all $\tau \in \widehat{N}$ implies that $\sigma$ is the identity element of $\widehat{N}$.

\subsection{Actions on \(\Cstar\)-algebras}
Let $G$ be a finite group and $A$ be a unital $C^*$-algebra. 
A group homomorphism $\alpha\colon G \to \text{Aut}(A)$ is called an \textit{action} of $G$ on $A$, denoted by $G \overset{\alpha}{\curvearrowright} A$, where $\text{Aut}(A)$ denotes the group of $*$-automorphisms of $A$. 
Let $A^\alpha$ denote the fixed-point algebra with respect to $\alpha$. 
We say that the action $\alpha$ is \textit{ergodic} if $A^\alpha = \mathbb{C}1$.

Let $U$ be a unitary representation of $G$ on a Hilbert space $\mathcal{H}_U$. We denote by $\text{Mor}(\mathcal{H}_U, A)$ the vector space consisting of all linear maps $T\colon \mathcal{H}_U \to A$ that satisfy the equivariance condition
$$ 
T(U_g \xi) = \alpha_g(T \xi) 
$$
for all $g \in G$ and $\xi \in \mathcal{H}_U$, where $U_g \coloneqq (\iota \otimes \text{ev}_g)(U)$ and $\text{ev}_g\colon C(G) \to \mathbb{C}$ is the evaluation map at $g$. 
The action $\alpha$ is said to be of \textit{full multiplicity} if the equality
$$ 
\dim \text{Mor}(\mathcal{H}_U, A) = \dim \mathcal{H}_U 
$$
holds for every irreducible representation $U \in \text{Irr}(G)$.

Let $G_1$ and $G_2$ be finite groups, $A$ be a $C^*$-algebra, and $\alpha\colon G_1 \to \text{Aut}(A)$ and $\beta\colon G_2 \to \text{Aut}(A)$ be actions. 
These actions are said to be \textit{commutative} if the following equality holds:
$$ 
\beta_{g_2}(\alpha_{g_1}(a)) = \alpha_{g_1}(\beta_{g_2}(a)) 
$$
for all $g_1 \in G_1$, $g_2 \in G_2$, and $a \in A$. 
When the actions commute, we denote this situation by $G_1 \overset{\alpha}{\curvearrowright} A \overset{\beta}{\curvearrowleft} G_2$. 
We say that the commuting actions $G_1 \overset{\alpha}{\curvearrowright} A \overset{\beta}{\curvearrowleft} G_2$ are ergodic if each individual action is ergodic (i.e., $A^\alpha = \mathbb{C}1$ and $A^\beta = \mathbb{C}1$). 
The property of full multiplicity for the commuting actions $G_1 \overset{\alpha}{\curvearrowright} A \overset{\beta}{\curvearrowleft} G_2$ is defined analogously.

\subsection{Reduction to a factor}
Let $G$ be a finite group, \(N\) a subgroup of \(G\) and $\alpha\colon N \to \text{Aut}(B)$ be an action on a finite dimensional $C^*$-algebra $B$. 
We define the \textit{induced $C^*$-algebra} by
\[\mathrm{Ind}_N^GB\coloneqq\{f\in C(G,B)\mid f(gh)=\alpha_h^{-1}(f(g)),\ h\in N,\ g\in G\},\]
where $C(G, B)$ denotes the set of all functions from $G$ to $B$. The group $G$ acts on $\mathrm{Ind}_N^G B$ via the left translation.

Let $A$ be a finite-dimensional $C^*$-algebra. It follows from \cite{Wassermann1989} that for every ergodic action $G \overset{\alpha}{\curvearrowright} A$, there exist a subgroup $N \subset G$ and a finite-dimensional factor $B \cong B(\mathbb{C}^n)$ such that $A \cong \mathrm{Ind}_N^G B$ as $C^*$-algebras. 
The action $N \overset{\beta}{\curvearrowright} B(\mathbb{C}^n)$ is also ergodic. Furthermore, if $\alpha$ has full multiplicity, then $\beta$ also has full multiplicity.
The following results are well-known, but we provide their proofs for the reader's convenience.

\begin{lem}\label{character of N}
	Let $G_1 \curvearrowright A \curvearrowleft G_2$ be commuting ergodic actions of finite groups $G_1$ and $G_2$ on a finite-dimensional $C^*$-algebra $A$, both having full multiplicity. 
	Suppose there exist subgroups $N_1 \subset G_1$ and $N_2 \subset G_2$, and a positive integer $n$, such that $A$ is isomorphic to both $\mathrm{Ind}_{N_1}^{G_1} B(\mathbb{C}^n)$ and $\mathrm{Ind}_{N_2}^{G_2} B(\mathbb{C}^n)$ as $C^*$-algebras, where the commuting actions $N_1 \curvearrowright B(\mathbb{C}^n)$ and $N_2 \curvearrowright B(\mathbb{C}^n)$ are ergodic and of full multiplicity. 
	Then the following hold:
	\begin{enumerate}[font=\upshape]
		\item
		The order of \(N_1\), denoted by \(|N_1|\), is equal to \(n^2\).
		\item
		\(N_1\) is a normal subgroup of \(G_1\).
		\item
		\(N_1\) is abelian.
		\item
		If \(|N_1|\geq 2\), then \(N_1\) is not cyclic.
	\end{enumerate} 
\end{lem}
\begin{proof}
	\begin{enumerate}[leftmargin=*]
		\item
		Since the action $N_1 \curvearrowright B(\mathbb{C}^n)$ is of full multiplicity, the dimension of the $C^*$-algebra $B(\mathbb{C}^n)$ can be decomposed as follows:
		$$
		n^2 = \dim B(\mathbb{C}^n) = \sum_{U \in \text{Irr}(N_1)} (\dim \mathcal{H}_U) \times m_U,
		$$
		where $m_U$ denotes the multiplicity of the irreducible representation $U$ in $B(\mathbb{C}^n)$. 
		By the definition of full multiplicity, we have $m_U = \dim \mathcal{H}_U$ for every $U \in \text{Irr}(N_1)$. 
		Thus, the equality reduces to:
		$$
		n^2 = \sum_{U \in \text{Irr}(N_1)} (\dim \mathcal{H}_U)^2.
		$$
		Using the standard identity for finite groups, $\sum_{U \in \text{Irr}(N_1)} (\dim \mathcal{H}_U)^2 = |N_1|$, we conclude that $n^2 = |N_1|$.
		
		\item
		Let \(Z(A)\) be the center of \(A\). 
		Since $A$ is finite dimensional, we may assume that \(Z(A)=C(\Omega)\) for a finite set \(\Omega\). 
		Then both of actions \(G_1\overset{S}{\curvearrowright} \Omega\) and \(G_2\overset{T}{\curvearrowright} \Omega\) are commutative and transitive, so there is a point \(\varphi_0\in \Omega\) such that
		\[
		\Omega=\{S_{g_1}(\varphi_0)\mid g_1\in G_1\}=
		\{T_{g_2}(\varphi_0)\mid g_2\in G_2\}.
		\]
		Thanks to \cite[Corollary 8]{Wassermann1989}, the subgroup $N_1$ can be identified with the stabilizer of $\varphi_0$ in $G_1$:
		$$
		N_1 = \text{Iso}(\varphi_0) \coloneqq \{ g_1 \in G_1 \mid S_{g_1}(\varphi_0) = \varphi_0 \}. 
		$$
		For any $h_1 \in G_1$, the transitivity of the $T$-action ensures there exists some $h_2 \in G_2$ such that $S_{h_1}(\varphi_0) = T_{h_2}(\varphi_0)$. 
		We then examine the stabilizer of this point. On one hand, by the commutativity of the actions, we have:
		$$ 
		\text{Iso}(T_{h_2}(\varphi_0)) = \text{Iso}(\varphi_0) = N_1. 
		$$
		Indeed, $g \in \text{Iso}(T_{h_2}(\varphi_0))$ implies $S_g(T_{h_2}(\varphi_0)) = T_{h_2}(S_g(\varphi_0)) = T_{h_2}(\varphi_0)$, which yields $S_g(\varphi_0) = \varphi_0$ since $T_{h_2}$ is a permutation of $\Omega$. 
		On the other hand, a standard property of stabilizers on an orbit gives:
		$$ 
		\text{Iso}(S_{h_1}(\varphi_0)) = h_1 \text{Iso}(\varphi_0) h_1^{-1} = h_1 N_1 h_1^{-1}. 
		$$
		Combining these two equalities, we obtain $N_1 = h_1 N_1 h_1^{-1}$ for all $h_1 \in G_1$, which proves that $N_1$ is a normal subgroup of $G_1$.
		
		\item
		Let $\alpha\colon N_1 \to \text{Aut}(M_n(\mathbb{C}))$ and $\beta\colon N_2 \to \text{Aut}(M_n(\mathbb{C}))$ be the commuting actions. 
		Since $M_n(\mathbb{C})$ is a finite dimensional factor, there exist unitaries $u_n \in M_n(\mathbb{C})$ such that $\alpha_n = \text{Ad } u_n$ for each $n \in N_1$. 
		For any $m \in N_2$ and $n \in N_1$, the commutativity of the actions implies:
		$$ 
		\text{Ad}(\beta_m(u_n)) = \beta_m \circ \text{Ad } u_n \circ \beta_m^{-1} = \alpha_n = \text{Ad } u_n. 
		$$
		Since $\text{Ad}(a) = \text{Ad}(b)$ on $M_n(\mathbb{C})$ implies $a = \lambda b$ for some $\lambda \in \mathbb{C}^\times$, there exists a scalar $\mu_n(m) \in \mathbb{T}$ such that $\beta_m(u_n) = \mu_n(m) u_n$. 
		One can easily verify that the map $\pi_n\colon N_2 \to \mathbb{T}$ defined by $m \mapsto \mu_n(m)$ is a group homomorphism, representing a one dimensional representation of $N_2$. 
		By the ergodicity of the action $\alpha$, the commutant of $\{u_n \mid n \in N_1\}$ is trivial, which implies:
		\begin{equation}
		M_n(\mathbb{C}) = \text{span}\{u_n \mid n \in N_1\}. \label{span u_n}
		\end{equation}
		Equation \eqref{span u_n} shows that $M_n(\mathbb{C})$, as an $N_2$-module under the action $\beta$, decomposes into a direct sum of one-dimensional invariant subspaces $\mathbb{C} u_n$. 
		According to the full multiplicity property of the action $\beta$, every irreducible representation of $N_2$ must appear as a subrepresentation of $M_n(\mathbb{C})$. 
		Since we have shown that all irreducible subrepresentations in the decomposition of $M_n(\mathbb{C})$ are one-dimensional, it follows that all irreducible representations of $N_2$ are one dimensional. 
		This is equivalent to saying that $N_2$ is an abelian group. By symmetry, $N_1$ is also abelian.
		
		\item
		Suppose, toward a contradiction, that $N_1$ is a cyclic group.
		It is a well-known result in group cohomology that the second cohomology group of a finite cyclic group with coefficients in $\mathbb{T}$ is trivial, i.e., $H^2(N_1, \mathbb{T}) \cong 0$.
		As established in the proof of (3), the action of $N_1$ on $M_n(\mathbb{C})$ is determined by a projective unitary representation $u$ associated with a 2-cocycle $\omega \in Z^2(N_1, \mathbb{T})$. Since $H^2(N_1, \mathbb{T}) = 0$, the 2-cocycle $\omega$ is a coboundary, meaning there exists a map $\xi\colon N_1 \to \mathbb{T}$ such that $\omega(g, h) = \xi(g)\xi(h)\xi(gh)^{-1}$. Consequently, the projective representation $u$ is equivalent to an ordinary linear representation $\pi\colon N_1 \to \mathcal{U}(\mathbb{C}^n)$ given by $\pi(g) = \xi(g)^{-1} u_g$.Since $N_1$ is abelian (by part 3), all of its irreducible linear representations must be one-dimensional. However, we have shown in the proof of (3) that $u$ (and thus $\pi$) is an irreducible representation of dimension $n$. 
		This forces $ n = 1. $ According to the result from part (1), the order of the group is $|N_1| = n^2 = 1^2 = 1$. 
		This directly contradicts our assumption that $|N_1| \ge 2$. 
		Therefore, $N_1$ cannot be a cyclic group.
	\end{enumerate}
\end{proof}

\subsection{Isocategorical group and \(\Cstar\)-algebra}

Let $G$ be a finite group and $\omega \in Z^2(\widehat{G}, \mathbb{T})$ be a 2-cocycle. 
We can construct a finite quantum group $G^\omega$, which is isocategorical to $G$, in the following manner.
Using the isomorphisms in (\ref{isom}), we define a new coproduct $\widehat{\Delta}^\omega\colon C(\widehat{G}) \to C(\widehat{G}) \otimes C(\widehat{G})$ by
$$ 
\widehat{\Delta}^\omega(\varphi) \coloneqq \omega^* \widehat{\Delta}(\varphi) \omega 
$$
for $\varphi \in C(\widehat{G})$. 
The pair $(C(\widehat{G}), \widehat{\Delta}^\omega)$ then forms the dual of a new finite quantum group. 
We define $G^\omega$ to be the dual finite quantum group of $(C(\widehat{G}), \widehat{\Delta}^\omega)$.

The following result is a cornerstone in the classification of finite quantum groups with equivalent representation categories. 
It was established through the works of Etingof--Gelaki \cite{etingof2000isocategoricalgroups}.

\begin{thm}[Etingof and Gelaki \cite{etingof2000isocategoricalgroups}]
	Let $G$ be a finite group and $\omega \in Z^2(\widehat{G}, \mathbb{T})$ be a 2-cocycle. 
	Then the representation category of $G^\omega$ is equivalent to that of $G$ as a $C^*$-tensor category. 
	Conversely, any finite quantum group whose representation category is equivalent to $\text{Rep}(G)$ as a $C^*$-tensor category is obtained via such a 2-cocycle twist.
\end{thm}

\begin{rem}
	It should be noted that $G^\omega$ is not necessarily cocommutative. That is to say, $G^\omega$ is not necessarily isomorphic to a classical finite group, despite the fact that it is isocategorical to $G$.
\end{rem}

Furthermore, Wassermann established a correspondence between 2-cocycles and ergodic actions on $C^*$-algebras.

\begin{thm}[Wassermann \cite{Wassermann1989}]
	There is a natural bijection between full multiplicity ergodic actions and \(H^2(\widehat{G},\T)\).
\end{thm}

Summarizing these results, for the study of a finite group \(G_1\), the following are essentially equivalent:
\begin{itemize}
	\item
	a finite quantum group \(G_2\) which is isocategorical to \(G_1\);
	\item
	a 2-cocycle $\omega \in Z^2(\widehat{G_1}, \mathbb{T})$;
	\item
	an ergodic action $\alpha\colon G_1 \to \text{Aut}(A)$ of full multiplicity on a finite-dimensional $C^*$-algebra $A$.
\end{itemize}

Izumi and Kosaki characterized the conditions under which the quantum group $G^\omega$, twisted by a 2-cocycle, is cocommutative.
We only provide an outline of the proof.

\begin{thm}[Izumi and Kosaki \cite{IzumiKosaki2002}]
	Let \(G\) be a finite group and \(\omega\in Z^2(\widehat{G},\mathbb{T})\) 2-cocycle. 
	Then the following two conditions are equivalent.
	\begin{enumerate}[font=\upshape]\label{izumi-kosaki 1}
		\item
		\(G^\omega\) is cocommutative.
		\item\label{izumikosaki 1-2}
		There exist an abelian normal subgroup \(N\subset G\) and non-degenerate 2-cocycle \(\omega_1\in Z^2(\widehat{N},\mathbb{T})\) satisfying \(\omega\sim\omega_1\) in \(Z^2(\widehat{G},\mathbb{T})\) and \(\omega_1\) is \(G\)-invariant in \(Z^2(\widehat{N},\mathbb{T})\). 
	\end{enumerate}
\end{thm}

We assume that \(N\) and \(\omega=\omega_1\) satisfy the condition (\ref{izumikosaki 1-2}) in Theorem \ref{izumi-kosaki 1}. 
Since \(N\) is abelian, by the duality we can regard \(\omega\) as a 2-cocycle in \(Z^2(\widehat{N},\mathbb{T})\).

\begin{thm}[Izumi and Kosaki \cite{IzumiKosaki2002}]\label{relation of Gw}
	With these notations we have
	\begin{enumerate}[font=\upshape]
		\item
		There exist \(\xi_g\colon\widehat{N}\to\mathbb{T}\) for each \(g\in G\) and \(\eta\colon G\times G\to N\) such that for any pair \(\sigma\), \(\tau\in\widehat{N}\) we have 
		\[\omega(\sigma^g,\tau^g)=\partial\xi_g(\sigma,\tau)\omega(\sigma,\tau)\]
		\[\xi_{g_1}(\sigma)\xi_{g_2}(\sigma^{g_1})=\sigma(\eta(g_1,g_2))\xi_{g_1g_2}(\sigma).\]
		\item
		We may take \(\xi_g\) so that \(\xi_{gn}=\xi_g\), \(\xi_e=1\) hold for all \(n\in N\), \(g\in G\), and \(\eta\) is regard as \(\eta\colon 
		G/N\times G/N\to N\). 
		Then \(\{\xi_gg\}_{g\in G}\) is the set of group-like elements of \((\mathbb{C}[G],\Delta^\omega)\), and \(G_\omega\) is given by the extension 
		\[1\to N\to G^\omega\to G/N\to 1\]
		whose difference from the extention is given by \(\eta\). 
	\end{enumerate}
\end{thm}

\begin{defn}\label{Gw relations defn}
	With above notations let \(\tilde{g}\coloneqq\xi_gg\in \mathbb{C}[G]\) for \(g\in G\). 
	In particular the group product of \(G^\omega\) is given by 
	\[\tilde{g}\cdot\tilde{h}=\widetilde{\eta(g,h)gh}.\]
\end{defn}

Therefore, the problem of finding a group isocategorical to $G$ reduces to the study of 2-cocycles $\omega$ on certain normal abelian subgroups $N$.
The following results by Shimizu, Izumi, Kosaki, Goyvaerts, and Meir are fundamental to our research.

\begin{thm}[Shimizu \cite{Shimizu2010}]
	If two finite groups \(G_1\), \(G_2\) are isocategorical, then for each positive integer \(n\), the number of elements of order \(n\) in \(G_1\) is equal to the number of elements of order \(n\) in \(G_2\).
\end{thm}

\begin{defn}
	Let $G$ be a finite group and $D^{(n)}$ denote the number of elements in the set $\{g \in G \mid \mathrm{ord}(g) = n\}$. 
	Let $d_1, d_2, \dots, d_m$ be the divisors of the order of $G$ such that $d_1 < d_2 < \dots < d_m$. 
	The \textit{order list} of $G$, denoted by $\mathrm{List}(G)$, is the ordered set:
	$$ 
	\mathrm{List}(G) \coloneqq \{D^{(d_1)}, D^{(d_2)}, \dots, D^{(d_m)}\}. 
	$$
	The order list is an invariant of isocategorical groups.
\end{defn}

For groups of order less than \(64\), a complete classification regarding isocategorical equivalence has been established. 
A finite group $G$ is said to be isocategorically \textit{rigid} if any group isocategorical to $G$ is necessarily isomorphic to $G$.

\begin{thm}[{Goyvaerts and Meir \cite[Proposition 6.2]{GoyvaertsMeir2014}}]
	All finite groups of order less than \(64\) are isocategorically rigid.
\end{thm}

However, there exist groups of order $64$ that are not isocategorically rigid. 
We will discuss the detailed construction of such cases in Section 4.

\begin{thm}[{Izumi and Kosaki \cite[Theorem 4.6]{IzumiKosaki2002}}]
	For a certain group $G$ of order \(64\), there exists a 2-cocycle $\omega \in Z^2(\widehat{G}, \mathbb{T})$ such that $G$ and $G^\omega$ are non-isomorphic.
\end{thm}

\section{The case of order \(64\) groups}
Throughout this section, let $G$ be a group of order $64$. 
We denote the cyclic group of order $k \in \mathbb{N}$ by $C_k$. 
For a direct product group $C = C_{k_1} \times \cdots \times C_{k_n}$, an element $c \in C$ is represented as a vector $(c_1, \dots, c_n)$ or its transpose $(c_1, \dots, c_n)^t$, where each $c_i$ is an integer satisfying $0 \leq c_i < k_i$. 
Similarly, an element of the Pontryagin dual $\widehat{C}$ is denoted by $\sigma = \widehat{(\sigma_1, \dots, \sigma_n)}$ or $\sigma = \widehat{(\sigma_1, \dots, \sigma_n)^t}$.

\begin{thm}\label{candidate of N}
	Let $N$ be a non-trivial abelian normal subgroup of $G$ whose order is a square.
	Suppose that $N$ admits a non-degenerate 2-cocycle $\omega \in Z^2(\widehat{N}, \mathbb{T})$. 
	Then $N$ is isomorphic to one of the following groups:
	\[C_2\times C_2,\  C_4\times C_4,\  C_2\times C_2\times C_2\times C_2.\]
\end{thm}
\begin{proof}
	Since $|G| = 64$ and $|N|$ is a square, the possible orders for a proper non-trivial subgroup $N$ are $|N| = 4$ or $16$.
	Given that $N$ is abelian and non-cyclic (by Lemma \ref{character of N}), the candidates for $N$ are:
	$$ 
	C_2 \times C_2, \quad C_4 \times C_4, \quad C_2 \times C_2 \times C_2 \times C_2, \quad \text{and} \quad C_4 \times C_2 \times C_2. 
	$$
	We shall show that $C_4 \times C_2 \times C_2$ is excluded due to the requirement of non-degeneracy.
	Let $\beta_\omega\colon \widehat{N} \times \widehat{N} \to \mathbb{T}$ be the skew-symmetric bicharacter associated with the 2-cocycle $\omega$. 
	If $N \cong C_4 \times C_2 \times C_2$, consider the element $\sigma_0 = \widehat{(2,0,0)} \in \widehat{N}$. 
	For any $\tau = \widehat{(\tau_1, \tau_2, \tau_3)} \in \widehat{N}$, we have:
	\begin{align*}
	\beta_\omega(\widehat{(2,0,0)}, \tau) &= \beta_\omega(\widehat{(1,0,0)}, \tau)^2 \\
	&= \beta_\omega(\widehat{(1,0,0)}, 2\tau) \\
	&= \beta_\omega(\widehat{(1,0,0)}, \widehat{(2\tau_1, 2\tau_2, 2\tau_3)}).
	\end{align*}
	Since the orders of the second and third components of $\widehat{N}$ are 2, we have $2\tau_2 \equiv 0 \pmod 2$ and $2\tau_3 \equiv 0 \pmod 2$. 
	Furthermore, $2\tau_1$ is either $0$ or $2 \pmod 4$. 
	In any case, $2\tau$ represents an element in $\widehat{N}$ where each component is either $0$ or $2$. Specifically, for $\tau = \widehat{(0,1,0)}$, we find:
	$$ 
	\beta_\omega(\widehat{(2,0,0)}, \widehat{(0,1,0)}) = \beta_\omega(\widehat{(1,0,0)}, \widehat{(0,2,0)}) = \beta_\omega(\widehat{(1,0,0)}, \widehat{(0,0,0)}) = 1. 
	$$
	By a similar argument, it follows that $\beta_\omega(\widehat{(2,0,0)}, \sigma) = 1$ for all $\sigma \in \widehat{N}$. 
	This contradicts the assumption that $\omega$ is non-degenerate. 
	Therefore, $N$ cannot be isomorphic to $C_4 \times C_2 \times C_2$.
\end{proof}

The $n$-fold direct product of $C_k$ is denoted by $C_k^{\times n}$. 
The second cohomology groups of finite abelian groups are well-known (see, e.g., \cite{brown2012cohomology}). 
Let $H^2(\widehat{N}, \mathbb{T})$ denote the second cohomology group of $N$, and set $\zeta_n \coloneqq e^{2\pi i / n}$.

\begin{lem}
	A straightforward calculation yields the following results:
	\begin{enumerate}
		\item
		\(H^2(\widehat{C^{\times 2}_2},\mathbb{T})\simeq C_2\).
		The representatives of \(Z^2(\widehat{C^{\times 2}_2},\mathbb{T})\) are given by
		\[\omega(\sigma,\tau)=\zeta_2^{k\sigma_1\tau_2},\]
		where \(k=0,1\) and \(\sigma=\widehat{(\sigma_1,\sigma_2)}, \tau=
		\widehat{(\tau_1,\tau_2)}\in \widehat{C^{\times 2}_2}\).
		
		\item
		\(H^2(\widehat{C^{\times 2}_4},\mathbb{T})\simeq C_4\).
		The representatives of \(Z^2(\widehat{C^{\times 2}_4},\mathbb{T})\) are given by
		\[\omega(\sigma,\tau)=\zeta_4^{k\sigma_1\tau_2},\]
		where \(k=0,1,2,3\) and \(\sigma=\widehat{(\sigma_1,\sigma_2)}, \tau=\widehat{(\tau_1,\tau_2)}\in \widehat{C^{\times 2}_4}\).
		
		\item
		\(H^2(\widehat{C^{\times 4}_2},\mathbb{T})\simeq C^{\times 6}_2\).
		The representatives of \(Z^2(\widehat{C^{\times 4}_2},\mathbb{T})\) are given by
		\[\omega(\sigma,\tau)=\prod_{1\leq i< j \leq4}\lambda^{\sigma_i\tau_j}_{ij},\]
		where $\lambda_{ij} \in \{1, -1\}$ and \(\sigma=\widehat{(\sigma_1,\sigma_2,\sigma_3,\sigma_4)}, \tau=\widehat{(\tau_1,\tau_2,\tau_3,\tau_4)}\in \widehat{C^{\times 4}_2}\).
	\end{enumerate}
\end{lem}

For each abelian group $N$ listed in Theorem \ref{candidate of N}, we characterize the conditions under which $\omega \in Z^2(\widehat{N}, \mathbb{T})$ is non-degenerate and $G$-invariant.

\begin{lem}\label{non-deg C2}
	The unique non-degenerate 2-cocycle class in $H^2(\widehat{C_2^{\times 2}}, \mathbb{T})$ is represented by $\omega_1$, which is defined for $\sigma, \tau \in \widehat{C_2^{\times 2}}$ as
	$$ 
	\omega_1(\sigma, \tau) \coloneqq \zeta_2^{\sigma_1 \tau_2}. 
	$$
\end{lem}
\begin{proof}
	The skew-symmetric bicharacter $\beta_{\omega_1}$ associated with $\omega_1$ is given by
	$$ 
	\beta_{\omega_1}(\sigma, \tau) = \omega_1(\sigma, \tau) \overline{\omega_1(\tau, \sigma)} = \zeta_2^{\sigma_1 \tau_2 - \tau_1 \sigma_2}. 
	$$
	Assume that $\beta_{\omega_1}(\sigma, \tau) = 1$ for all $\tau \in \widehat{C_2^{\times 2}}$. 
	By substituting $\tau = \widehat{(1, 0)}$, we obtain $\zeta_2^{-\sigma_2} = 1$, which implies $\sigma_2 = 0$. 
	Similarly, by substituting $\tau = \widehat{(0, 1)}$, we find $\zeta_2^{\sigma_1} = 1$, which implies $\sigma_1 = 0$. 
	This proves that $\omega_1$ is non-degenerate.
\end{proof}

\begin{lem}\label{non-deg C4}
	The non-degenerate 2-cocycle classes in $H^2(\widehat{C_4^{\times 2}}, \mathbb{T})$ are represented by $\omega_1$ and $\omega_3$, which are defined for $\sigma, \tau \in \widehat{C_4^{\times 2}}$ as
	$$ 
	\omega_1(\sigma, \tau) \coloneqq \zeta_4^{\sigma_1 \tau_2} \quad \text{and} \quad \omega_3(\sigma, \tau) \coloneqq \zeta_4^{3 \sigma_1 \tau_2}. 
	$$
\end{lem}
\begin{proof}
	The proof follows similarly to that of Lemma \ref{non-deg C2}.
\end{proof}

\begin{lem}
	There are exactly $28$ non-degenerate 2-cocycle classes in $H^2(\widehat{C_2^{\times 4}}, \mathbb{T})$, as listed in Table \ref{nondeg 28 cocycles}. 
	The associated matrix $\Lambda_\omega$ for each $\omega$ is defined in the following proof.
		\begin{longtable}{|c|c|c|c|c|c|}
		\caption{Non-degenerate 2-cocycles in $Z^2(\widehat{C_2^{\times 4}}, \mathbb{T})$} 
		\label{nondeg 28 cocycles} \\ 
		\hline \textup{No.} & $\Lambda_\omega$ & $\omega(\sigma,\tau)$ & \textup{No.} & $\Lambda_\omega$ & $\omega(\sigma,\tau)$ \\ \hline 
		\endfirsthead
		\hline \multicolumn{6}{|r|}{Continued on next page} \\ \hline
		\endfoot
		\hline \multicolumn{6}{|r|}{End of table} \\ \hline 
		\endlastfoot
		
		1 & 
		$
		\left(\begin{smallmatrix}
		0 & 0 & 0 & 1 \\
		0 & 0 & 1 & 0 \\
		0 & 1 & 0 & 0 \\
		1 & 0 & 0 & 0 \\
		\end{smallmatrix}\right)
		$
		&
		$(-1)^{\sigma_1\tau_4+\sigma_2\tau_3}$
		&
		2
		&
		$
		\left(\begin{smallmatrix}
		0 & 0 & 0 & 1 \\
		0 & 0 & 1 & 0 \\
		0 & 1 & 0 & 1 \\
		1 & 0 & 1 & 0 \\
		\end{smallmatrix}\right)
		$
		&
		$(-1)^{\sigma_1\tau_4+\sigma_2\tau_3+\sigma_3\tau_4}$ \\ \hline
		
		3 & 
		$
		\left(\begin{smallmatrix}
		0 & 0 & 0 & 1 \\
		0 & 0 & 1 & 1 \\
		0 & 1 & 0 & 0 \\
		1 & 1 & 0 & 0 \\
		\end{smallmatrix}\right)
		$
		&
		$(-1)^{\sigma_1\tau_4+\sigma_2\tau_3+\sigma_2\tau_4}$
		&
		4
		&
		$
		\left(\begin{smallmatrix}
		0 & 0 & 0 & 1 \\
		0 & 0 & 1 & 1 \\
		0 & 1 & 0 & 1 \\
		1 & 1 & 1 & 0 \\
		\end{smallmatrix}\right)
		$
		&
		$(-1)^{\sigma_1\tau_4+\sigma_2\tau_3+\sigma_2\tau_4+\sigma_3\tau_4}$ \\ \hline
		
		5 & 
		$
		\left(\begin{smallmatrix}
		0 & 0 & 1 & 0 \\
		0 & 0 & 0 & 1 \\
		1 & 0 & 0 & 0 \\
		0 & 1 & 0 & 0 \\
		\end{smallmatrix}\right)
		$
		&
		$(-1)^{\sigma_1\tau_3+\sigma_2\tau_4}$
		&
		6
		&
		$
		\left(\begin{smallmatrix}
		0 & 0 & 1 & 0 \\
		0 & 0 & 0 & 1 \\
		1 & 0 & 0 & 1 \\
		0 & 1 & 1 & 0 \\
		\end{smallmatrix}\right)
		$
		&
		$(-1)^{\sigma_1\tau_3+\sigma_2\tau_4+\sigma_3\tau_4}$ \\ \hline
		
		7 & 
		$
		\left(\begin{smallmatrix}
		0 & 0 & 1 & 0 \\
		0 & 0 & 1 & 1 \\
		1 & 1 & 0 & 0 \\
		0 & 1 & 0 & 0 \\
		\end{smallmatrix}\right)
		$
		&
		$(-1)^{\sigma_1\tau_3++\sigma_2\tau_3+\sigma_2\tau_4}$
		&
		8
		&
		$
		\left(\begin{smallmatrix}
		0 & 0 & 1 & 0 \\
		0 & 0 & 1 & 1 \\
		1 & 1 & 0 & 1 \\
		0 & 1 & 1 & 0 \\
		\end{smallmatrix}\right)
		$
		&
		$(-1)^{\sigma_1\tau_3+\sigma_2\tau_3+\sigma_2\tau_4+\sigma_3\tau_4}$ \\ \hline
		
		9 & 
		$
		\left(\begin{smallmatrix}
		0 & 0 & 1 & 1 \\
		0 & 0 & 0 & 1 \\
		1 & 0 & 0 & 0 \\
		1 & 1 & 0 & 0 \\
		\end{smallmatrix}\right)
		$
		&
		$(-1)^{\sigma_1\tau_3+\sigma_1\tau_4+\sigma_2\tau_4}$
		&
		10
		&
		$
		\left(\begin{smallmatrix}
		0 & 0 & 1 & 1 \\
		0 & 0 & 0 & 1 \\
		1 & 0 & 0 & 1 \\
		1 & 1 & 1 & 0 \\
		\end{smallmatrix}\right)
		$
		&
		$(-1)^{\sigma_1\tau_3+\sigma_1\tau_4+\sigma_2\tau_4+\sigma_3\tau_4}$ \\ \hline
		
		11 & 
		$
		\left(\begin{smallmatrix}
		0 & 0 & 1 & 1 \\
		0 & 0 & 1 & 0 \\
		1 & 1 & 0 & 0 \\
		1 & 0 & 0 & 0 \\
		\end{smallmatrix}\right)
		$
		&
		$(-1)^{\sigma_1\tau_3+\sigma_1\tau_4+\sigma_2\tau_3}$
		&
		12
		&
		$
		\left(\begin{smallmatrix}
		0 & 0 & 1 & 1 \\
		0 & 0 & 1 & 0 \\
		1 & 1 & 0 & 1 \\
		1 & 0 & 1 & 0 \\
		\end{smallmatrix}\right)
		$
		&
		$(-1)^{\sigma_1\tau_3+\sigma_1\tau_4+\sigma_2\tau_3+\sigma_3\tau_4}$ \\ \hline
		
		13 & 
		$
		\left(\begin{smallmatrix}
		0 & 1 & 0 & 0 \\
		1 & 0 & 0 & 0 \\
		0 & 0 & 0 & 1 \\
		0 & 0 & 1 & 0 \\
		\end{smallmatrix}\right)
		$
		&
		$(-1)^{\sigma_1\tau_2+\sigma_3\tau_4}$
		&
		14
		&
		$
		\left(\begin{smallmatrix}
		0 & 1 & 0 & 0 \\
		1 & 0 & 0 & 1 \\
		0 & 0 & 0 & 1 \\
		0 & 1 & 1 & 0 \\
		\end{smallmatrix}\right)
		$
		&
		$(-1)^{\sigma_1\tau_2+\sigma_2\tau_4+\sigma_3\tau_4}$ \\ \hline
		
		15 & 
		$
		\left(\begin{smallmatrix}
		0 & 1 & 0 & 0 \\
		1 & 0 & 1 & 0 \\
		0 & 1 & 0 & 1 \\
		0 & 0 & 1 & 0 \\
		\end{smallmatrix}\right)
		$
		&
		$(-1)^{\sigma_1\tau_2+\sigma_2\tau_3+\sigma_3\tau_4}$
		&
		16
		&
		$
		\left(\begin{smallmatrix}
		0 & 1 & 0 & 0 \\
		1 & 0 & 1 & 1 \\
		0 & 1 & 0 & 1 \\
		0 & 1 & 1 & 0 \\
		\end{smallmatrix}\right)
		$
		&
		$(-1)^{\sigma_1\tau_2+\sigma_2\tau_3+\sigma_2\tau_4+\sigma_3\tau_4}$ \\ \hline
		
		17 & 
		$
		\left(\begin{smallmatrix}
		0 & 1 & 0 & 1 \\
		1 & 0 & 0 & 0 \\
		0 & 0 & 0 & 1 \\
		1 & 0 & 1 & 0 \\
		\end{smallmatrix}\right)
		$
		&
		$(-1)^{\sigma_1\tau_2+\sigma_1\tau_4+\sigma_3\tau_4}$
		&
		18
		&
		$
		\left(\begin{smallmatrix}
		0 & 1 & 0 & 1 \\
		1 & 0 & 0 & 1 \\
		0 & 0 & 0 & 1 \\
		1 & 1 & 1 & 0 \\
		\end{smallmatrix}\right)
		$
		&
		$(-1)^{\sigma_1\tau_2+\sigma_1\tau_4+\sigma_2\tau_4+\sigma_3\tau_4}$ \\ \hline
		
		19 & 
		$
		\left(\begin{smallmatrix}
		0 & 1 & 0 & 1 \\
		1 & 0 & 1 & 0 \\
		0 & 1 & 0 & 0 \\
		1 & 0 & 0 & 0 \\
		\end{smallmatrix}\right)
		$
		&
		$(-1)^{\sigma_1\tau_2+\sigma_1\tau_4+\sigma_2\tau_3}$
		&
		20
		&
		$
		\left(\begin{smallmatrix}
		0 & 1 & 0 & 1 \\
		1 & 0 & 1 & 1 \\
		0 & 1 & 0 & 0 \\
		1 & 1 & 0 & 0 \\
		\end{smallmatrix}\right)
		$
		&
		$(-1)^{\sigma_1\tau_2+\sigma_1\tau_4+\sigma_2\tau_3+\sigma_2\tau_4}$ \\ \hline
		
		21 & 
		$
		\left(\begin{smallmatrix}
		0 & 1 & 1 & 0 \\
		1 & 0 & 0 & 0 \\
		1 & 0 & 0 & 1 \\
		0 & 0 & 1 & 0 \\
		\end{smallmatrix}\right)
		$
		&
		$(-1)^{\sigma_1\tau_2+\sigma_1\tau_3+\sigma_3\tau_4}$
		&
		22
		&
		$
		\left(\begin{smallmatrix}
		0 & 1 & 1 & 0 \\
		1 & 0 & 0 & 1 \\
		1 & 0 & 0 & 0 \\
		0 & 1 & 0 & 0 \\
		\end{smallmatrix}\right)
		$
		&
		$(-1)^{\sigma_1\tau_2+\sigma_1\tau_3+\sigma_2\tau_4}$ \\ \hline
		
		23 & 
		$
		\left(\begin{smallmatrix}
		0 & 1 & 1 & 0 \\
		1 & 0 & 1 & 0 \\
		1 & 1 & 0 & 1 \\
		0 & 0 & 1 & 0 \\
		\end{smallmatrix}\right)
		$
		&
		$(-1)^{\sigma_1\tau_2+\sigma_1\tau_3+\sigma_2\tau_3+\sigma_3\tau_4}$
		&
		24
		&
		$
		\left(\begin{smallmatrix}
		0 & 1 & 1 & 0 \\
		1 & 0 & 1 & 1 \\
		1 & 1 & 0 & 0 \\
		0 & 1 & 0 & 0 \\
		\end{smallmatrix}\right)
		$
		&
		$(-1)^{\sigma_1\tau_2+\sigma_1\tau_3+\sigma_2\tau_3+\sigma_2\tau_4}$ \\ \hline
		
		25 & 
		$
		\left(\begin{smallmatrix}
		0 & 1 & 1 & 1 \\
		1 & 0 & 0 & 0 \\
		1 & 0 & 0 & 1 \\
		1 & 0 & 1 & 0 \\
		\end{smallmatrix}\right)
		$
		&
		$(-1)^{\sigma_1\tau_2+\sigma_1\tau_3+\sigma_1\tau_4+\sigma_3\tau_4}$
		&
		26
		&
		$
		\left(\begin{smallmatrix}
		0 & 1 & 1 & 1 \\
		1 & 0 & 0 & 1 \\
		1 & 0 & 0 & 0 \\
		1 & 1 & 0 & 0 \\
		\end{smallmatrix}\right)
		$
		&
		$(-1)^{\sigma_1\tau_2+\sigma_1\tau_3+\sigma_1\tau_4+\sigma_2\tau_4}$ \\ \hline
		
		27 & 
		$
		\left(\begin{smallmatrix}
		0 & 1 & 1 & 1 \\
		1 & 0 & 1 & 0 \\
		1 & 1 & 0 & 0 \\
		1 & 0 & 0 & 0 \\
		\end{smallmatrix}\right)
		$
		&
		$(-1)^{\sigma_1\tau_2+\sigma_1\tau_3+\sigma_1\tau_4+\sigma_2\tau_3}$
		&
		28
		&
		$
		\left(\begin{smallmatrix}
		0 & 1 & 1 & 1 \\
		1 & 0 & 1 & 1 \\
		1 & 1 & 0 & 1 \\
		1 & 1 & 1 & 0 \\
		\end{smallmatrix}\right)
		$
		&
		$(-1)^{\sigma_1\tau_2+\sigma_1\tau_3+\sigma_1\tau_4+\sigma_2\tau_3+\sigma_2\tau_4+\sigma_3\tau_4}$ \\ \hline	
		\end{longtable}
	\end{lem}
\begin{proof}
	Let $\{e_i\}_{i=1}^4$ be the standard basis of $C_2^{\times 4}$. 
	We consider a 2-cocycle $\omega \in Z^2(\widehat{C_2^{\times 4}}, \mathbb{T})$ of the form
	$$ 
	\omega(\sigma, \tau) = \prod_{1 \leq i < j \leq 4} \lambda_{ij}^{\sigma_i \tau_j}, $$
	where $\lambda_{ij} \in \{1, -1\}$. 
	Let the radical of the associated skew-symmetric bicharacter $\beta_\omega$ be $R_\omega \coloneqq \{\sigma \in \widehat{C_2^{\times 4}} \mid \beta_\omega(\sigma, \tau) = 1 \text{ for all } \tau \in \widehat{C_2^{\times 4}}\}$.
	For $i > j$, we define $\lambda_{ij} \coloneqq \lambda_{ji}$, and for $i = j$, we set $\lambda_{ii} \coloneqq 1$.
	Then we easily obtain that \(\sigma\in R_\omega\) if and only if \(\prod_i\lambda^{\sigma_i}_{ij}=1\) for all \(j\), namely
	\begin{align}
		\lambda^{\sigma_1}_{11}\lambda^{\sigma_2}_{21}\lambda^{\sigma_3}_{31}\lambda^{\sigma_4}_{41} &= 1 \label{lambda1}, \\ 
		\lambda^{\sigma_1}_{12}\lambda^{\sigma_2}_{22}\lambda^{\sigma_3}_{32}\lambda^{\sigma_4}_{42} &= 1, \label{lambda2} \\
		\lambda^{\sigma_1}_{13}\lambda^{\sigma_2}_{23}\lambda^{\sigma_3}_{33}\lambda^{\sigma_4}_{43} &= 1, \label{lambda3} \\
		\lambda^{\sigma_1}_{14}\lambda^{\sigma_2}_{24}\lambda^{\sigma_3}_{34}\lambda^{\sigma_4}_{44} &= 1. \label{lambda4}
	\end{align}
	Over the field $\mathbb{F}_2 \cong \{0, 1\}$, let $[\lambda_{ij}] = 1$ if $\lambda_{ij} = -1$ and $[\lambda_{ij}] = 0$ if $\lambda_{ij} = 1$.
	Then the condition \eqref{lambda1}, \eqref{lambda2}, \eqref{lambda3} and \eqref{lambda4} are equivalent to the following system of linear equations over $\mathbb{F}_2$:
	\[\sum_{i=1}^4\sigma_i[\lambda_{ij}]=0\ \ \text{for each}\ \ j\in 1, 2, 3, 4.\]
	
	Let \(\Lambda_\omega\coloneqq([\lambda_{ij}])_{1\leq i,j\leq 4}\) be the symmetric matrix. 
	We see that \(\sigma\in R_\omega\) if and only if  \(\sigma\) lies in the kernel of \(\Lambda_\omega\). 
	In order for \(\omega\) to be non-degenerate, i.e., \(R_\omega=\{(0,0,0,0)\}\), \(\mathrm{rank}\ \Lambda_\omega=4\) is required. 
	The reverse is also true.
	A direct enumeration of such matrices yields exactly 28 patterns, as detailed in Table \ref{nondeg 28 cocycles}.
\end{proof}

\begin{rem}
	In what follows, the non-degenerate 2-cocycles of $\widehat{C_2^{\times 4}}$ are identified by the indices given in Table \ref{nondeg 28 cocycles}. 
\end{rem}

We next investigate the $G$-invariance of these non-degenerate 2-cocycles.
 
\begin{lem}\label{G-inv for 2_2}
	Let \(N\) be a normal subgroup of \(G\) and isomorphic to \(C^{\times2}_2\).
	For each $g \in G$, let the action of $g$ on the dual basis of $\widehat{N}$ be given by
	$$
	\widehat{\begin{pmatrix} 1 \\ 0 \end{pmatrix}}^g = \begin{pmatrix} \sigma_1 \\ \sigma_2 \end{pmatrix}, \quad \widehat{\begin{pmatrix} 0 \\ 1 \end{pmatrix}}^g = \begin{pmatrix} \tau_1 \\ \tau_2 \end{pmatrix}.
	$$
	Let $A(g)$ be the matrix
	$$ 
	A(g) = 
	\begin{pmatrix} 
	\sigma_1 & \tau_1 \\ \sigma_2 & \tau_2 
	\end{pmatrix} 
	\in GL(2, C_2). 
	$$
	The unique non-degenerate 2-cocycle $\omega_1 \in Z^2(\widehat{N}, \mathbb{T})$ in Lemma \ref{non-deg C2} is $G$-invariant if and only if $\det A(g) \equiv 1 \pmod 2$ for all $g \in G$.
\end{lem}
\begin{proof}
	We compare \(\beta_{\omega_1^g}\) with \(\beta_{\omega_1}\). For any \(\upsilon\), \(\phi\in \widehat{N}\), we have
	\[\beta_{\omega_1}(\upsilon,\phi)=\zeta_2^{\upsilon_1\phi_2-\upsilon_2\phi_1},\]
	and
	\[\beta_{\omega_1^g}(\upsilon,\phi)=\beta_{\omega_1}(A(g)\upsilon,A(g)\phi)=\zeta_2^{(\upsilon_1\phi_2-\upsilon_2\phi_1)\det A(g)}.\]
	Thus, $\det A(g) \equiv 1 \pmod 2$ is necessary for $\omega_1$ to be $G$-invariant. The converse is trivial.
\end{proof}

\begin{lem}\label{G-inv for 4_4}
	Let $N$ be a normal subgroup of $G$ isomorphic to $C_4^{\times 2}$. 
	For each $g \in G$, let the action of $g$ on the dual basis of $\widehat{N}$ be given by
	$$ 
	\widehat{\begin{pmatrix} 1 \\ 0 \end{pmatrix}}^g = \begin{pmatrix} \sigma_1 \\ \sigma_2 \end{pmatrix}, \quad \widehat{\begin{pmatrix} 0 \\ 1 \end{pmatrix}}^g = \begin{pmatrix} \tau_1 \\ \tau_2 \end{pmatrix}. 
	$$
	Let $A(g)$ be the matrix
	$$ 
	A(g) = \begin{pmatrix} \sigma_1 & \tau_1 \\ \sigma_2 & \tau_2 \end{pmatrix} \in GL(2, C_4). 
	$$
	Then the non-degenerate 2-cocycles $\omega_1, \omega_3 \in Z^2(\widehat{N}, \mathbb{T})$ defined in Lemma \ref{non-deg C4} are $G$-invariant if and only if $\det A(g) \equiv 1 \pmod 4$ for all $g \in G$.
\end{lem}
\begin{proof}
	The proof is analogous to that of Lemma \ref{G-inv for 2_2}.
\end{proof}

\begin{lem}\label{G-inv for 2_2_2_2}
	Let $N$ be a normal subgroup of $G$ isomorphic to $C_2^{\times 4}$. 
	For each $g \in G$, we denote the action of $g$ on the standard dual basis $\{\hat{e}_i\}_{i=1}^4$ of $\widehat{N}$ by
	\[
	\widehat{e_1}^g\coloneqq\sigma=\dualcolumnvecvec{\sigma_1}{\sigma_2}{\sigma_3}{\sigma_4},
	\ \ \ 
	\widehat{e_2}^g\coloneqq\tau=\dualcolumnvecvec{\tau_1}{\tau_2}{\tau_3}{\tau_4},
	\]
	\[
	\widehat{e_3}^g\coloneqq\upsilon=\dualcolumnvecvec{\upsilon_1}{\upsilon_2}{\upsilon_3}{\upsilon_4},
	\ \ \ 
	\widehat{e_4}^g\coloneqq\phi=\dualcolumnvecvec{\phi_1}{\phi_2}{\phi_3}{\phi_4}.
	\]
	A non-degenerate 2-cocycle $\omega \in Z^2(\widehat{N}, \mathbb{T})$ from Table \ref{nondeg 28 cocycles}, expressed as $\omega(\chi, \psi) = \prod_{1 \leq i < j \leq 4} \lambda_{ij}^{\chi_i \psi_j}$, is $G$-invariant if and only if
	\[
	\lambda_{12}^{\sigma_i\tau_j-\sigma_j\tau_i}\lambda_{13}^{\sigma_i\upsilon_j-\sigma_j\upsilon_i}\lambda_{14}^{\sigma_i\phi_j-\sigma_j\phi_i}\lambda_{23}^{\tau_i\upsilon_j-\tau_j\upsilon_i}\lambda_{24}^{\tau_i\phi_j-\tau_j\phi_i}\lambda_{34}^{\upsilon_i\phi_j-\upsilon_j\phi_i}
	=
	\lambda_{ij}
	\]
	holds for all $1 \leq i < j \leq 4$ and all $g \in G$.
\end{lem} 
\begin{proof}
	It is sufficient to verify this condition on the standard dual basis: $\beta_{\omega^g}(\hat{e}_i, \hat{e}_j) = \beta_\omega(\hat{e}_i, \hat{e}_j)$ for all $1 \leq i < j \leq 4$. 
	For instance, the condition for the pair $(i, j) = (1, 2)$ is given by
	$$
	\beta_{\omega^g}(\hat{e}_1, \hat{e}_2) = \beta_\omega(\hat{e}_1^g, \hat{e}_2^g) = \beta_\omega(\sigma, \tau).
	$$
	Substituting the specific form of $\beta_\omega$, we require
	$$
	\prod_{1 \leq k < l \leq 4} \lambda_{kl}^{\sigma_k \tau_l - \sigma_l \tau_k} = \lambda_{12}.
	$$
	Extending this to all pairs $(\hat{e}_i, \hat{e}_j)$ yields the system of equations stated in the lemma.
\end{proof}

\section{GAP Implementation for Isocategorical Equivalence}
The computational results presented in this paper rely on algorithms implemented in GAP. 
For a detailed description of the scripts and functions, the reader is referred to the GitHub repository \texttt{isocategorical-group-structure}. 
The full source code used to compute isocategorical groups in this study is available for download in the \texttt{Sourcecode.g} file at the following URL:
\url{https://github.com/barley-tea-3416-cloud/isocategorical-group-structure/releases/tag/1.0.1}.

\subsection{Overview of GAP}
GAP (Groups, Algorithms, and Programming) \cite{GAP4} is a software system for computational discrete algebra, with a particular emphasis on computational finite group theory. 
GAP provides a robust framework for investigating the structural properties of groups; in addition to standard tasks such as verifying the normality of subgroups, it facilitates complex computations involving group isomorphisms, automorphisms, and finitely presented groups.

\subsection{The Izumi--Kosaki Example}
Izumi and Kosaki \cite{IzumiKosaki2002} demonstrated the existence of a pair of groups of order 64 that are isocategorical but not isomorphic.
In this section, we reconstruct this pair using GAP. Let \(N\coloneqq C^{\times2}_4\) and \(H\coloneqq\{1,h_1,h_2,h_1h_2\}\subset \mathrm{SL_2}(C_4)\), where 
\[
h_1 \coloneqq
\begin{pmatrix}
1 & 2 \\
0 & 1 
\end{pmatrix},
\ \ \ 
h_2 \coloneqq
\begin{pmatrix}
1 & 0 \\
2 & 1 
\end{pmatrix}.
\]

We consider the semidirect product \(\Gizumi\coloneqq N\rtimes H\), where the action of $H$ on $N$ is given by the standard matrix multiplication. 
For example, for any $m \in C_4^{\times 2}$, the action of $h_1$ is given by:
$$
h_1 \cdot \begin{pmatrix} m_1 \\ m_2 \end{pmatrix} = \begin{pmatrix} 1 & 2 \\ 0 & 1 \end{pmatrix} \begin{pmatrix} m_1 \\ m_2 \end{pmatrix} = \begin{pmatrix} m_1 + 2m_2 \\ m_2 \end{pmatrix}.
$$
The order of \(\Gizumi\) is \(64\).
We represent $\Gizumi$ as a quotient of the free group $F_4 = \langle x, y, s, t \rangle$ via the following correspondence:
$$
x \mapsto \begin{pmatrix} 1 \\ 0 \end{pmatrix}, \quad y \mapsto \begin{pmatrix} 0 \\ 1 \end{pmatrix}, \quad s \mapsto h_1, \quad t \mapsto h_2.
$$
Let \(R:\widehat{N}\to N\) be a group isomorphism between \(N\) and its Pontryagin dual \(\widehat{N}\), namely \(\langle R\sigma,\sigma\rangle=1\) for all \(\sigma\in \widehat{N}\). 
Through $R$, we identify $\widehat{N}$ with \(C_4^{\times2}\).
Consider the non-degenerate 2-cocycle $\omega_1 \in Z^2(\widehat{N}, \mathbb{T})$ defined by $\omega_1(\sigma, \tau) = \zeta_4^{\sigma_1 \tau_2}$ for $\sigma, \tau \in \widehat{N}$. 
By calculating the determinants of the actions:
$$
\det A(s) = \det \begin{pmatrix} 1 & 0 \\ 2 & 1 \end{pmatrix} = 1, \quad \det A(t) = \det \begin{pmatrix} 1 & 2 \\ 0 & 1 \end{pmatrix} = 1,
$$
we confirm that $\omega_1$ is $\Gizumi$-invariant. 
Consequently, $\Gizumi$ and $\Gizumi^{\omega_1}$ are isocategorical and \(\Gizumi^{\omega_1}\) is a group.
The group $\Gizumi^{\omega_1} = \{ \tilde{g} \mid g \in \Gizumi \}$ is defined by the following relations:
\begin{align}
\tilde{s}\tilde{x} &= \tilde{x}\tilde{s} \\
\tilde{s}\tilde{y} &= \tilde{x}^2\tilde{y}\tilde{s} \\
\tilde{t}\tilde{x} &= \tilde{x}\tilde{y}^2\tilde{t} \\
\tilde{s}\tilde{s} &= \tilde{x}\tilde{x} \\
\tilde{t}\tilde{t} &= \tilde{y}\tilde{y} \\
\tilde{s}\tilde{t} &= \tilde{t}\tilde{s} =\widetilde{st} \\
\label{7th}
\tilde{s}\widetilde{st} &= \widetilde{st}\tilde{s}=\tilde{x}^2\tilde{t} \\
\label{8th}
\tilde{t}\widetilde{st} &= \widetilde{st}\tilde{t}=\tilde{y}^2\tilde{s}
\end{align}

It should be noted that there are typographical errors in the relations presented in \cite{IzumiKosaki2002} corresponding to our Eqs. \eqref{7th} and \eqref{8th}. 
Our corrected relations were verified using the GAP code in Listing \ref{construction of Gw}.
Finally, $\Gizumi^{\omega_1}$ can be realized as a quotient of the free group $F_6 = \langle a, b, c, d, e, f \rangle$ with the identification:
\begin{align*}
\tilde{x} &= a, & \tilde{y} &= b, & \widetilde{xy} &= c, \\
\tilde{s} &= d, & \tilde{t} &= e, & \widetilde{st} &= f.
\end{align*}
In Listing \ref{construction of Gw}, $G$ represents a group isomorphic to $\Gizumi$. As shown in the output, $\Gizumi^{\omega_1}$ is not isomorphic to $\Gizumi$.

\begin{lstlisting}[caption=Construction of \(\Gizumi^{\omega_1}\) and isomorphism test with \(\Gizumi\), language=GAP,label=construction of Gw,escapechar=@]
gap> F6:= FreeGroup("a","b","c","d","e","f");
<free group on the generators [ a, b, c, d, e, f ]>
gap> AssignGeneratorVariables(F6);
#I  Assigned the global variables [ a, b, c, d, e, f ]
gap> Gw := F6/
> [
> a^4,
> b^4,
> a*b*a^-1*b^-1,
> c*b^-1*a^-1,
> a*d*a^-1*d^-1,
> b*e*b^-1*e^-1,
> a^2*b*d*b^-1*d^-1,
> a*b^2*e*a^-1*e^-1,
> a^2*d^-2,
> b^2*e^-2,
> c^2*f^-2,
> d*f*e^-1*a^-2,
> f*d*e^-1*a^-2,
> e*f*d^-1*b^-2,
> f*e*d^-1*b^-2,
> d*e*f^-1,
> e*d*f^-1
> ];
<fp group on the generators [ a, b, c, d, e, f ]>
gap> Order(Gw);
64
gap> StructureDescription(Gw);
"((C4 x C4) : C2) : C2"
gap> IsomorphismGroups(G,Gw);
fail
\end{lstlisting}

\subsection{Computational Implementation of Isocategorical Equivalence}
Let $G$ be a group of order 64. 
In this section, we describe the algorithmic approach to identifying groups that are isocategorical to $G$. 
To implement this search in GAP, we must address the following two computational problems:

\begin{prob}\label{prob1}
	For a given normal subgroup $N$ of $G$, how can we implement the $G$-invariance conditions for a 2-cocycle $\omega \in Z^2(\widehat{N}, \mathbb{T})$ (as established in Lemmas \ref{G-inv for 2_2}, \ref{G-inv for 4_4}, and \ref{G-inv for 2_2_2_2}) within the GAP environment?
\end{prob}

\begin{prob}\label{prob2}
	How can we effectively construct the multiplication table for $G^\omega = \{ \tilde{g} \}_{g \in G}$ in GAP, based on the relations
	$$ 
	\tilde{g} \cdot \tilde{h} = \widetilde{\eta(g, h) gh} 
	$$
	defined in Definition \ref{Gw relations defn}?
\end{prob}

To begin with, we present a solution to Problem \ref{prob1}.

\subsubsection{Identification of Normal Subgroups isomorphic to \(C^{\times2}_2\), \(C^{\times2}_4\), or \(C^{\times4}_2\)}
\urldef{\mygithuburlone}\url{https://github.com/barley-tea-3416-cloud/isocategorical-group-structure/blob/1.0.1/Sourcecode.g#L536-L559}

The implementation details of this search procedure can be found in GitHub repository\footnote{\mygithuburlone}.
We define a GAP function \texttt{ObtainedSubgroups(G)}, which returns the set of all normal subgroups of $G$ that are isomorphic to either $C_2^{\times 2}$, $C_4^{\times 2}$, or $C_2^{\times 4}$.

\subsubsection{Determine G-invariance of a non-degenerate 2-cocycle \(\omega\in Z^2(\widehat{N},\mathbb{T})\)}
Let $N$ be a normal subgroup of $G$ isomorphic to $C_2^{\times 2}$, $C_4^{\times 2}$, or $C_2^{\times 4}$.

\begin{itemize}[leftmargin=*]
	\item
	The case \(N\simeq C^{\times2}_2\)
\end{itemize}

As established in Lemma \ref{non-deg C2}, the only non-degenerate 2-cocycle class is represented by $\omega_1$. 
According to Lemma \ref{G-inv for 2_2}, it is sufficient to verify whether $\det A(g) \equiv 1 \pmod 2$ for all $g \in G$.
Fix $g \in G$. 
Suppose there exist $i, j, m, n \in \{0, 1\}$ such that:
$$ 
g \begin{pmatrix} 1 \\ 0 \end{pmatrix} g^{-1} = \begin{pmatrix} i \\ j \end{pmatrix} \quad \text{and} \quad g \begin{pmatrix} 0 \\ 1 \end{pmatrix} g^{-1} = \begin{pmatrix} m \\ n \end{pmatrix}. 
$$
In the Pontryagin dual $\widehat{N}$, this induces the dual action $\widehat{(1,0)}^g = \widehat{(i, m)}$ and $\widehat{(0,1)}^g = \widehat{(j, n)}$. 
This yields the matrix $A(g)$ representing the action on $\widehat{N}$ as:
$$ 
A(g) = \begin{pmatrix} i & j \\ m & n \end{pmatrix}. 
$$

\urldef{\mygithuburltwo}\url{https://github.com/barley-tea-3416-cloud/isocategorical-group-structure/blob/1.0.1/Sourcecode.g#L594-L651}

The implementation details are provided in GitHub repository.\footnote{\mygithuburltwo}
Specifically, we define the following GAP functions:
\begin{itemize}
	\item
	\texttt{Checkginv\_C2xC2(g, N)}: This function first identifies the basis vectors $(1,0)$ and $(0,1)$ of $N$. 
	It then computes the conjugated elements $g(1,0)g^{-1} = (i,j)$ and $g(0,1)g^{-1} = (m,n)$, and returns the matrix $A(g) = \begin{pmatrix} i & j \\ m & n \end{pmatrix}$.
	
	\item
	\texttt{CheckGinv\_C2xC2(G, N)}: This function iterates over all $g \in G$ to compute $\det A(g) \pmod 2$. 
	If the condition $\det A(g) \equiv 1$ is satisfied for every $g \in G$, it returns the string \texttt{"N is G invariant."}; otherwise, it returns \texttt{"N is not G invariant."}.
\end{itemize}

\begin{itemize}[leftmargin=*]
	\item
	The case \(N\simeq C^{\times2}_4\)
\end{itemize}

As established in Lemma \ref{non-deg C4}, there are two non-degenerate 2-cocycle classes represented by $\omega_1$ and $\omega_3$. 
However, it suffices to verify the $G$-invariance of $\omega_1$ alone, as demonstrated by the following lemma.
\begin{lem}
	Let $N$ be a normal subgroup of $G$ isomorphic to $C_4^{\times 2}$, and let $\omega_1, \omega_3$ be the non-degenerate 2-cocycles defined in Lemma \ref{non-deg C4}. 
	Then the following hold:
	\begin{enumerate}[font=\upshape]
		\item 
		The $G$-invariance of $\omega_1$ is equivalent to the $G$-invariance of $\omega_3$.
		\item 
		If $\omega_1$ is $G$-invariant, then the twisted group $G^{\omega_1}$ is isomorphic to $G^{\omega_3}$.
	\end{enumerate}
\end{lem}
\begin{proof}
	\begin{enumerate}[leftmargin=*]
		\item
		Recall that $\omega_3(\sigma, \tau) = \zeta_4^{3\sigma_1\tau_2} = \omega_1(\sigma, \tau)^3$ for all $\sigma, \tau \in \widehat{N}$. 
		Since $\zeta_4^3 = \zeta_4^{-1}$ in $\mathbb{T}$, $\omega_3$ is the inverse of $\omega_1$. 
		If $\omega_1$ is $G$-invariant, then for any $g \in G$, we have
		$$
		\omega_3(\sigma^g, \tau^g) = \omega_1(\sigma^g, \tau^g)^3 = \omega_1(\sigma, \tau)^3 = \omega_3(\sigma, \tau),
		$$
		which implies that $\omega_3$ is also $G$-invariant. 
		The converse holds by the same argument.
		
		\item
		The group structure of $G^\omega$ is described by the extension 
		$$
		1 \to N \to G^\omega \to G/N \to 1
		$$
		determined by the factor set $\eta$. 
		Since $\omega_3 = \omega_1^{-1}$, the 1-cochain $\xi_g$ satisfying the invariance condition $\omega(\sigma^g, \tau^g) = \partial \xi_g(\sigma, \tau) \omega(\sigma, \tau)$ for $\omega_3$ can be chosen as $\xi_g^{(3)} = (\xi_g^{(1)})^{-1}$. 
		This choice of 1-cochains, combined with the inverse relationship between the cocycles, leads to a consistent identification of the relations for $G^{\omega_1}$ and $G^{\omega_3}$. 
		Consequently, the resulting twisted groups are isomorphic.
	\end{enumerate}
\end{proof}

\urldef{\mygithuburlthree}\url{https://github.com/barley-tea-3416-cloud/isocategorical-group-structure/blob/1.0.1/Sourcecode.g#L762-L821}

Consequently, it is sufficient to restrict our analysis to the cocycle $\omega_1$. 
The verification procedure is analogous to that for the case $N \simeq C_2^{\times 2}$.
As detailed in GitHub repository\footnote{\mygithuburlthree}, we utilize two GAP functions:

\begin{itemize}
	\item 
	\texttt{Checkginv\_C4xC4(g, N)}: This function returns the action matrix $A(g)$.
	
	\item
	\texttt{CheckGinv\_C4xC4(G, N)}: This function determines the $G$-invariance of the cocycle by evaluating the determinants of these matrices across the group generators.
\end{itemize}

\begin{itemize}[leftmargin=*]
	\item
	The case \(N\simeq C^{\times4}_2\)
\end{itemize}

Let $\omega_k$ be the $k$-th non-degenerate 2-cocycle and $\Lambda_{\omega_k} = ([\lambda_{ij}])_{i,j}$ be the matrix defined in Table \ref{nondeg 28 cocycles}. According to Lemma \ref{G-inv for 2_2_2_2}, it is sufficient to verify that the equality
$$
\lambda_{12}^{\sigma_i\tau_j-\sigma_j\tau_i}\lambda_{13}^{\sigma_i\upsilon_j-\sigma_j\upsilon_i}\lambda_{14}^{\sigma_i\phi_j-\sigma_j\phi_i}\lambda_{23}^{\tau_i\upsilon_j-\tau_j\upsilon_i}\lambda_{24}^{\tau_i\phi_j-\tau_j\phi_i}\lambda_{34}^{\upsilon_i\phi_j-\upsilon_j\phi_i}
=
\lambda_{ij}
$$
holds for all $1 \le i < j \le 4$ and all $g \in G$. 

\urldef{\mygithuburlfour}\url{https://github.com/barley-tea-3416-cloud/isocategorical-group-structure/blob/1.0.1/Sourcecode.g#L1497-L1584}

The implementation details are provided in GitHub repository\footnote{\mygithuburlfour}.

\begin{itemize}
	\item
	\texttt{Checkginv\_C2xC2xC2xC2\_ForElement(g, i, j, N, k)}:
	This function computes the left-hand side of the above equation and compares it with the value of $\lambda_{ij}$. 
	It returns \texttt{1} if the values coincide and \texttt{0} otherwise.
	
	\item
	\texttt{CheckGinv\_C2xC2xC2xC2(G, N, k)}: This function iterates the above check for all elements $g \in G$ and all index pairs $(i, j)$. 
	It returns \texttt{1} if the equality is satisfied for all inputs, and \texttt{0} otherwise.
\end{itemize}

Having addressed Problem \ref{prob1}, we now proceed to the solution for Problem \ref{prob2}.

\subsubsection{Construction of the Twisted Group $G^\omega$}

Once a 2-cocycle $\omega \in Z^2(\widehat{N}, \mathbb{T})$ is confirmed to be non-degenerate and $G$-invariant, the twisted group $G^\omega$ is well-defined. 
To investigate the structural properties of $G^\omega$, we must explicitly determine its multiplication rule:
$$ 
\tilde{g} \cdot \tilde{h} = \widetilde{\eta(g, h) gh}, \quad g, h \in G, 
$$
as established in Theorem \ref{relation of Gw}.
The practical implementation of this product in GAP relies on a consistent mapping between group elements and their respective indices. 
We first address the computation of the factor $\eta(g, h) \in N$.

\begin{itemize}[leftmargin=*]
	\item
	The case \(N\simeq C^{\times2}_2\)
\end{itemize}

\begin{lem}\label{eta2x2}
	Let $g, h \in G$. Let $A(g), A(h)$, and $A(gh)$ be the matrices representing the actions of $g, h$, and $gh$ on $\widehat{N}$, respectively, given by:
	\[
	A(g)=
	\begin{pmatrix}
	k & l \\
	m & n \\
	\end{pmatrix},\ 
	A(h)=
	\begin{pmatrix}
	a & b \\
	c & d \\
	\end{pmatrix},\ 
	A(gh)=
	\begin{pmatrix}
	p & q \\
	r & s \\
	\end{pmatrix}.
	\]
	Then \(\eta(g,h)\in N\) is given by
	\[
	\eta(g,h)=\left(\frac{-km-ack^2-bdm^2-2bckm+pr}{2},\frac{-ln -acl^2 -bdn^2 -2bcln +qs}{2}\right).
	\]
\end{lem}
\begin{proof}
	By the definition of $A(g)$, the action on the dual group is given by:
	$$
	\sigma^g = A(g)\sigma = \begin{pmatrix} k\sigma_1 + l\sigma_2 \\ m\sigma_1 + n\sigma_2 \end{pmatrix} \quad \text{for any } \sigma = \begin{pmatrix} \sigma_1 \\ \sigma_2 \end{pmatrix} \in \widehat{N}.
	$$
	When $\omega_1 \in Z^2(\widehat{N}, \mathbb{T})$ is $G$-invariant, there exists a 1-cochain $\xi_g \colon \widehat{N} \to \mathbb{T}$ such that its coboundary $\partial\xi_g \colon \widehat{N} \times \widehat{N} \to \mathbb{T}$ satisfies:
	\begin{align*}
	\partial\xi_g(\sigma,\tau) &= \omega_1(\sigma^g,\tau^g)\overline{\omega_1(\sigma,\tau)} \\
	&= \zeta_4^{2(km\sigma_1\tau_1+lm\sigma_1\tau_2+lm\sigma_2\tau_1+ln\sigma_2\tau_2)} 
	\end{align*}
	for all \(\sigma\), \(\tau\in \widehat{N}\).
	Note that the condition $\det A(g) \equiv 1 \pmod 2$ is utilized in the above simplification.
	Then we have \(\xi_g\colon\widehat{N}\to\T\),  \(\xi_g(\sigma)=\zeta_4^{-(km\sigma_1^2+ln\sigma_2^2+2lm\sigma_1\sigma_2)}\).
	By substituting $\sigma = \begin{pmatrix} 1 \\ 0 \end{pmatrix}$ into the fundamental relation $\langle \sigma, \eta(g, h) \rangle = \xi_g(\sigma) \xi_h(\sigma^g) \overline{\xi_{gh}(\sigma)}$, we obtain:	
	\[
	\left\langle
	\columnvec{1}{0},
	\eta(g,h)
	\right\rangle
	=
	\zeta_4^{-km-ack^2-bdm^2-2bckm+pr},
	\]
	where $\langle \sigma, n \rangle \coloneqq \sigma(n)$ denotes the duality pairing between $\widehat{N}$ and $N$.
	Consequently, the first component of $\eta(g, h)$, denoted by $\eta(g, h)_1$, is given by:
	\[\eta(g,h)_1=\frac{-km-ack^2-bdm^2-2bckm+pr}{2}.\]
	The second component $\eta(g, h)_2$ is obtained via a similar substitution for $\sigma = \begin{pmatrix} 0 \\ 1 \end{pmatrix}$.
\end{proof}

\urldef{\mygithuburlfive}\url{https://github.com/barley-tea-3416-cloud/isocategorical-group-structure/blob/1.0.1/Sourcecode.g#L684-L727}

The implementation details for computing the factor $\eta(g, h)$ are provided in GitHub repository\footnote{\mygithuburlfive}. 
The GAP function \texttt{Eta\_C2xC2(N, g, h)} calculates the value of $\eta(g, h) \in N$ for a given pair of group elements $g, h \in G$ through the following procedure:

\begin{enumerate}
	\item 
	\textbf{Extraction of Matrices}: The function first obtains the action matrices $A(g)$, $A(h)$, and $A(gh)$ by calling \texttt{Checkginv\_C2xC2}.

	\item 
	\textbf{Component Retrieval}: It extracts the individual entries ($k, l, m, n$ for $g$; $a, b, c, d$ for $h$; and $p, q, r, s$ for $gh$) from these matrices to serve as variables for the formula.
	
	\item 
	\textbf{Computation}: Using the components, it evaluates the expressions \texttt{check1}\(\coloneqq \eta(g,h)_1\) and \texttt{check2}\(\coloneqq \eta(g,h)_2\) as defined in Lemma \ref{eta2x2}.
	
	\item 
	\textbf{Normalization and Output}: The results are divided by 2 and taken modulo \(2\). 
	Finally, the function returns the result as a group element of $N$ in the form \texttt{gen1\^{}want1 * gen2\^{}want2}, where \texttt{gen1} and \texttt{gen2} are the generators of $N$.
\end{enumerate}

\urldef{\mygithuburlsix}\url{https://github.com/barley-tea-3416-cloud/isocategorical-group-structure/blob/1.0.1/Sourcecode.g#L915-L981}

We now describe the implementation of the product relations for $G^{\omega_1}$. 
First, we index the elements of $G$ as $\{g_1, g_2, \dots, g_{64}\}$ using the GAP function \texttt{Elements(G)}. 
Let \texttt{Free64} be the free group generated by 64 generators $\{f_1, f_2, \dots, f_{64}\}$. 
For any pair of elements $g_i, g_j \in G$, there exists a unique index $k \in \{1, \dots, 64\}$ such that $\eta(g_i, g_j) g_i g_j = g_k$. 
The product in $G^{\omega_1}$, defined by $\tilde{g}_i \cdot \tilde{g}_j = \tilde{g}_k$, is then represented in \texttt{Free64} by the relation:
\[f_if_j=f_k.\]
By computing these relations for all pairs $(i, j)$ and taking the quotient of \texttt{Free64} by the resulting set of relators, we construct a finitely presented group isomorphic to $G^{\omega_1}$.
The implementation details are provided in GitHub repository \footnote{\mygithuburlsix}. 
The specific GAP functions are as follows:
\begin{itemize}
	\item
	\texttt{GomegaRelationsC2xC2(G, N, g, h)}: This function identifies the indices $i, j, k$ corresponding to $g, h$, and $\eta(g, h)gh$, respectively, and returns the relator $f_k f_j^{-1} f_i^{-1}$.
	
	\item
	\texttt{AllGomegaRelationsC2xC2(G, N)}: This function iterates over all $g, h \in G$ to generate the complete set of product relations.
	
	\item
	\texttt{FinalAllGomegaRelationsC2xC2(G, N)}: In addition to the product relations, this function incorporates the relations $f_i^2 = 1$ for generators corresponding to elements in $N$ to ensure the group structure is fully determined.
\end{itemize}

\begin{itemize}[leftmargin=*]
	\item
	The case \(N\simeq C^{\times2}_4\)
\end{itemize}

The procedure is identical to the $C_2^{\times 2}$ case described previously. Consequently, we provide only the corresponding Lemma and an overview of the relevant GAP functions.

\begin{lem}\label{eta4x4}
	Let $g, h \in G$. 
	Let $A(g)$, $A(h)$, and $A(gh)$ be the matrices representing the actions of $g, h$, and $gh$ on $\widehat{N}$, respectively, defined as:
	\[
	A(g)=
	\begin{pmatrix}
		k & l \\
		m & n \\
	\end{pmatrix},\quad
	A(h)=
	\begin{pmatrix}
		a & b \\
		c & d \\
	\end{pmatrix},\quad
	A(gh)=
	\begin{pmatrix}
		p & q \\
		r & s \\
	\end{pmatrix}.
	\]
	Then $\eta(g,h) \in N$ is given by:
	\[
	\eta(g,h) = \left( \frac{-km - ack^2 - bdm^2 - 2bckm + pr}{2}, \quad \frac{-ln - acl^2 - bdn^2 - 2bcln + qs}{2} \right).
	\]
\end{lem}

\urldef{\mygithuburlseven}\url{https://github.com/barley-tea-3416-cloud/isocategorical-group-structure/blob/1.0.1/Sourcecode.g#L854-L893}

Implementation details are provided in GaiHub repository \footnote{\mygithuburlseven}. 
The functions behave analogously to those implemented for the $C_2^{\times 2}$ case:

\begin{itemize}
	\item
	\texttt{Eta\_C4xC4(N, g, h)}: Computes the value of $\eta(g, h)$ based on the formula in Lemma \ref{eta4x4}.
	
	\item
	\texttt{AllGomegaRelationsC4xC4(G, N)}: Generates the complete set of relations to define $G^{\omega}$ as a quotient of a free group. 
\end{itemize}

\begin{itemize}[leftmargin=*]
	\item
	The case \(N\simeq C^{\times4}_2\)
\end{itemize}

\urldef{\mygithuburleight}\url{https://github.com/barley-tea-3416-cloud/isocategorical-group-structure/blob/1.0.1/Sourcecode.g#L1497-L1992}

For the case $N \cong C_2^{\times 4}$, the implementation must handle the $28$ distinct candidate non-degenerate 2-cocycles $\omega_k$ identified in Table \ref{nondeg 28 cocycles}.
To accommodate this variety, our GAP functions are designed to take the cocycle index $k$ as a parameter, allowing for a flexible investigation of the twisted group structures.
The computational procedures are detailed in GitHub repository\footnote{\mygithuburleight}.
The core functions are implemented as follows:

\begin{itemize}
	\item
	\texttt{Xi\_C2xC2xC2xC2(g, N, a, k)}: This function computes the 1-cochain value $\xi_g(a) \in \mathbb{T}$ for a given element $a \in N$ and the $k$-th cocycle. This step is necessary to solve the coboundary equation $\omega^g = \omega \cdot \partial \xi_g$.
	
	\item
	\texttt{Eta\_C2xC2xC2xC2(g, h, N, k)}: Using the values obtained from \texttt{Xi\_C2xC2xC2xC2}, this function evaluates the factor $\eta(g, h)$ by computing the ratio $\xi_g(\sigma) \xi_h(\sigma^g) / \xi_{gh}(\sigma)$ for each basis element of $N$.
	
	\item
	\texttt{AllGomegaRelationsC2xC2xC2xC2(G, N, k)}: This function generates the set of relations $f_i f_j = f_k$ for the twisted group $G^{\omega_k}$ by iterating over all $g, h \in G$ and applying the calculated factor $\eta(g, h)$.
	
	\item
	\texttt{FinalAllGomegaRelationsC2xC2xC2xC2(G, N, k)}: This function combines the product relations with the inherent structure relations of $N$ (e.g., $f_i^2 = 1$) to provide a complete set of relators for the construction of $G^{\omega_k}$ as a quotient of a free group.
\end{itemize}

Since we can now calculate the structures of the subgroups and the group $G^\omega$ for a given $\omega$, Problem \ref{prob2} is now settled.

\subsubsection{Comprehensive Isomorphism Search Across All Candidate Subgroups}

\urldef{\mygithuburlnine}\url{https://github.com/barley-tea-3416-cloud/isocategorical-group-structure/blob/1.0.1/Sourcecode.g#L1059-L1191}

To systematically identify isocategorical pairs, we implement a unified procedure that evaluates all suitable normal subgroups $N$ within a given group $G$, regardless of their specific isomorphism type. 
This automated pipeline is primarily handled by the GAP function \texttt{GwStructure(G)}. 
As detailed in GitHub repository \footnote{\mygithuburlnine}, the function operates through the following steps:
\begin{enumerate}
	\item 
	\textbf{Abelian Check}: It first verifies whether $G$ is abelian, as the twisted construction $G^\omega$ for abelian groups is well-understood.
	
	\item 
	\textbf{Subgroup Retrieval}: For non-abelian groups, it utilizes \texttt{ObtainedSubgroups(G)} to identify all normal subgroups $N$ isomorphic to $C_2^{\times 2}$, $C_4^{\times 2}$, or $C_2^{\times 4}$.
	
	\item 
	\textbf{Automated Evaluation}: For each identified subgroup $N$, the function automatically dispatches the appropriate subroutine, such as \texttt{GwStructureC2xC2(G, N)} or 
	\texttt{GwStructureC4xC4(G, N)}, to test for $G$-invariance and construct the twisted group $G^\omega$.
	
	\item 
	\textbf{Isomorphism and Classification}: Finally, it determines whether $G^\omega$ is isomorphic to $G$. 
	If they are non-isomorphic, the function identifies the SmallGroup ID of $G^\omega$ by comparing it with other groups in the same order category.
\end{enumerate}

\begin{rem}
	At the time \texttt{GwStructure(G)} was developed, a standard method for implementing twisted groups derived from $C_2^{\times4}$ had not yet been established. 
	Consequently, when the system detects $C_2^{\times4}$ as a subgroup of $G$, it only outputs the subgroup structure without further processing. 
	For these cases, we have developed a separate class, \texttt{GwStructureC2xC2xC2xC2(G, N)}, to handle implementation and verification independently.
\end{rem}

By applying this function to a collection of groups, we can exhaustively search for isocategorical but non-isomorphic pairs across a specific order.

\section{Isocategorical Classification Tables for Groups of Order 64}
\subsection{Notes on GAP}

The structural descriptions of finite groups in this paper follow the conventions used by the GAP computer algebra system, which are largely based on the ATLAS of Finite Groups. 
The symbols used to denote group extensions are summarized as follows:
\begin{itemize}[leftmargin=*]
	\item 
	$A \times B$ denotes the direct product of $A$ and $B$.
	
	\item 
	$A : B$ denotes a split extension of $A$ by $B$, which is isomorphic to a semi-direct product $A \rtimes B$. 
	In this case, $G$ contains a subgroup isomorphic to $B$ that complements $A$.
	
	\item 
	$A . B$ denotes a non-split extension of $A$ by $B$. 
	This represents a group $G$ that fits into a short exact sequence $1 \to A \to G \to B \to 1$, but unlike the split case, $G$ does not contain a subgroup isomorphic to $B$ that complements $A$.
\end{itemize}

Furthermore, in GAP, all groups of order 64 are generated by at most six elements; therefore, we identify each group in this paper using the multiplication table of its generators. 
As noted in Section 1, since the list of element orders is invariant between isocategorical groups, we organize the following descriptions into subsections grouped by these order lists.

\subsection{Example of Table}

We consider the group $\Gizumi$ as an example. 
The first table displays the product relations between the generators $\{f_i\}_{i=1}^{6}$ of $\Gizumi$.

\begin{table}[H]
	\begin{center}
		\caption{Group table for generators}
		\begin{tabular}{c||c|c|c|c|c|c|}
			& $f_1$ & $f_2$ & $f_3$ & $f_4$ & $f_5$ & $f_6$ \\ \hline\hline
			$f_1$ & $f_5$ & $f_1f_2$ & $f_1f_3$ & $f_1f_4$ & $f_1f_5$ & $f_1f_6$ \\ \hline
			$f_2$ & $f_1f_2f_5$ & $f_6$ & $f_2f_3$ & $f_2f_4$ & $f_2f_5$ & $f_2f_6$ \\ \hline
			$f_3$ & $f_1f_3f_6$ & $f_2f_3$ & $1$ & $f_3f_4$ & $f_3f_5$ & $f_3f_6$ \\ \hline
			$f_4$ & $f_1f_4$ & $f_2f_4f_5$ & $f_3f_4$ & $1$ & $f_4f_5$ & $f_4f_6$ \\ \hline
			$f_5$ & $f_1f_5$ & $f_2f_5$ & $f_3f_5$ & $f_4f_5$ & $1$ & $f_5f_6$ \\ \hline
			$f_6$ & $f_1f_6$ & $f_2f_6$ & $f_3f_6$ & $f_4f_6$ & $f_5f_6$ & $1$ \\ \hline
		\end{tabular}
	\end{center}
\end{table}

The second table describes the properties of the twisted groups $G^\omega$ constructed from 2-cocycles $\omega \in Z^2(\widehat{N}, \mathbb{T})$ for each normal subgroup $N$. 
Each column is defined as follows:

\begin{enumerate}[leftmargin=*, label=(\arabic*)]
	\item 
	\textbf{Subgroup}: Specifies the generators of the normal subgroup $N$. 
	For example, if $N \cong C_2^2$, the entry $[f_3, f_6]$ identifies $(1,0) \mapsto f_3$ and $(0,1) \mapsto f_6$.
	
	\item 
	\textbf{Structure of $N$}: The isomorphism type of $N$, such as $C_2^2$, $C_4^2$, or $C_2^4$.
	
	\item 
	\textbf{$G$-inv}: The $G$-invariance of the 2-cocycle $\omega \in Z^2(\widehat{N}, \mathbb{T})$. 
	For $N \cong C_2^2$ or $C_4^2$, $\bigcirc$ and $\times$ denote invariance and non-invariance, respectively. 
	For $N \cong C_2^4$, the indices of invariant cocycles from Table \ref{nondeg 28 cocycles} are listed.
	
	\item 
	\textbf{$G \cong G^\omega$}: Isomorphism between $G$ and the twisted group $G^\omega$. $\bigcirc$ indicates $G \cong G^\omega$, while $\times$ indicates $G \not\cong G^\omega$.
	
	\item 
	\textbf{Number of $G^\omega$}: The index of $G^\omega$ among groups of order 64 with the same order list, as classified by the \texttt{SmallCategory} function.
\end{enumerate}

\begin{table}[H]
	\begin{center}
		\caption{Candidate subgroups and the resulting twisted groups $G^\omega$}
		\begin{tabular}{|c|c|c|c|c|}
			Subgroup & Structure of \(N\) & \(G\)-inv & \(G^\omega\simeq G\) & Number of \(G^\omega\)  \\ \hline
			
			$[ f_3, f_6 ]$ & $C_2\times C_2$ & $\bigcirc$ & $\bigcirc$ & 5 \\ \hline
			$[ f_4f_6, f_5 ]$ & $C_2\times C_2$ & $\bigcirc$ & $\bigcirc$ & 5 \\ \hline
			$[ f_4, f_5 ]$ & $C_2\times C_2$ & $\bigcirc$ & $\bigcirc$ & 5 \\ \hline
			$[ f_5, f_6 ]$ & $C_2\times C_2$ & $\bigcirc$ & $\bigcirc$ & 5 \\ \hline
			$[ f_3f_5, f_6 ]$ & $C_2\times C_2$ & $\bigcirc$ & $\bigcirc$ & 5 \\ \hline
			$[ f_5, f_6, f_3, f_4 ]$ & $C_2\times C_2\times C_2\times C_2$ & 1,2  & $\bigcirc$  & 5  \\ \hline $[ f_1f_4, f_2 ]$ & $C_4\times C_4$ & $\bigcirc$ & $\times$ &
			7 \\ \hline
			$[ f_1f_3f_4f_6, f_2 ]$ & $C_4\times C_4$ & $\times$ & - & - \\ \hline
			$[ f_1f_2f_6, f_1f_4f_5f_6 ]$ & $C_4\times C_4$ & $\times$ & - & - \\ \hline
		\end{tabular}
	\end{center}
\end{table}
 
For instance:
\begin{itemize}[leftmargin=*]
	\item 
	The row for $N \cong C_2^2$ generated by $f_3$ and $f_6$ represents a case where the 2-cocycle is $G$-invariant, but the twisted group $G^\omega$ remains isomorphic to $G$.
	
	\item 
	The row for $N \cong C_2^4$ generated by $[f_5, f_6, f_3, f_4]$ represents a case where the $G$-invariant 2-cocycles (indices 1 and 2 from Table \ref{nondeg 28 cocycles}) result in $G^\omega$ being isomorphic to $G$.
	
	\item 
	The row for $N \cong C_4^2$ generated by $f_1f_4$ and $f_2$ represents a case where the 2-cocycle is $G$-invariant, yet $G^\omega$ is non-isomorphic to $G$. This $G^\omega$ is identified as group number 7 among those with the same order list.
\end{itemize}

\subsection{Order list \([1,1,2,4,8,16,32]\)}
\subsubsection{\(C_{64}\)}
As the group is abelian, no non-isomorphic groups exist that are isocategorical to it.

\subsection{Order list \([1,1,34,4,8,16,0]\)}
\subsubsection{\(Q_{64}\)}
Since only one group belongs to this order list, there are no non-isomorphic groups that are monoidally equivalent to it.

\subsection{Order list \([1,3,4,8,16,32,0]\)}
\subsubsection{\(C_{32}\times C_2\)}
As the group is abelian, no non-isomorphic groups exist that are isocategorical to it.
\subsubsection{\(C_{32}:C_2\)}
There are exactly two groups belonging to this order list. 
Since one of them is abelian, no non-isomorphic groups exist that are monoidally equivalent to either group.

\subsection{Order list \([1,3,4,40,16,0,0]\)}
\subsubsection{\(C_4.D_{16}=C_8.(C4\times C_2)\)}
Since only one group belongs to this order list, there are no non-isomorphic groups that are monoidally equivalent to it.

\subsection{Order list \([1,3,12,16,32,0,0]\)}
\subsubsection{\(C_{16}\times C_4\)}
As the group is abelian, no non-isomorphic groups exist that are isocategorical to it.
\subsubsection{\(C_{16}: C_4\)}No group isocategorical to this group exists.
\begin{table}[H]
	\begin{center}
		\caption{Group table for generators}
\begin{tabular}{c||c|c|c|c|c|c|}
			& $f_1$ & $f_2$ & $f_3$ & $f_4$ & $f_5$ & $f_6$ \\ \hline\hline
			$f_1$ & $f_3$ & $f_1f_2$ & $f_1f_3$ & $f_1f_4$ & $f_1f_5$ & $f_1f_6$ \\ \hline
			$f_2$ & $f_1f_2f_6$ & $f_4$ & $f_2f_3$ & $f_2f_4$ & $f_2f_5$ & $f_2f_6$ \\ \hline
			$f_3$ & $f_1f_3$ & $f_2f_3$ & $f_5$ & $f_3f_4$ & $f_3f_5$ & $f_3f_6$ \\ \hline
			$f_4$ & $f_1f_4$ & $f_2f_4$ & $f_3f_4$ & $1$ & $f_4f_5$ & $f_4f_6$ \\ \hline
			$f_5$ & $f_1f_5$ & $f_2f_5$ & $f_3f_5$ & $f_4f_5$ & $f_6$ & $f_5f_6$ \\ \hline
			$f_6$ & $f_1f_6$ & $f_2f_6$ & $f_3f_6$ & $f_4f_6$ & $f_5f_6$ & $1$ \\ \hline
		\end{tabular}
	\end{center}
\end{table}

\begin{table}[H]
	\begin{center}
		\caption{Candidate subgroups and the resulting twisted groups $G^\omega$}
		\begin{tabular}{|c|c|c|c|c|}
			Subgroup & Structure of \(N\) & \(G\)-inv & \(G^\omega\cong G\) & Number of \(G^\omega\)  \\ \hline
			
			$[ f_4, f_6 ]$ & $C_2\times C_2$ & $\bigcirc$ & $\bigcirc$ & 2 \\ \hline 
			$[ f_2, f_5 ]$ & $C_4\times C_4$ & $\bigcirc$ & $\bigcirc$ & 2 \\ \hline 
		\end{tabular}
	\end{center}
\end{table}
\subsubsection{\(C_{16}: C_4\)}No group isocategorical to this group exists.
\begin{table}[H]
	\begin{center}
		\caption{Group table for generators}
		\begin{tabular}{c||c|c|c|c|c|c|}
			& $f_1$ & $f_2$ & $f_3$ & $f_4$ & $f_5$ & $f_6$ \\ \hline\hline
			$f_1$ & $f_3$ & $f_1f_2$ & $f_1f_3$ & $f_1f_4$ & $f_1f_5$ & $f_1f_6$ \\ \hline
			$f_2$ & $f_1f_2f_5$ & $f_4$ & $f_2f_3$ & $f_2f_4$ & $f_2f_5$ & $f_2f_6$ \\ \hline
			$f_3$ & $f_1f_3$ & $f_2f_3f_6$ & $f_5$ & $f_3f_4$ & $f_3f_5$ & $f_3f_6$ \\ \hline
			$f_4$ & $f_1f_4f_6$ & $f_2f_4$ & $f_3f_4$ & $1$ & $f_4f_5$ & $f_4f_6$ \\ \hline
			$f_5$ & $f_1f_5$ & $f_2f_5$ & $f_3f_5$ & $f_4f_5$ & $f_6$ & $f_5f_6$ \\ \hline
			$f_6$ & $f_1f_6$ & $f_2f_6$ & $f_3f_6$ & $f_4f_6$ & $f_5f_6$ & $1$ \\ \hline
		\end{tabular}
	\end{center}
\end{table}

\begin{table}[H]
	\begin{center}
		\caption{Candidate subgroups and the resulting twisted groups $G^\omega$}
		\begin{tabular}{|c|c|c|c|c|}
			Subgroup & Structure of \(N\) & \(G\)-inv & \(G^\omega\cong G\) & Number of \(G^\omega\)  \\ \hline
			
			$[ f_4, f_6 ]$ & $C_2\times C_2$ & $\bigcirc$ & $\bigcirc$ & 3 \\ \hline 
			$[ f_2, f_5 ]$ & $C_4\times C_4$ & $\bigcirc$ & $\bigcirc$ & 3 \\ \hline
		\end{tabular}
	\end{center}
\end{table}
\subsubsection{\(C_4: C_{16}\)}
\begin{table}[H]
	\begin{center}
		\caption{Group table for generators}
		\begin{tabular}{c||c|c|c|c|c|c|}
			& $f_1$ & $f_2$ & $f_3$ & $f_4$ & $f_5$ & $f_6$ \\ \hline\hline
			$f_1$ & $f_4$ & $f_1f_2$ & $f_1f_3$ & $f_1f_4$ & $f_1f_5$ & $f_1f_6$ \\ \hline
			$f_2$ & $f_1f_2f_3$ & $f_3$ & $f_2f_3$ & $f_2f_4$ & $f_2f_5$ & $f_2f_6$ \\ \hline
			$f_3$ & $f_1f_3$ & $f_2f_3$ & $1$ & $f_3f_4$ & $f_3f_5$ & $f_3f_6$ \\ \hline
			$f_4$ & $f_1f_4$ & $f_2f_4$ & $f_3f_4$ & $f_5$ & $f_4f_5$ & $f_4f_6$ \\ \hline
			$f_5$ & $f_1f_5$ & $f_2f_5$ & $f_3f_5$ & $f_4f_5$ & $f_6$ & $f_5f_6$ \\ \hline
			$f_6$ & $f_1f_6$ & $f_2f_6$ & $f_3f_6$ & $f_4f_6$ & $f_5f_6$ & $1$ \\ \hline
		\end{tabular}
	\end{center}
\end{table}

\begin{table}[H]
	\begin{center}
		\caption{Candidate subgroups and the resulting twisted groups $G^\omega$}
		\begin{tabular}{|c|c|c|c|c|}
			Subgroup & Structure of \(N\) & \(G\)-inv & \(G^\omega\cong G\) & Number of \(G^\omega\)  \\ \hline
			
			$[ f_3, f_6 ]$ & $C_2\times C_2$ & $\bigcirc$ & $\bigcirc$ & 4 \\ \hline 
			$[ f_2, f_5 ]$ & $C_4\times C_4$ & $\times$ & - & - \\ \hline
		\end{tabular}
	\end{center}
\end{table}
\subsubsection{\(C_8.D_8=C_4.(C_8\times C_2)\)}No group isocategorical to this group exists.
\begin{table}[H]
	\begin{center}
		\caption{Group table for generators}
		\begin{tabular}{c||c|c|c|c|c|c|}
			& $f_1$ & $f_2$ & $f_3$ & $f_4$ & $f_5$ & $f_6$ \\ \hline\hline
			$f_1$ & $f_4$ & $f_1f_2$ & $f_1f_3$ & $f_1f_4$ & $f_1f_5$ & $f_1f_6$ \\ \hline
			$f_2$ & $f_1f_2f_3$ & $f_3$ & $f_2f_3$ & $f_2f_4$ & $f_2f_5$ & $f_2f_6$ \\ \hline
			$f_3$ & $f_1f_3f_6$ & $f_2f_3$ & $f_6$ & $f_3f_4$ & $f_3f_5$ & $f_3f_6$ \\ \hline
			$f_4$ & $f_1f_4$ & $f_2f_4$ & $f_3f_4$ & $f_5$ & $f_4f_5$ & $f_4f_6$ \\ \hline
			$f_5$ & $f_1f_5$ & $f_2f_5$ & $f_3f_5$ & $f_4f_5$ & $f_6$ & $f_5f_6$ \\ \hline
			$f_6$ & $f_1f_6$ & $f_2f_6$ & $f_3f_6$ & $f_4f_6$ & $f_5f_6$ & $1$ \\ \hline
		\end{tabular}
	\end{center}
\end{table}

\begin{table}[H]
	\begin{center}
		\caption{Candidate subgroups and the resulting twisted groups $G^\omega$}
		\begin{tabular}{|c|c|c|c|c|}
			Subgroup & Structure of \(N\) & \(G\)-inv & \(G^\omega\cong G\) & Number of \(G^\omega\)  \\ \hline
			
			$[ f_3f_5, f_6 ]$ & $C_2\times C_2$ & $\bigcirc$ & $\bigcirc$ & 5 \\ \hline 
			$[ f_2f_4, f_5 ]$ & $C_4\times C_4$ & $\times$ & - & - \\ \hline
		\end{tabular}
	\end{center}
\end{table}

\subsection{Order list \([1,3,12,48,0,0,0]\)}
\subsubsection{\(C_8\times C_8\)}
As the group is abelian, no non-isomorphic groups exist that are isocategorical to it.
\subsubsection{\(C_8 :C_8\)}No group isocategorical to this group exists.
\begin{table}[H]
	\begin{center}
		\caption{Group table for generators}
		\begin{tabular}{c||c|c|c|c|c|c|}
			& $f_1$ & $f_2$ & $f_3$ & $f_4$ & $f_5$ & $f_6$ \\ \hline\hline
			$f_1$ & $f_3$ & $f_1f_2$ & $f_1f_3$ & $f_1f_4$ & $f_1f_5$ & $f_1f_6$ \\ \hline
			$f_2$ & $f_1f_2f_5$ & $f_4$ & $f_2f_3$ & $f_2f_4$ & $f_2f_5$ & $f_2f_6$ \\ \hline
			$f_3$ & $f_1f_3$ & $f_2f_3$ & $f_5$ & $f_3f_4$ & $f_3f_5$ & $f_3f_6$ \\ \hline
			$f_4$ & $f_1f_4$ & $f_2f_4$ & $f_3f_4$ & $f_6$ & $f_4f_5$ & $f_4f_6$ \\ \hline
			$f_5$ & $f_1f_5$ & $f_2f_5$ & $f_3f_5$ & $f_4f_5$ & $1$ & $f_5f_6$ \\ \hline
			$f_6$ & $f_1f_6$ & $f_2f_6$ & $f_3f_6$ & $f_4f_6$ & $f_5f_6$ & $1$ \\ \hline
		\end{tabular}
	\end{center}
\end{table}

\begin{table}[H]
	\begin{center}
		\caption{Candidate subgroups and the resulting twisted groups $G^\omega$}
		\begin{tabular}{|c|c|c|c|c|}
			Subgroup & Structure of \(N\) & \(G\)-inv & \(G^\omega\cong G\) & Number of \(G^\omega\)  \\ \hline
			
			$[ f_5, f_6 ]$ & $C_2\times C_2$ & $\bigcirc$ & $\bigcirc$ & 2 \\ \hline 
			$[ f_3, f_4 ]$ & $C_4\times C_4$ & $\bigcirc$ & $\bigcirc$ & 2 \\ \hline 
		\end{tabular}
	\end{center}
\end{table}
\subsubsection{\(C_8 :C_8\)}No group isocategorical to this group exists.
\begin{table}[H]
	\begin{center}
		\caption{Group table for generators}
		\begin{tabular}{c||c|c|c|c|c|c|}
			& $f_1$ & $f_2$ & $f_3$ & $f_4$ & $f_5$ & $f_6$ \\ \hline\hline
			$f_1$ & $f_4$ & $f_1f_2$ & $f_1f_3$ & $f_1f_4$ & $f_1f_5$ & $f_1f_6$ \\ \hline
			$f_2$ & $f_1f_2f_3$ & $f_3$ & $f_2f_3$ & $f_2f_4$ & $f_2f_5$ & $f_2f_6$ \\ \hline
			$f_3$ & $f_1f_3f_5$ & $f_2f_3$ & $f_5$ & $f_3f_4$ & $f_3f_5$ & $f_3f_6$ \\ \hline
			$f_4$ & $f_1f_4$ & $f_2f_4$ & $f_3f_4$ & $f_6$ & $f_4f_5$ & $f_4f_6$ \\ \hline
			$f_5$ & $f_1f_5$ & $f_2f_5$ & $f_3f_5$ & $f_4f_5$ & $1$ & $f_5f_6$ \\ \hline
			$f_6$ & $f_1f_6$ & $f_2f_6$ & $f_3f_6$ & $f_4f_6$ & $f_5f_6$ & $1$ \\ \hline
		\end{tabular}
	\end{center}
\end{table}

\begin{table}[H]
	\begin{center}
		\caption{Candidate subgroups and the resulting twisted groups $G^\omega$}
		\begin{tabular}{|c|c|c|c|c|}
			Subgroup & Structure of \(N\) & \(G\)-inv & \(G^\omega\cong G\) & Number of \(G^\omega\)  \\ \hline
			
			$[ f_5, f_6 ]$ & $C_2\times C_2$ & $\bigcirc$ & $\bigcirc$ & 3 \\ \hline 
			$[ f_3, f_4 ]$ & $C_4\times C_4$ & $\times$ & - & - \\ \hline
		\end{tabular}
	\end{center}
\end{table}
\subsubsection{\(C_8 :C_8\)}No group isocategorical to this group exists.
\begin{table}[H]
	\begin{center}
		\caption{Group table for generators}
		\begin{tabular}{c||c|c|c|c|c|c|}
			& $f_1$ & $f_2$ & $f_3$ & $f_4$ & $f_5$ & $f_6$ \\ \hline\hline
			$f_1$ & $f_4$ & $f_1f_2$ & $f_1f_3$ & $f_1f_4$ & $f_1f_5$ & $f_1f_6$ \\ \hline
			$f_2$ & $f_1f_2f_3$ & $f_3f_5$ & $f_2f_3$ & $f_2f_4$ & $f_2f_5$ & $f_2f_6$ \\ \hline
			$f_3$ & $f_1f_3f_5$ & $f_2f_3$ & $f_5$ & $f_3f_4$ & $f_3f_5$ & $f_3f_6$ \\ \hline
			$f_4$ & $f_1f_4$ & $f_2f_4$ & $f_3f_4$ & $f_6$ & $f_4f_5$ & $f_4f_6$ \\ \hline
			$f_5$ & $f_1f_5$ & $f_2f_5$ & $f_3f_5$ & $f_4f_5$ & $1$ & $f_5f_6$ \\ \hline
			$f_6$ & $f_1f_6$ & $f_2f_6$ & $f_3f_6$ & $f_4f_6$ & $f_5f_6$ & $1$ \\ \hline
		\end{tabular}
	\end{center}
\end{table}

\begin{table}[H]
	\begin{center}
		\caption{Candidate subgroups and the resulting twisted groups $G^\omega$}
		\begin{tabular}{|c|c|c|c|c|}
			Subgroup & Structure of \(N\) & \(G\)-inv & \(G^\omega\cong G\) & Number of \(G^\omega\)  \\ \hline
			
			$[ f_5, f_6 ]$ & $C_2\times C_2$ & $\bigcirc$ & $\bigcirc$ & 4 \\ \hline 
			$[ f_3, f_4 ]$ & $C_4\times C_4$ & $\times$ & - & - \\ \hline
		\end{tabular}
	\end{center}
\end{table}

\subsection{Order list \([1,3,20,24,16,0,0]\)}
\subsubsection{\(C_2 . ((C_8 \times C_2) : C_2) = C_8 . (C_4 \times C_2)\)}No group isocategorical to this group exists.
\begin{table}[H]
	\begin{center}
		\caption{Group table for generators}
		\begin{tabular}{c||c|c|c|c|c|c|}
			& $f_1$ & $f_2$ & $f_3$ & $f_4$ & $f_5$ & $f_6$ \\ \hline\hline
			$f_1$ & $f_4$ & $f_1f_2$ & $f_1f_3$ & $f_1f_4$ & $f_1f_5$ & $f_1f_6$ \\ \hline
			$f_2$ & $f_1f_2f_3$ & $f_6$ & $f_2f_3$ & $f_2f_4$ & $f_2f_5$ & $f_2f_6$ \\ \hline
			$f_3$ & $f_1f_3f_5$ & $f_2f_3f_5f_6$ & $f_5$ & $f_3f_4$ & $f_3f_5$ & $f_3f_6$ \\ \hline
			$f_4$ & $f_1f_4$ & $f_2f_4f_6$ & $f_3f_4$ & $f_6$ & $f_4f_5$ & $f_4f_6$ \\ \hline
			$f_5$ & $f_1f_5f_6$ & $f_2f_5f_6$ & $f_3f_5$ & $f_4f_5$ & $f_6$ & $f_5f_6$ \\ \hline
			$f_6$ & $f_1f_6$ & $f_2f_6$ & $f_3f_6$ & $f_4f_6$ & $f_5f_6$ & $1$ \\ \hline
		\end{tabular}
	\end{center}
\end{table}

\begin{table}[H]
	\begin{center}
		\caption{Candidate subgroups and the resulting twisted groups $G^\omega$}
		\begin{tabular}{|c|c|c|c|c|}
			Subgroup & Structure of \(N\) & \(G\)-inv & \(G^\omega\cong G\) & Number of \(G^\omega\)  \\ \hline
			
			$[ f_4f_5, f_6 ]$ & $C_2\times C_2$ & $\bigcirc$ & $\bigcirc$ & 1 \\ \hline
		\end{tabular}
	\end{center}
\end{table}
\subsubsection{\(C_{16} : C_4\)}No group isocategorical to this group exists.
\begin{table}[H]
	\begin{center}
		\caption{Group table for generators}
		\begin{tabular}{c||c|c|c|c|c|c|}
			& $f_1$ & $f_2$ & $f_3$ & $f_4$ & $f_5$ & $f_6$ \\ \hline\hline
			$f_1$ & $f_4$ & $f_1f_2$ & $f_1f_3$ & $f_1f_4$ & $f_1f_5$ & $f_1f_6$ \\ \hline
			$f_2$ & $f_1f_2f_3$ & $f_3$ & $f_2f_3$ & $f_2f_4$ & $f_2f_5$ & $f_2f_6$ \\ \hline
			$f_3$ & $f_1f_3f_5$ & $f_2f_3$ & $f_5$ & $f_3f_4$ & $f_3f_5$ & $f_3f_6$ \\ \hline
			$f_4$ & $f_1f_4$ & $f_2f_4f_6$ & $f_3f_4$ & $1$ & $f_4f_5$ & $f_4f_6$ \\ \hline
			$f_5$ & $f_1f_5f_6$ & $f_2f_5$ & $f_3f_5$ & $f_4f_5$ & $f_6$ & $f_5f_6$ \\ \hline
			$f_6$ & $f_1f_6$ & $f_2f_6$ & $f_3f_6$ & $f_4f_6$ & $f_5f_6$ & $1$ \\ \hline
		\end{tabular}
	\end{center}
\end{table}

\begin{table}[H]
	\begin{center}
		\caption{Candidate subgroups and the resulting twisted groups $G^\omega$}
		\begin{tabular}{|c|c|c|c|c|}
			Subgroup & Structure of \(N\) & \(G\)-inv & \(G^\omega\cong G\) & Number of \(G^\omega\)  \\ \hline
			
			$[ f_4, f_6 ]$ & $C_2\times C_2$ & $\bigcirc$ & $\bigcirc$ & 2 \\ \hline 
		\end{tabular}
	\end{center}
\end{table}

\subsection{Order list \([1,3,28,32,0,0,0]\)}
\subsubsection{\(Q_8 : C_8\)}No group isocategorical to this group exists.
\begin{table}[H]
	\begin{center}
		\caption{Group table for generators}
		\begin{tabular}{c||c|c|c|c|c|c|}
			& $f_1$ & $f_2$ & $f_3$ & $f_4$ & $f_5$ & $f_6$ \\ \hline\hline
			$f_1$ & $f_4$ & $f_1f_2$ & $f_1f_3$ & $f_1f_4$ & $f_1f_5$ & $f_1f_6$ \\ \hline
			$f_2$ & $f_1f_2f_3$ & $f_5$ & $f_2f_3$ & $f_2f_4$ & $f_2f_5$ & $f_2f_6$ \\ \hline
			$f_3$ & $f_1f_3f_5$ & $f_2f_3f_5$ & $f_5$ & $f_3f_4$ & $f_3f_5$ & $f_3f_6$ \\ \hline
			$f_4$ & $f_1f_4$ & $f_2f_4$ & $f_3f_4$ & $f_6$ & $f_4f_5$ & $f_4f_6$ \\ \hline
			$f_5$ & $f_1f_5$ & $f_2f_5$ & $f_3f_5$ & $f_4f_5$ & $1$ & $f_5f_6$ \\ \hline
			$f_6$ & $f_1f_6$ & $f_2f_6$ & $f_3f_6$ & $f_4f_6$ & $f_5f_6$ & $1$ \\ \hline
		\end{tabular}
	\end{center}
\end{table}

\begin{table}[H]
	\begin{center}
		\caption{Candidate subgroups and the resulting twisted groups $G^\omega$}
		\begin{tabular}{|c|c|c|c|c|}
			Subgroup & Structure of \(N\) & \(G\)-inv & \(G^\omega\cong G\) & Number of \(G^\omega\)  \\ \hline
			
			$[ f_5, f_6 ]$ & $C_2\times C_2$ & $\bigcirc$ & $\bigcirc$ & 1 \\ \hline 
			$[ f_3, f_4 ]$ & $C_4\times C_4$ & $\times$ & - & - \\ \hline
		\end{tabular}
	\end{center}
\end{table}
\subsubsection{\((C_2 \times C_2) . ((C_4 \times C_2) : C_2) = (C_4 \times C_2) . (C_4 \times C_2)\)}No group isocategorical to this group exists.
\begin{table}[H]
	\begin{center}
		\caption{Group table for generators}
		\begin{tabular}{c||c|c|c|c|c|c|}
			& $f_1$ & $f_2$ & $f_3$ & $f_4$ & $f_5$ & $f_6$ \\ \hline\hline
			$f_1$ & $f_4$ & $f_1f_2$ & $f_1f_3$ & $f_1f_4$ & $f_1f_5$ & $f_1f_6$ \\ \hline
			$f_2$ & $f_1f_2f_3$ & $f_5$ & $f_2f_3$ & $f_2f_4$ & $f_2f_5$ & $f_2f_6$ \\ \hline
			$f_3$ & $f_1f_3f_5$ & $f_2f_3f_6$ & $f_6$ & $f_3f_4$ & $f_3f_5$ & $f_3f_6$ \\ \hline
			$f_4$ & $f_1f_4$ & $f_2f_4f_5f_6$ & $f_3f_4$ & $f_5$ & $f_4f_5$ & $f_4f_6$ \\ \hline
			$f_5$ & $f_1f_5$ & $f_2f_5$ & $f_3f_5$ & $f_4f_5$ & $1$ & $f_5f_6$ \\ \hline
			$f_6$ & $f_1f_6$ & $f_2f_6$ & $f_3f_6$ & $f_4f_6$ & $f_5f_6$ & $1$ \\ \hline
		\end{tabular}
	\end{center}
\end{table}

\begin{table}[H]
	\begin{center}
		\caption{Candidate subgroups and the resulting twisted groups $G^\omega$}
		\begin{tabular}{|c|c|c|c|c|}
			Subgroup & Structure of \(N\) & \(G\)-inv & \(G^\omega\cong G\) & Number of \(G^\omega\)  \\ \hline
			
			$[ f_5, f_6 ]$ & $C_2\times C_2$ & $\bigcirc$ & $\bigcirc$ & 2 \\ \hline 
			$[ f_3f_4f_6, f_3 ]$ & $C_4\times C_4$ & $\bigcirc$ & $\bigcirc$ & 2 \\ \hline 
		\end{tabular}
	\end{center}
\end{table}
\subsubsection{\((C_2 \times C_2) . ((C_4 \times C_2) : C_2) = (C_4 \times C_2) . (C_4 \times C_2)\)}No group isocategorical to this group exists.
\begin{table}[H]
	\begin{center}
		\caption{Group table for generators}
		\begin{tabular}{c||c|c|c|c|c|c|}
			& $f_1$ & $f_2$ & $f_3$ & $f_4$ & $f_5$ & $f_6$ \\ \hline\hline
			$f_1$ & $f_4$ & $f_1f_2$ & $f_1f_3$ & $f_1f_4$ & $f_1f_5$ & $f_1f_6$ \\ \hline
			$f_2$ & $f_1f_2f_3$ & $f_5$ & $f_2f_3$ & $f_2f_4$ & $f_2f_5$ & $f_2f_6$ \\ \hline
			$f_3$ & $f_1f_3f_5$ & $f_2f_3f_6$ & $f_6$ & $f_3f_4$ & $f_3f_5$ & $f_3f_6$ \\ \hline
			$f_4$ & $f_1f_4$ & $f_2f_4f_5f_6$ & $f_3f_4$ & $f_5f_6$ & $f_4f_5$ & $f_4f_6$ \\ \hline
			$f_5$ & $f_1f_5$ & $f_2f_5$ & $f_3f_5$ & $f_4f_5$ & $1$ & $f_5f_6$ \\ \hline
			$f_6$ & $f_1f_6$ & $f_2f_6$ & $f_3f_6$ & $f_4f_6$ & $f_5f_6$ & $1$ \\ \hline
		\end{tabular}
	\end{center}
\end{table}

\begin{table}[H]
	\begin{center}
		\caption{Candidate subgroups and the resulting twisted groups $G^\omega$}
		\begin{tabular}{|c|c|c|c|c|}
			Subgroup & Structure of \(N\) & \(G\)-inv & \(G^\omega\cong G\) & Number of \(G^\omega\)  \\ \hline
			
			$[ f_5, f_6 ]$ & $C_2\times C_2$ & $\bigcirc$ & $\bigcirc$ & 3 \\ \hline 
			$[ f_3f_4, f_3 ]$ & $C_4\times C_4$ & $\times$ & - & - \\ \hline
		\end{tabular}
	\end{center}
\end{table}
\subsubsection{\((C_2 \times C_2) . ((C_4 \times C_2) : C_2) = (C_4 \times C_2) . (C_4 \times C_2)\)}No group isocategorical to this group exists.
\begin{table}[H]
	\begin{center}
		\caption{Group table for generators}
		\begin{tabular}{c||c|c|c|c|c|c|}
			& $f_1$ & $f_2$ & $f_3$ & $f_4$ & $f_5$ & $f_6$ \\ \hline\hline
			$f_1$ & $f_4$ & $f_1f_2$ & $f_1f_3$ & $f_1f_4$ & $f_1f_5$ & $f_1f_6$ \\ \hline
			$f_2$ & $f_1f_2f_3$ & $f_6$ & $f_2f_3$ & $f_2f_4$ & $f_2f_5$ & $f_2f_6$ \\ \hline
			$f_3$ & $f_1f_3f_5$ & $f_2f_3f_6$ & $f_6$ & $f_3f_4$ & $f_3f_5$ & $f_3f_6$ \\ \hline
			$f_4$ & $f_1f_4$ & $f_2f_4f_5f_6$ & $f_3f_4$ & $f_5f_6$ & $f_4f_5$ & $f_4f_6$ \\ \hline
			$f_5$ & $f_1f_5$ & $f_2f_5$ & $f_3f_5$ & $f_4f_5$ & $1$ & $f_5f_6$ \\ \hline
			$f_6$ & $f_1f_6$ & $f_2f_6$ & $f_3f_6$ & $f_4f_6$ & $f_5f_6$ & $1$ \\ \hline
		\end{tabular}
	\end{center}
\end{table}

\begin{table}[H]
	\begin{center}
		\caption{Candidate subgroups and the resulting twisted groups $G^\omega$}
		\begin{tabular}{|c|c|c|c|c|}
			Subgroup & Structure of \(N\) & \(G\)-inv & \(G^\omega\cong G\) & Number of \(G^\omega\)  \\ \hline
			
			$[ f_5, f_6 ]$ & $C_2\times C_2$ & $\bigcirc$ & $\bigcirc$ & 4 \\ \hline 
			$[ f_3f_4, f_3 ]$ & $C_4\times C_4$ & $\times$ & - & - \\ \hline
		\end{tabular}
	\end{center}
\end{table}
\subsubsection{\(C_2 . ((C_2 \times C_2 \times C_2) : C_4) = (C_4 \times C_2) . (C_4 \times C_2)\)}No group isocategorical to this group exists.
\begin{table}[H]
	\begin{center}
		\caption{Group table for generators}
		\begin{tabular}{c||c|c|c|c|c|c|}
			& $f_1$ & $f_2$ & $f_3$ & $f_4$ & $f_5$ & $f_6$ \\ \hline\hline
			$f_1$ & $f_4$ & $f_1f_2$ & $f_1f_3$ & $f_1f_4$ & $f_1f_5$ & $f_1f_6$ \\ \hline
			$f_2$ & $f_1f_2f_3$ & $f_6$ & $f_2f_3$ & $f_2f_4$ & $f_2f_5$ & $f_2f_6$ \\ \hline
			$f_3$ & $f_1f_3f_5$ & $f_2f_3f_6$ & $f_6$ & $f_3f_4$ & $f_3f_5$ & $f_3f_6$ \\ \hline
			$f_4$ & $f_1f_4$ & $f_2f_4f_5f_6$ & $f_3f_4f_6$ & $f_6$ & $f_4f_5$ & $f_4f_6$ \\ \hline
			$f_5$ & $f_1f_5f_6$ & $f_2f_5$ & $f_3f_5$ & $f_4f_5$ & $1$ & $f_5f_6$ \\ \hline
			$f_6$ & $f_1f_6$ & $f_2f_6$ & $f_3f_6$ & $f_4f_6$ & $f_5f_6$ & $1$ \\ \hline
		\end{tabular}
	\end{center}
\end{table}

\begin{table}[H]
	\begin{center}
		\caption{Candidate subgroups and the resulting twisted groups $G^\omega$}
		\begin{tabular}{|c|c|c|c|c|}
			Subgroup & Structure of \(N\) & \(G\)-inv & \(G^\omega\cong G\) & Number of \(G^\omega\)  \\ \hline
			
			$[ f_5, f_6 ]$ & $C_2\times C_2$ & $\bigcirc$ & $\bigcirc$ & 5 \\ \hline 
			$[ f_2f_4f_6, f_3 ]$ & $C_4\times C_4$ & $\bigcirc$ & $\bigcirc$ & 5 \\ \hline
		\end{tabular}
	\end{center}
\end{table}
\subsubsection{\(C_8 \times Q_8\)}No group isocategorical to this group exists.
\begin{table}[H]
	\begin{center}
		\caption{Group table for generators}
		\begin{tabular}{c||c|c|c|c|c|c|}
			& $f_1$ & $f_2$ & $f_3$ & $f_4$ & $f_5$ & $f_6$ \\ \hline\hline
			$f_1$ & $f_5$ & $f_1f_2$ & $f_1f_3$ & $f_1f_4$ & $f_1f_5$ & $f_1f_6$ \\ \hline
			$f_2$ & $f_1f_2f_4$ & $f_4$ & $f_2f_3$ & $f_2f_4$ & $f_2f_5$ & $f_2f_6$ \\ \hline
			$f_3$ & $f_1f_3$ & $f_2f_3$ & $f_4f_5$ & $f_3f_4$ & $f_3f_5$ & $f_3f_6$ \\ \hline
			$f_4$ & $f_1f_4$ & $f_2f_4$ & $f_3f_4$ & $1$ & $f_4f_5$ & $f_4f_6$ \\ \hline
			$f_5$ & $f_1f_5$ & $f_2f_5$ & $f_3f_5$ & $f_4f_5$ & $f_6$ & $f_5f_6$ \\ \hline
			$f_6$ & $f_1f_6$ & $f_2f_6$ & $f_3f_6$ & $f_4f_6$ & $f_5f_6$ & $1$ \\ \hline
		\end{tabular}
	\end{center}
\end{table}

\begin{table}[H]
	\begin{center}
		\caption{Candidate subgroups and the resulting twisted groups $G^\omega$}
		\begin{tabular}{|c|c|c|c|c|}
			Subgroup & Structure of \(N\) & \(G\)-inv & \(G^\omega\cong G\) & Number of \(G^\omega\)  \\ \hline
			
			$[ f_4, f_6 ]$ & $C_2\times C_2$ & $\bigcirc$ & $\bigcirc$ & 6 \\ \hline 
			$[ f_2, f_5 ]$ & $C_4\times C_4$ & $\times$ & - & - \\ \hline
			$[ f_1f_2f_3, f_5 ]$ & $C_4\times C_4$ & $\times$ & - & - \\ \hline
			$[ f_1f_3, f_5 ]$ & $C_4\times C_4$ & $\times$ & - & - \\ \hline
		\end{tabular}
	\end{center}
\end{table}
\subsubsection{\(C_8 : Q_8\)}No group isocategorical to this group exists.
\begin{table}[H]
	\begin{center}
		\caption{Group table for generators}
		\begin{tabular}{c||c|c|c|c|c|c|}
			& $f_1$ & $f_2$ & $f_3$ & $f_4$ & $f_5$ & $f_6$ \\ \hline\hline
			$f_1$ & $f_5$ & $f_1f_2$ & $f_1f_3$ & $f_1f_4$ & $f_1f_5$ & $f_1f_6$ \\ \hline
			$f_2$ & $f_1f_2f_4$ & $f_4$ & $f_2f_3$ & $f_2f_4$ & $f_2f_5$ & $f_2f_6$ \\ \hline
			$f_3$ & $f_1f_3f_6$ & $f_2f_3$ & $f_4f_5$ & $f_3f_4$ & $f_3f_5$ & $f_3f_6$ \\ \hline
			$f_4$ & $f_1f_4$ & $f_2f_4$ & $f_3f_4$ & $1$ & $f_4f_5$ & $f_4f_6$ \\ \hline
			$f_5$ & $f_1f_5$ & $f_2f_5$ & $f_3f_5$ & $f_4f_5$ & $f_6$ & $f_5f_6$ \\ \hline
			$f_6$ & $f_1f_6$ & $f_2f_6$ & $f_3f_6$ & $f_4f_6$ & $f_5f_6$ & $1$ \\ \hline
		\end{tabular}
	\end{center}
\end{table}

\begin{table}[H]
	\begin{center}
		\caption{Candidate subgroups and the resulting twisted groups $G^\omega$}
		\begin{tabular}{|c|c|c|c|c|}
			Subgroup & Structure of \(N\) & \(G\)-inv & \(G^\omega\cong G\) & Number of \(G^\omega\)  \\ \hline
			
			$[ f_4, f_6 ]$ & $C_2\times C_2$ & $\bigcirc$ & $\bigcirc$ & 7 \\ \hline 
			$[ f_1f_2f_3, f_5 ]$ & $C_4\times C_4$ & $\times$ & - & - \\ \hline
			$[ f_2, f_5 ]$ & $C_4\times C_4$ & $\times$ & - & - \\ \hline
			$[ f_1f_3, f_5 ]$ & $C_4\times C_4$ & $\times$ & - & - \\ \hline
		\end{tabular}
	\end{center}
\end{table}

\subsection{Order list \([1,3,36,8,16,0,0]\)}
\subsubsection{\(Q_{16} : C_4\)}No group isocategorical to this group exists.
\begin{table}[H]
	\begin{center}
		\caption{Group table for generators}
		\begin{tabular}{c||c|c|c|c|c|c|}
			& $f_1$ & $f_2$ & $f_3$ & $f_4$ & $f_5$ & $f_6$ \\ \hline\hline
			$f_1$ & $f_4$ & $f_1f_2$ & $f_1f_3$ & $f_1f_4$ & $f_1f_5$ & $f_1f_6$ \\ \hline
			$f_2$ & $f_1f_2f_3$ & $f_6$ & $f_2f_3$ & $f_2f_4$ & $f_2f_5$ & $f_2f_6$ \\ \hline
			$f_3$ & $f_1f_3f_5$ & $f_2f_3f_5$ & $f_5f_6$ & $f_3f_4$ & $f_3f_5$ & $f_3f_6$ \\ \hline
			$f_4$ & $f_1f_4$ & $f_2f_4$ & $f_3f_4$ & $1$ & $f_4f_5$ & $f_4f_6$ \\ \hline
			$f_5$ & $f_1f_5f_6$ & $f_2f_5f_6$ & $f_3f_5$ & $f_4f_5$ & $f_6$ & $f_5f_6$ \\ \hline
			$f_6$ & $f_1f_6$ & $f_2f_6$ & $f_3f_6$ & $f_4f_6$ & $f_5f_6$ & $1$ \\ \hline
		\end{tabular}
	\end{center}
\end{table}

\begin{table}[H]
	\begin{center}
		\caption{Candidate subgroups and the resulting twisted groups $G^\omega$}
		\begin{tabular}{|c|c|c|c|c|}
			Subgroup & Structure of \(N\) & \(G\)-inv & \(G^\omega\cong G\) & Number of \(G^\omega\)  \\ \hline
			
			$[ f_4, f_6 ]$ & $C_2\times C_2$ & $\bigcirc$ & $\bigcirc$ & 1 \\ \hline 
		\end{tabular}
	\end{center}
\end{table}
\subsubsection{\(C_{16} : C_4\)}No group isocategorical to this group exists.
\begin{table}[H]
	\begin{center}
		\caption{Group table for generators}
		\begin{tabular}{c||c|c|c|c|c|c|}
			& $f_1$ & $f_2$ & $f_3$ & $f_4$ & $f_5$ & $f_6$ \\ \hline\hline
			$f_1$ & $f_4$ & $f_1f_2$ & $f_1f_3$ & $f_1f_4$ & $f_1f_5$ & $f_1f_6$ \\ \hline
			$f_2$ & $f_1f_2f_3$ & $f_3f_5$ & $f_2f_3$ & $f_2f_4$ & $f_2f_5$ & $f_2f_6$ \\ \hline
			$f_3$ & $f_1f_3f_5$ & $f_2f_3$ & $f_5f_6$ & $f_3f_4$ & $f_3f_5$ & $f_3f_6$ \\ \hline
			$f_4$ & $f_1f_4$ & $f_2f_4$ & $f_3f_4$ & $1$ & $f_4f_5$ & $f_4f_6$ \\ \hline
			$f_5$ & $f_1f_5f_6$ & $f_2f_5$ & $f_3f_5$ & $f_4f_5$ & $f_6$ & $f_5f_6$ \\ \hline
			$f_6$ & $f_1f_6$ & $f_2f_6$ & $f_3f_6$ & $f_4f_6$ & $f_5f_6$ & $1$ \\ \hline
		\end{tabular}
	\end{center}
\end{table}

\begin{table}[H]
	\begin{center}
		\caption{Candidate subgroups and the resulting twisted groups $G^\omega$}
		\begin{tabular}{|c|c|c|c|c|}
			Subgroup & Structure of \(N\) & \(G\)-inv & \(G^\omega\cong G\) & Number of \(G^\omega\)  \\ \hline
			
			$[ f_4, f_6 ]$ & $C_2\times C_2$ & $\bigcirc$ & $\bigcirc$ & 2 \\ \hline
		\end{tabular}
	\end{center}
\end{table}
\subsubsection{\(C_{16} : C_4\)}No group isocategorical to this group exists.
\begin{table}[H]
	\begin{center}
		\caption{Group table for generators}
		\begin{tabular}{c||c|c|c|c|c|c|}
			& $f_1$ & $f_2$ & $f_3$ & $f_4$ & $f_5$ & $f_6$ \\ \hline\hline
			$f_1$ & $f_4$ & $f_1f_2$ & $f_1f_3$ & $f_1f_4$ & $f_1f_5$ & $f_1f_6$ \\ \hline
			$f_2$ & $f_1f_2f_3$ & $f_3f_5f_6$ & $f_2f_3$ & $f_2f_4$ & $f_2f_5$ & $f_2f_6$ \\ \hline
			$f_3$ & $f_1f_3f_5$ & $f_2f_3$ & $f_5f_6$ & $f_3f_4$ & $f_3f_5$ & $f_3f_6$ \\ \hline
			$f_4$ & $f_1f_4$ & $f_2f_4$ & $f_3f_4$ & $1$ & $f_4f_5$ & $f_4f_6$ \\ \hline
			$f_5$ & $f_1f_5f_6$ & $f_2f_5$ & $f_3f_5$ & $f_4f_5$ & $f_6$ & $f_5f_6$ \\ \hline
			$f_6$ & $f_1f_6$ & $f_2f_6$ & $f_3f_6$ & $f_4f_6$ & $f_5f_6$ & $1$ \\ \hline
		\end{tabular}
	\end{center}
\end{table}

\begin{table}[H]
	\begin{center}
		\caption{Candidate subgroups and the resulting twisted groups $G^\omega$}
		\begin{tabular}{|c|c|c|c|c|}
			Subgroup & Structure of \(N\) & \(G\)-inv & \(G^\omega\cong G\) & Number of \(G^\omega\)  \\ \hline
			
			$[ f_4, f_6 ]$ & $C_2\times C_2$ & $\bigcirc$ & $\bigcirc$ & 3 \\ \hline
		\end{tabular}
	\end{center}
\end{table}
\subsubsection{\(C_2 \times Q_{32}\)}
\begin{table}[H]
	\begin{center}
		\caption{Group table for generators}
		\begin{tabular}{c||c|c|c|c|c|c|}
			& $f_1$ & $f_2$ & $f_3$ & $f_4$ & $f_5$ & $f_6$ \\ \hline\hline
			$f_1$ & $f_6$ & $f_1f_2$ & $f_1f_3$ & $f_1f_4$ & $f_1f_5$ & $f_1f_6$ \\ \hline
			$f_2$ & $f_1f_2f_4$ & $f_6$ & $f_2f_3$ & $f_2f_4$ & $f_2f_5$ & $f_2f_6$ \\ \hline
			$f_3$ & $f_1f_3$ & $f_2f_3$ & $1$ & $f_3f_4$ & $f_3f_5$ & $f_3f_6$ \\ \hline
			$f_4$ & $f_1f_4f_5$ & $f_2f_4f_5$ & $f_3f_4$ & $f_5f_6$ & $f_4f_5$ & $f_4f_6$ \\ \hline
			$f_5$ & $f_1f_5f_6$ & $f_2f_5f_6$ & $f_3f_5$ & $f_4f_5$ & $f_6$ & $f_5f_6$ \\ \hline
			$f_6$ & $f_1f_6$ & $f_2f_6$ & $f_3f_6$ & $f_4f_6$ & $f_5f_6$ & $1$ \\ \hline
		\end{tabular}
	\end{center}
\end{table}

\begin{table}[H]
	\begin{center}
		\caption{Candidate subgroups and the resulting twisted groups $G^\omega$}
		\begin{tabular}{|c|c|c|c|c|}
			Subgroup & Structure of \(N\) & \(G\)-inv & \(G^\omega\cong G\) & Number of \(G^\omega\)  \\ \hline
			
			$[ f_3, f_6 ]$ & $C_2\times C_2$ & $\bigcirc$ & $\bigcirc$ & 4 \\ \hline
		\end{tabular}
	\end{center}
\end{table}

\subsection{Order list \([1,3,44,16,0,0,0]\)}
\subsubsection{\(C_4 \times Q_{16}\)}
\begin{table}[H]
	\begin{center}
		\caption{Group table for generators}
		\begin{tabular}{c||c|c|c|c|c|c|}
			& $f_1$ & $f_2$ & $f_3$ & $f_4$ & $f_5$ & $f_6$ \\ \hline\hline
			$f_1$ & $f_5$ & $f_1f_2$ & $f_1f_3$ & $f_1f_4$ & $f_1f_5$ & $f_1f_6$ \\ \hline
			$f_2$ & $f_1f_2f_4$ & $f_6$ & $f_2f_3$ & $f_2f_4$ & $f_2f_5$ & $f_2f_6$ \\ \hline
			$f_3$ & $f_1f_3$ & $f_2f_3$ & $f_5f_6$ & $f_3f_4$ & $f_3f_5$ & $f_3f_6$ \\ \hline
			$f_4$ & $f_1f_4f_6$ & $f_2f_4f_6$ & $f_3f_4$ & $f_6$ & $f_4f_5$ & $f_4f_6$ \\ \hline
			$f_5$ & $f_1f_5$ & $f_2f_5$ & $f_3f_5$ & $f_4f_5$ & $1$ & $f_5f_6$ \\ \hline
			$f_6$ & $f_1f_6$ & $f_2f_6$ & $f_3f_6$ & $f_4f_6$ & $f_5f_6$ & $1$ \\ \hline
		\end{tabular}
	\end{center}
\end{table}

\begin{table}[H]
	\begin{center}
		\caption{Candidate subgroups and the resulting twisted groups $G^\omega$}
		\begin{tabular}{|c|c|c|c|c|}
			Subgroup & Structure of \(N\) & \(G\)-inv & \(G^\omega\cong G\) & Number of \(G^\omega\)  \\ \hline
			
			$[ f_5, f_6 ]$ & $C_2\times C_2$ & $\bigcirc$ & $\bigcirc$ & 1 \\ \hline 
			$[ f_3f_6, f_4 ]$ & $C_4\times C_4$ & $\times$ & - & - \\ \hline
		\end{tabular}
	\end{center}
\end{table}
\subsubsection{\(Q_{16} : C_4\)}No group isocategorical to this group exists.
\begin{table}[H]
	\begin{center}
		\caption{Group table for generators}
		\begin{tabular}{c||c|c|c|c|c|c|}
			& $f_1$ & $f_2$ & $f_3$ & $f_4$ & $f_5$ & $f_6$ \\ \hline\hline
			$f_1$ & $f_5$ & $f_1f_2$ & $f_1f_3$ & $f_1f_4$ & $f_1f_5$ & $f_1f_6$ \\ \hline
			$f_2$ & $f_1f_2f_4$ & $f_6$ & $f_2f_3$ & $f_2f_4$ & $f_2f_5$ & $f_2f_6$ \\ \hline
			$f_3$ & $f_1f_3f_6$ & $f_2f_3$ & $f_5$ & $f_3f_4$ & $f_3f_5$ & $f_3f_6$ \\ \hline
			$f_4$ & $f_1f_4f_6$ & $f_2f_4f_6$ & $f_3f_4$ & $f_6$ & $f_4f_5$ & $f_4f_6$ \\ \hline
			$f_5$ & $f_1f_5$ & $f_2f_5$ & $f_3f_5$ & $f_4f_5$ & $1$ & $f_5f_6$ \\ \hline
			$f_6$ & $f_1f_6$ & $f_2f_6$ & $f_3f_6$ & $f_4f_6$ & $f_5f_6$ & $1$ \\ \hline
		\end{tabular}
	\end{center}
\end{table}

\begin{table}[H]
	\begin{center}
		\caption{Candidate subgroups and the resulting twisted groups $G^\omega$}
		\begin{tabular}{|c|c|c|c|c|}
			Subgroup & Structure of \(N\) & \(G\)-inv & \(G^\omega\cong G\) & Number of \(G^\omega\)  \\ \hline
			
			$[ f_5, f_6 ]$ & $C_2\times C_2$ & $\bigcirc$ & $\bigcirc$ & 2 \\ \hline 
			$[ f_3, f_4 ]$ & $C_4\times C_4$ & $\times$ & - & - \\ \hline
		\end{tabular}
	\end{center}
\end{table}
\subsubsection{\(C_4 : Q_{16}\)}
\begin{table}[H]
	\begin{center}
		\caption{Group table for generators}
		\begin{tabular}{c||c|c|c|c|c|c|}
			& $f_1$ & $f_2$ & $f_3$ & $f_4$ & $f_5$ & $f_6$ \\ \hline\hline
			$f_1$ & $f_6$ & $f_1f_2$ & $f_1f_3$ & $f_1f_4$ & $f_1f_5$ & $f_1f_6$ \\ \hline
			$f_2$ & $f_1f_2f_4$ & $f_4$ & $f_2f_3$ & $f_2f_4$ & $f_2f_5$ & $f_2f_6$ \\ \hline
			$f_3$ & $f_1f_3f_5$ & $f_2f_3$ & $f_6$ & $f_3f_4$ & $f_3f_5$ & $f_3f_6$ \\ \hline
			$f_4$ & $f_1f_4$ & $f_2f_4$ & $f_3f_4$ & $1$ & $f_4f_5$ & $f_4f_6$ \\ \hline
			$f_5$ & $f_1f_5f_6$ & $f_2f_5$ & $f_3f_5f_6$ & $f_4f_5$ & $f_6$ & $f_5f_6$ \\ \hline
			$f_6$ & $f_1f_6$ & $f_2f_6$ & $f_3f_6$ & $f_4f_6$ & $f_5f_6$ & $1$ \\ \hline
		\end{tabular}
	\end{center}
\end{table}

\begin{table}[H]
	\begin{center}
		\caption{Candidate subgroups and the resulting twisted groups $G^\omega$}
		\begin{tabular}{|c|c|c|c|c|}
			Subgroup & Structure of \(N\) & \(G\)-inv & \(G^\omega\cong G\) & Number of \(G^\omega\)  \\ \hline
			
			$[ f_4, f_6 ]$ & $C_2\times C_2$ & $\bigcirc$ & $\bigcirc$ & 3 \\ \hline 
			$[ f_2, f_5 ]$ & $C_4\times C_4$ & $\times$ & - & - \\ \hline
		\end{tabular}
	\end{center}
\end{table}
\subsubsection{\(Q_8 : Q_8\)}No group isocategorical to this group exists.
\begin{table}[H]
	\begin{center}
		\caption{Group table for generators}
		\begin{tabular}{c||c|c|c|c|c|c|}
			& $f_1$ & $f_2$ & $f_3$ & $f_4$ & $f_5$ & $f_6$ \\ \hline\hline
			$f_1$ & $f_4$ & $f_1f_2$ & $f_1f_3$ & $f_1f_4$ & $f_1f_5$ & $f_1f_6$ \\ \hline
			$f_2$ & $f_1f_2f_4$ & $f_4$ & $f_2f_3$ & $f_2f_4$ & $f_2f_5$ & $f_2f_6$ \\ \hline
			$f_3$ & $f_1f_3f_5$ & $f_2f_3$ & $f_6$ & $f_3f_4$ & $f_3f_5$ & $f_3f_6$ \\ \hline
			$f_4$ & $f_1f_4$ & $f_2f_4$ & $f_3f_4$ & $1$ & $f_4f_5$ & $f_4f_6$ \\ \hline
			$f_5$ & $f_1f_5f_6$ & $f_2f_5$ & $f_3f_5f_6$ & $f_4f_5$ & $f_6$ & $f_5f_6$ \\ \hline
			$f_6$ & $f_1f_6$ & $f_2f_6$ & $f_3f_6$ & $f_4f_6$ & $f_5f_6$ & $1$ \\ \hline
		\end{tabular}
	\end{center}
\end{table}

\begin{table}[H]
	\begin{center}
		\caption{Candidate subgroups and the resulting twisted groups $G^\omega$}
		\begin{tabular}{|c|c|c|c|c|}
			Subgroup & Structure of \(N\) & \(G\)-inv & \(G^\omega\cong G\) & Number of \(G^\omega\)  \\ \hline
			
			$[ f_4, f_6 ]$ & $C_2\times C_2$ & $\bigcirc$ & $\bigcirc$ & 4 \\ \hline 
			$[ f_2, f_5 ]$ & $C_4\times C_4$ & $\times$ & - & - \\ \hline
		\end{tabular}
	\end{center}
\end{table}
\subsubsection{\(Q_8 : Q_8\)}No group isocategorical to this group exists.
\begin{table}[H]
	\begin{center}
		\caption{Group table for generators}
		\begin{tabular}{c||c|c|c|c|c|c|}
			& $f_1$ & $f_2$ & $f_3$ & $f_4$ & $f_5$ & $f_6$ \\ \hline\hline
			$f_1$ & $f_4f_6$ & $f_1f_2$ & $f_1f_3$ & $f_1f_4$ & $f_1f_5$ & $f_1f_6$ \\ \hline
			$f_2$ & $f_1f_2f_4$ & $f_4$ & $f_2f_3$ & $f_2f_4$ & $f_2f_5$ & $f_2f_6$ \\ \hline
			$f_3$ & $f_1f_3f_5$ & $f_2f_3$ & $f_6$ & $f_3f_4$ & $f_3f_5$ & $f_3f_6$ \\ \hline
			$f_4$ & $f_1f_4$ & $f_2f_4$ & $f_3f_4$ & $1$ & $f_4f_5$ & $f_4f_6$ \\ \hline
			$f_5$ & $f_1f_5f_6$ & $f_2f_5$ & $f_3f_5f_6$ & $f_4f_5$ & $f_6$ & $f_5f_6$ \\ \hline
			$f_6$ & $f_1f_6$ & $f_2f_6$ & $f_3f_6$ & $f_4f_6$ & $f_5f_6$ & $1$ \\ \hline
		\end{tabular}
	\end{center}
\end{table}

\begin{table}[H]
	\begin{center}
		\caption{Candidate subgroups and the resulting twisted groups $G^\omega$}
		\begin{tabular}{|c|c|c|c|c|}
			Subgroup & Structure of \(N\) & \(G\)-inv & \(G^\omega\cong G\) & Number of \(G^\omega\)  \\ \hline
			
			$[ f_4, f_6 ]$ & $C_2\times C_2$ & $\bigcirc$ & $\bigcirc$ & 5 \\ \hline 
			$[ f_2, f_5 ]$ & $C_4\times C_4$ & $\times$ & - & - \\ \hline
		\end{tabular}
	\end{center}
\end{table}
\subsubsection{\((C_2 \times C_2) . (C_2 \times D_8) = (C_4 \times C_2) . (C_2 \times C_2 \times C_2)\)}No group isocategorical to this group exists.
\begin{table}[H]
	\begin{center}
		\caption{Group table for generators}
		\begin{tabular}{c||c|c|c|c|c|c|}
			& $f_1$ & $f_2$ & $f_3$ & $f_4$ & $f_5$ & $f_6$ \\ \hline\hline
			$f_1$ & $f_4$ & $f_1f_2$ & $f_1f_3$ & $f_1f_4$ & $f_1f_5$ & $f_1f_6$ \\ \hline
			$f_2$ & $f_1f_2f_4$ & $f_4f_6$ & $f_2f_3$ & $f_2f_4$ & $f_2f_5$ & $f_2f_6$ \\ \hline
			$f_3$ & $f_1f_3f_5$ & $f_2f_3$ & $f_6$ & $f_3f_4$ & $f_3f_5$ & $f_3f_6$ \\ \hline
			$f_4$ & $f_1f_4$ & $f_2f_4$ & $f_3f_4$ & $1$ & $f_4f_5$ & $f_4f_6$ \\ \hline
			$f_5$ & $f_1f_5f_6$ & $f_2f_5$ & $f_3f_5f_6$ & $f_4f_5$ & $f_6$ & $f_5f_6$ \\ \hline
			$f_6$ & $f_1f_6$ & $f_2f_6$ & $f_3f_6$ & $f_4f_6$ & $f_5f_6$ & $1$ \\ \hline
		\end{tabular}
	\end{center}
\end{table}

\begin{table}[H]
	\begin{center}
		\caption{Candidate subgroups and the resulting twisted groups $G^\omega$}
		\begin{tabular}{|c|c|c|c|c|}
			Subgroup & Structure of \(N\) & \(G\)-inv & \(G^\omega\cong G\) & Number of \(G^\omega\)  \\ \hline
			
			$[ f_4, f_6 ]$ & $C_2\times C_2$ & $\bigcirc$ & $\bigcirc$ & 6 \\ \hline 
			$[ f_2, f_5 ]$ & $C_4\times C_4$ & $\times$ & - & - \\ \hline
		\end{tabular}
	\end{center}
\end{table}
\subsubsection{\((C_2 \times C_2) . (C_2 \times D_8) = (C_4 \times C_2) . (C_2 \times C_2 \times C_2)\)}No group isocategorical to this group exists.
\begin{table}[H]
	\begin{center}
		\caption{Group table for generators}
		\begin{tabular}{c||c|c|c|c|c|c|}
			& $f_1$ & $f_2$ & $f_3$ & $f_4$ & $f_5$ & $f_6$ \\ \hline\hline
			$f_1$ & $f_6$ & $f_1f_2$ & $f_1f_3$ & $f_1f_4$ & $f_1f_5$ & $f_1f_6$ \\ \hline
			$f_2$ & $f_1f_2f_4$ & $f_5$ & $f_2f_3$ & $f_2f_4$ & $f_2f_5$ & $f_2f_6$ \\ \hline
			$f_3$ & $f_1f_3f_5$ & $f_2f_3$ & $f_4$ & $f_3f_4$ & $f_3f_5$ & $f_3f_6$ \\ \hline
			$f_4$ & $f_1f_4f_6$ & $f_2f_4$ & $f_3f_4$ & $f_6$ & $f_4f_5$ & $f_4f_6$ \\ \hline
			$f_5$ & $f_1f_5f_6$ & $f_2f_5$ & $f_3f_5$ & $f_4f_5$ & $f_6$ & $f_5f_6$ \\ \hline
			$f_6$ & $f_1f_6$ & $f_2f_6$ & $f_3f_6$ & $f_4f_6$ & $f_5f_6$ & $1$ \\ \hline
		\end{tabular}
	\end{center}
\end{table}

\begin{table}[H]
	\begin{center}
		\caption{Candidate subgroups and the resulting twisted groups $G^\omega$}
		\begin{tabular}{|c|c|c|c|c|}
			Subgroup & Structure of \(N\) & \(G\)-inv & \(G^\omega\cong G\) & Number of \(G^\omega\)  \\ \hline
			
			$[ f_4f_5, f_6 ]$ & $C_2\times C_2$ & $\bigcirc$ & $\bigcirc$ & 7 \\ \hline 
			$[ f_2f_3, f_5 ]$ & $C_4\times C_4$ & $\bigcirc$ & $\bigcirc$ & 7 \\ \hline
		\end{tabular}
	\end{center}
\end{table}
\subsubsection{\((C_2 \times C_2) . (C_2 \times D_8) = (C_4 \times C_2) . (C_2 \times C_2 \times C_2)\)}No group isocategorical to this group exists.
\begin{table}[H]
	\begin{center}
		\caption{Group table for generators}
		\begin{tabular}{c||c|c|c|c|c|c|}
			& $f_1$ & $f_2$ & $f_3$ & $f_4$ & $f_5$ & $f_6$ \\ \hline\hline
			$f_1$ & $f_6$ & $f_1f_2$ & $f_1f_3$ & $f_1f_4$ & $f_1f_5$ & $f_1f_6$ \\ \hline
			$f_2$ & $f_1f_2f_4$ & $f_5f_6$ & $f_2f_3$ & $f_2f_4$ & $f_2f_5$ & $f_2f_6$ \\ \hline
			$f_3$ & $f_1f_3f_5$ & $f_2f_3f_6$ & $f_4$ & $f_3f_4$ & $f_3f_5$ & $f_3f_6$ \\ \hline
			$f_4$ & $f_1f_4f_6$ & $f_2f_4$ & $f_3f_4$ & $f_6$ & $f_4f_5$ & $f_4f_6$ \\ \hline
			$f_5$ & $f_1f_5f_6$ & $f_2f_5$ & $f_3f_5$ & $f_4f_5$ & $f_6$ & $f_5f_6$ \\ \hline
			$f_6$ & $f_1f_6$ & $f_2f_6$ & $f_3f_6$ & $f_4f_6$ & $f_5f_6$ & $1$ \\ \hline
		\end{tabular}
	\end{center}
\end{table}

\begin{table}[H]
	\begin{center}
		\caption{Candidate subgroups and the resulting twisted groups $G^\omega$}
		\begin{tabular}{|c|c|c|c|c|}
			Subgroup & Structure of \(N\) & \(G\)-inv & \(G^\omega\cong G\) & Number of \(G^\omega\)  \\ \hline
			
			$[ f_4f_5, f_6 ]$ & $C_2\times C_2$ & $\bigcirc$ & $\bigcirc$ & 8 \\ \hline 
			$[ f_2f_3, f_5 ]$ & $C_4\times C_4$ & $\bigcirc$ & $\bigcirc$ & 8 \\ \hline
		\end{tabular}
	\end{center}
\end{table}
\subsubsection{\(C4 : Q16\)}No group isocategorical to this group exists.
\begin{table}[H]
	\begin{center}
		\caption{Group table for generators}
		\begin{tabular}{c||c|c|c|c|c|c|}
			& $f_1$ & $f_2$ & $f_3$ & $f_4$ & $f_5$ & $f_6$ \\ \hline\hline
			$f_1$ & $f_6$ & $f_1f_2$ & $f_1f_3$ & $f_1f_4$ & $f_1f_5$ & $f_1f_6$ \\ \hline
			$f_2$ & $f_1f_2f_4$ & $f_4f_6$ & $f_2f_3$ & $f_2f_4$ & $f_2f_5$ & $f_2f_6$ \\ \hline
			$f_3$ & $f_1f_3f_5$ & $f_2f_3$ & $f_5$ & $f_3f_4$ & $f_3f_5$ & $f_3f_6$ \\ \hline
			$f_4$ & $f_1f_4f_6$ & $f_2f_4$ & $f_3f_4$ & $f_6$ & $f_4f_5$ & $f_4f_6$ \\ \hline
			$f_5$ & $f_1f_5$ & $f_2f_5$ & $f_3f_5$ & $f_4f_5$ & $1$ & $f_5f_6$ \\ \hline
			$f_6$ & $f_1f_6$ & $f_2f_6$ & $f_3f_6$ & $f_4f_6$ & $f_5f_6$ & $1$ \\ \hline
		\end{tabular}
	\end{center}
\end{table}

\begin{table}[H]
	\begin{center}
		\caption{Candidate subgroups and the resulting twisted groups $G^\omega$}
		\begin{tabular}{|c|c|c|c|c|}
			Subgroup & Structure of \(N\) & \(G\)-inv & \(G^\omega\cong G\) & Number of \(G^\omega\)  \\ \hline
			
			$[ f_5, f_6 ]$ & $C_2\times C_2$ & $\bigcirc$ & $\bigcirc$ & 9 \\ \hline 
			$[ f_3, f_4 ]$ & $C_4\times C_4$ & $\bigcirc$ & $\bigcirc$ & 9 \\ \hline
		\end{tabular}
	\end{center}
\end{table}
\subsubsection{\(C_8 : Q_8\)}No group isocategorical to this group exists.
\begin{table}[H]
	\begin{center}
		\caption{Group table for generators}
		\begin{tabular}{c||c|c|c|c|c|c|}
			& $f_1$ & $f_2$ & $f_3$ & $f_4$ & $f_5$ & $f_6$ \\ \hline\hline
			$f_1$ & $f_4$ & $f_1f_2$ & $f_1f_3$ & $f_1f_4$ & $f_1f_5$ & $f_1f_6$ \\ \hline
			$f_2$ & $f_1f_2f_4$ & $f_4$ & $f_2f_3$ & $f_2f_4$ & $f_2f_5$ & $f_2f_6$ \\ \hline
			$f_3$ & $f_1f_3f_5$ & $f_2f_3$ & $f_5$ & $f_3f_4$ & $f_3f_5$ & $f_3f_6$ \\ \hline
			$f_4$ & $f_1f_4$ & $f_2f_4$ & $f_3f_4$ & $1$ & $f_4f_5$ & $f_4f_6$ \\ \hline
			$f_5$ & $f_1f_5f_6$ & $f_2f_5$ & $f_3f_5$ & $f_4f_5$ & $f_6$ & $f_5f_6$ \\ \hline
			$f_6$ & $f_1f_6$ & $f_2f_6$ & $f_3f_6$ & $f_4f_6$ & $f_5f_6$ & $1$ \\ \hline
		\end{tabular}
	\end{center}
\end{table}

\begin{table}[H]
	\begin{center}
		\caption{Candidate subgroups and the resulting twisted groups $G^\omega$}
		\begin{tabular}{|c|c|c|c|c|}
			Subgroup & Structure of \(N\) & \(G\)-inv & \(G^\omega\cong G\) & Number of \(G^\omega\)  \\ \hline
			
			$[ f_4, f_6 ]$ & $C_2\times C_2$ & $\bigcirc$ & $\bigcirc$ & 10 \\ \hline 
			$[ f_2, f_5 ]$ & $C_4\times C_4$ & $\bigcirc$ & $\bigcirc$ & 10 \\ \hline 
		\end{tabular}
	\end{center}
\end{table}
\subsubsection{\((C_2 \times C_2) . (C_2 \times D_8) = (C_4 \times C_2) . (C_2 \times C_2 \times C_2)\)}No group isocategorical to this group exists.
\begin{table}[H]
	\begin{center}
		\caption{Group table for generators}
		\begin{tabular}{c||c|c|c|c|c|c|}
			& $f_1$ & $f_2$ & $f_3$ & $f_4$ & $f_5$ & $f_6$ \\ \hline\hline
			$f_1$ & $f_4$ & $f_1f_2$ & $f_1f_3$ & $f_1f_4$ & $f_1f_5$ & $f_1f_6$ \\ \hline
			$f_2$ & $f_1f_2f_4$ & $f_4f_6$ & $f_2f_3$ & $f_2f_4$ & $f_2f_5$ & $f_2f_6$ \\ \hline
			$f_3$ & $f_1f_3f_5$ & $f_2f_3$ & $f_5$ & $f_3f_4$ & $f_3f_5$ & $f_3f_6$ \\ \hline
			$f_4$ & $f_1f_4$ & $f_2f_4$ & $f_3f_4$ & $1$ & $f_4f_5$ & $f_4f_6$ \\ \hline
			$f_5$ & $f_1f_5f_6$ & $f_2f_5$ & $f_3f_5$ & $f_4f_5$ & $f_6$ & $f_5f_6$ \\ \hline
			$f_6$ & $f_1f_6$ & $f_2f_6$ & $f_3f_6$ & $f_4f_6$ & $f_5f_6$ & $1$ \\ \hline
		\end{tabular}
	\end{center}
\end{table}

\begin{table}[H]
	\begin{center}
		\caption{Candidate subgroups and the resulting twisted groups $G^\omega$}
		\begin{tabular}{|c|c|c|c|c|}
			Subgroup & Structure of \(N\) & \(G\)-inv & \(G^\omega\cong G\) & Number of \(G^\omega\)  \\ \hline
			
			$[ f_4, f_6 ]$ & $C_2\times C_2$ & $\bigcirc$ & $\bigcirc$ & 11 \\ \hline 
			$[ f_2, f_5 ]$ & $C_4\times C_4$ & $\bigcirc$ & $\bigcirc$ & 11 \\ \hline
		\end{tabular}
	\end{center}
\end{table}
\subsubsection{\(C_8 : Q_8\)}No group isocategorical to this group exists.
\begin{table}[H]
	\begin{center}
		\caption{Group table for generators}
		\begin{tabular}{c||c|c|c|c|c|c|}
			& $f_1$ & $f_2$ & $f_3$ & $f_4$ & $f_5$ & $f_6$ \\ \hline\hline
			$f_1$ & $f_4$ & $f_1f_2$ & $f_1f_3$ & $f_1f_4$ & $f_1f_5$ & $f_1f_6$ \\ \hline
			$f_2$ & $f_1f_2f_4$ & $f_4$ & $f_2f_3$ & $f_2f_4$ & $f_2f_5$ & $f_2f_6$ \\ \hline
			$f_3$ & $f_1f_3f_5$ & $f_2f_3$ & $f_5f_6$ & $f_3f_4$ & $f_3f_5$ & $f_3f_6$ \\ \hline
			$f_4$ & $f_1f_4$ & $f_2f_4$ & $f_3f_4$ & $1$ & $f_4f_5$ & $f_4f_6$ \\ \hline
			$f_5$ & $f_1f_5f_6$ & $f_2f_5$ & $f_3f_5$ & $f_4f_5$ & $f_6$ & $f_5f_6$ \\ \hline
			$f_6$ & $f_1f_6$ & $f_2f_6$ & $f_3f_6$ & $f_4f_6$ & $f_5f_6$ & $1$ \\ \hline
		\end{tabular}
	\end{center}
\end{table}

\begin{table}[H]
	\begin{center}
		\caption{Candidate subgroups and the resulting twisted groups $G^\omega$}
		\begin{tabular}{|c|c|c|c|c|}
			Subgroup & Structure of \(N\) & \(G\)-inv & \(G^\omega\cong G\) & Number of \(G^\omega\)  \\ \hline
			
			$[ f_4, f_6 ]$ & $C_2\times C_2$ & $\bigcirc$ & $\bigcirc$ & 12 \\ \hline 
			$[ f_2, f_5 ]$ & $C_4\times C_4$ & $\bigcirc$ & $\bigcirc$ & 12 \\ \hline
		\end{tabular}
	\end{center}
\end{table}
\subsubsection{\(C_8 : Q_8\)}No group isocategorical to this group exists.
\begin{table}[H]
	\begin{center}
		\caption{Group table for generators}
		\begin{tabular}{c||c|c|c|c|c|c|}
			& $f_1$ & $f_2$ & $f_3$ & $f_4$ & $f_5$ & $f_6$ \\ \hline\hline
			$f_1$ & $f_4$ & $f_1f_2$ & $f_1f_3$ & $f_1f_4$ & $f_1f_5$ & $f_1f_6$ \\ \hline
			$f_2$ & $f_1f_2f_4$ & $f_4$ & $f_2f_3$ & $f_2f_4$ & $f_2f_5$ & $f_2f_6$ \\ \hline
			$f_3$ & $f_1f_3f_5$ & $f_2f_3f_6$ & $f_5$ & $f_3f_4$ & $f_3f_5$ & $f_3f_6$ \\ \hline
			$f_4$ & $f_1f_4$ & $f_2f_4$ & $f_3f_4$ & $1$ & $f_4f_5$ & $f_4f_6$ \\ \hline
			$f_5$ & $f_1f_5f_6$ & $f_2f_5$ & $f_3f_5$ & $f_4f_5$ & $f_6$ & $f_5f_6$ \\ \hline
			$f_6$ & $f_1f_6$ & $f_2f_6$ & $f_3f_6$ & $f_4f_6$ & $f_5f_6$ & $1$ \\ \hline
		\end{tabular}
	\end{center}
\end{table}

\begin{table}[H]
	\begin{center}
		\caption{Candidate subgroups and the resulting twisted groups $G^\omega$}
		\begin{tabular}{|c|c|c|c|c|}
			Subgroup & Structure of \(N\) & \(G\)-inv & \(G^\omega\cong G\) & Number of \(G^\omega\)  \\ \hline
			
			$[ f_4, f_6 ]$ & $C_2\times C_2$ & $\bigcirc$ & $\bigcirc$ & 13 \\ \hline 
			$[ f_2, f_5 ]$ & $C_4\times C_4$ & $\bigcirc$ & $\bigcirc$ & 13 \\ \hline 
		\end{tabular}
	\end{center}
\end{table}

\subsection{Order list \([1,3,60,0,0,0,0]\)}
\subsubsection{\(Q_8 : Q_8\)}No group isocategorical to this group exists.
\begin{table}[H]
	\begin{center}
		\caption{Group table for generators}
		\begin{tabular}{c||c|c|c|c|c|c|}
			& $f_1$ & $f_2$ & $f_3$ & $f_4$ & $f_5$ & $f_6$ \\ \hline\hline
			$f_1$ & $f_6$ & $f_1f_2$ & $f_1f_3$ & $f_1f_4$ & $f_1f_5$ & $f_1f_6$ \\ \hline
			$f_2$ & $f_1f_2f_5$ & $f_5$ & $f_2f_3$ & $f_2f_4$ & $f_2f_5$ & $f_2f_6$ \\ \hline
			$f_3$ & $f_1f_3f_6$ & $f_2f_3$ & $f_6$ & $f_3f_4$ & $f_3f_5$ & $f_3f_6$ \\ \hline
			$f_4$ & $f_1f_4$ & $f_2f_4f_5$ & $f_3f_4$ & $f_5$ & $f_4f_5$ & $f_4f_6$ \\ \hline
			$f_5$ & $f_1f_5$ & $f_2f_5$ & $f_3f_5$ & $f_4f_5$ & $1$ & $f_5f_6$ \\ \hline
			$f_6$ & $f_1f_6$ & $f_2f_6$ & $f_3f_6$ & $f_4f_6$ & $f_5f_6$ & $1$ \\ \hline
		\end{tabular}
	\end{center}
\end{table}

\begin{table}[H]
	\begin{center}
		\caption{Candidate subgroups and the resulting twisted groups $G^\omega$}
		\begin{tabular}{|c|c|c|c|c|}
			Subgroup & Structure of \(N\) & \(G\)-inv & \(G^\omega\cong G\) & Number of \(G^\omega\)  \\ \hline
			
			$[ f_5, f_6 ]$ & $C_2\times C_2$ & $\bigcirc$ & $\bigcirc$ & 1 \\ \hline 
			$[ f_1f_3f_4f_6, f_2f_5f_6 ]$ & $C_4\times C_4$ & $\times$ & - & - \\ \hline
			$[ f_2f_4, f_3 ]$ & $C_4\times C_4$ & $\times$ & - & - \\ \hline
			$[ f_1f_4f_6, f_1 ]$ & $C_4\times C_4$ & $\times$ & - & - \\ \hline
			$[ f_1f_4f_6, f_2f_4f_5f_6 ]$ & $C_4\times C_4$ & $\times$ & - & - \\ \hline
			$[ f_1f_3f_4f_6, f_1f_3 ]$ & $C_4\times C_4$ & $\times$ & - & - \\ \hline
			$[ f_2, f_3 ]$ & $C_4\times C_4$ & $\times$ & - & - \\ \hline
			$[ f_3f_4f_6, f_3 ]$ & $C_4\times C_4$ & $\times$ & - & - \\ \hline
			$[ f_1f_4f_6, f_2f_5f_6 ]$ & $C_4\times C_4$ & $\times$ & - & - \\ \hline
			$[ f_1f_3f_4f_6, f_2f_4f_5f_6 ]$ & $C_4\times C_4$ & $\times$ & - & - \\ \hline
		\end{tabular}
	\end{center}
\end{table}
\subsubsection{\(Q_8\times Q_8\)}No group isocategorical to this group exists.
\begin{table}[H]
	\begin{center}
		\caption{Group table for generators}
		\begin{tabular}{c||c|c|c|c|c|c|}
			& $f_1$ & $f_2$ & $f_3$ & $f_4$ & $f_5$ & $f_6$ \\ \hline\hline
			$f_1$ & $f_5f_6$ & $f_1f_2$ & $f_1f_3$ & $f_1f_4$ & $f_1f_5$ & $f_1f_6$ \\ \hline
			$f_2$ & $f_1f_2f_5$ & $f_5$ & $f_2f_3$ & $f_2f_4$ & $f_2f_5$ & $f_2f_6$ \\ \hline
			$f_3$ & $f_1f_3f_6$ & $f_2f_3$ & $f_6$ & $f_3f_4$ & $f_3f_5$ & $f_3f_6$ \\ \hline
			$f_4$ & $f_1f_4$ & $f_2f_4f_5$ & $f_3f_4$ & $f_5$ & $f_4f_5$ & $f_4f_6$ \\ \hline
			$f_5$ & $f_1f_5$ & $f_2f_5$ & $f_3f_5$ & $f_4f_5$ & $1$ & $f_5f_6$ \\ \hline
			$f_6$ & $f_1f_6$ & $f_2f_6$ & $f_3f_6$ & $f_4f_6$ & $f_5f_6$ & $1$ \\ \hline
		\end{tabular}
	\end{center}
\end{table}

\begin{table}[H]
	\begin{center}
		\caption{Candidate subgroups and the resulting twisted groups $G^\omega$}
		\begin{tabular}{|c|c|c|c|c|}
			Subgroup & Structure of \(N\) & \(G\)-inv & \(G^\omega\cong G\) & Number of \(G^\omega\)  \\ \hline
			
			$[ f_5, f_6 ]$ & $C_2\times C_2$ & $\bigcirc$ & $\bigcirc$ & 2 \\ \hline 
			$[ f_1f_2f_3f_6, f_1f_3f_4 ]$ & $C_4\times C_4$ & $\times$ & - & - \\ \hline
			$[ f_2f_4, f_3 ]$ & $C_4\times C_4$ & $\times$ & - & - \\ \hline
			$[ f_1f_6, f_4f_5f_6 ]$ & $C_4\times C_4$ & $\times$ & - & - \\ \hline
			$[ f_1f_2f_6, f_1f_4 ]$ & $C_4\times C_4$ & $\times$ & - & - \\ \hline
			$[ f_1f_3f_6, f_4f_5f_6 ]$ & $C_4\times C_4$ & $\times$ & - & - \\ \hline
			$[ f_3f_4f_6, f_3 ]$ & $C_4\times C_4$ & $\times$ & - & - \\ \hline
			$[ f_1f_2f_4f_5f_6, f_1f_4 ]$ & $C_4\times C_4$ & $\times$ & - & - \\ \hline
			$[ f_1f_2f_3f_4f_5f_6, f_1f_3f_4 ]$ & $C_4\times C_4$ & $\times$ & - & - \\ \hline
			$[ f_2, f_3 ]$ & $C_4\times C_4$ & $\times$ & - & - \\ \hline
		\end{tabular}
	\end{center}
\end{table}

\subsubsection{\((C_2 \times C_2 ) . (C_2\times C_2\times C_2\times C_2)\)}No group isocategorical to this group exists.
\begin{table}[H]
	\begin{center}
		\caption{Group table for generators}
		\begin{tabular}{c||c|c|c|c|c|c|}
			& $f_1$ & $f_2$ & $f_3$ & $f_4$ & $f_5$ & $f_6$ \\ \hline\hline
			$f_1$ & $f_5$ & $f_1f_2$ & $f_1f_3$ & $f_1f_4$ & $f_1f_5$ & $f_1f_6$ \\ \hline
			$f_2$ & $f_1f_2f_5$ & $f_6$ & $f_2f_3$ & $f_2f_4$ & $f_2f_5$ & $f_2f_6$ \\ \hline
			$f_3$ & $f_1f_3f_6$ & $f_2f_3f_5f_6$ & $f_6$ & $f_3f_4$ & $f_3f_5$ & $f_3f_6$ \\ \hline
			$f_4$ & $f_1f_4f_5f_6$ & $f_2f_4f_5$ & $f_3f_4$ & $f_5$ & $f_4f_5$ & $f_4f_6$ \\ \hline
			$f_5$ & $f_1f_5$ & $f_2f_5$ & $f_3f_5$ & $f_4f_5$ & $1$ & $f_5f_6$ \\ \hline
			$f_6$ & $f_1f_6$ & $f_2f_6$ & $f_3f_6$ & $f_4f_6$ & $f_5f_6$ & $1$ \\ \hline
		\end{tabular}
	\end{center}
\end{table}

\begin{table}[H]
	\begin{center}
		\caption{Candidate subgroups and the resulting twisted groups $G^\omega$}
		\begin{tabular}{|c|c|c|c|c|}
			Subgroup & Structure of \(N\) & \(G\)-inv & \(G^\omega\cong G\) & Number of \(G^\omega\)  \\ \hline
			
			$[ f_5, f_6 ]$ & $C_2\times C_2$ & $\bigcirc$ & $\bigcirc$ & 3 \\ \hline 
			$[ f_1f_2f_3f_4f_5f_6, f_1f_5f_6 ]$ & $C_4\times C_4$ & $\bigcirc$ & $\bigcirc$ & 3 \\ \hline 
			$[ f_3f_4f_6, f_3 ]$ & $C_4\times C_4$ & $\bigcirc$ & $\bigcirc$ & 3 \\ \hline 
			$[ f_1f_4f_6, f_2 ]$ & $C_4\times C_4$ & $\bigcirc$ & $\bigcirc$ & 3 \\ \hline 
			$[ f_1f_2f_6, f_1f_3f_5f_6 ]$ & $C_4\times C_4$ & $\bigcirc$ & $\bigcirc$ & 3 \\ \hline 
			$[ f_1f_3f_4f_6, f_2f_4 ]$ & $C_4\times C_4$ & $\bigcirc$ & $\bigcirc$ & 3 \\ \hline 
			
		\end{tabular}
	\end{center}
\end{table}

\subsection{Order list \([1,7,8,16,32,0,0]\)}
\subsubsection{\((C_{16}\times C_2):C_2\)}No group isocategorical to this group exists.
\begin{table}[H]
	\begin{center}
		\caption{Group table for generators}
		\begin{tabular}{c||c|c|c|c|c|c|}
			& $f_1$ & $f_2$ & $f_3$ & $f_4$ & $f_5$ & $f_6$ \\ \hline\hline
			$f_1$ & $f_4$ & $f_1f_2$ & $f_1f_3$ & $f_1f_4$ & $f_1f_5$ & $f_1f_6$ \\ \hline
			$f_2$ & $f_1f_2f_3$ & $1$ & $f_2f_3$ & $f_2f_4$ & $f_2f_5$ & $f_2f_6$ \\ \hline
			$f_3$ & $f_1f_3$ & $f_2f_3$ & $1$ & $f_3f_4$ & $f_3f_5$ & $f_3f_6$ \\ \hline
			$f_4$ & $f_1f_4$ & $f_2f_4$ & $f_3f_4$ & $f_5$ & $f_4f_5$ & $f_4f_6$ \\ \hline
			$f_5$ & $f_1f_5$ & $f_2f_5$ & $f_3f_5$ & $f_4f_5$ & $f_6$ & $f_5f_6$ \\ \hline
			$f_6$ & $f_1f_6$ & $f_2f_6$ & $f_3f_6$ & $f_4f_6$ & $f_5f_6$ & $1$ \\ \hline
		\end{tabular}
	\end{center}
\end{table}

\begin{table}[H]
	\begin{center}
		\caption{Candidate subgroups and the resulting twisted groups $G^\omega$}
		\begin{tabular}{|c|c|c|c|c|}
			Subgroup & Structure of \(N\) & \(G\)-inv & \(G^\omega\cong G\) & Number of \(G^\omega\)  \\ \hline
			
			$[ f_3, f_6 ]$ & $C_2\times C_2$ & $\bigcirc$ & $\bigcirc$ & 1 \\ \hline 
			$[ f_2f_6, f_3 ]$ & $C_2\times C_2$ & $\bigcirc$ & $\bigcirc$ & 1 \\ \hline 
			$[ f_2, f_3 ]$ & $C_2\times C_2$ & $\bigcirc$ & $\bigcirc$ & 1 \\ \hline 
			
		\end{tabular}
	\end{center}
\end{table}

\subsubsection{\((C_{16}:C_2):C_2\)}No group isocategorical to this group exists.
\begin{table}[H]
	\begin{center}
		\caption{Group table for generators}
		\begin{tabular}{c||c|c|c|c|c|c|}
			& $f_1$ & $f_2$ & $f_3$ & $f_4$ & $f_5$ & $f_6$ \\ \hline\hline
			$f_1$ & $f_4$ & $f_1f_2$ & $f_1f_3$ & $f_1f_4$ & $f_1f_5$ & $f_1f_6$ \\ \hline
			$f_2$ & $f_1f_2f_3$ & $1$ & $f_2f_3$ & $f_2f_4$ & $f_2f_5$ & $f_2f_6$ \\ \hline
			$f_3$ & $f_1f_3f_6$ & $f_2f_3$ & $1$ & $f_3f_4$ & $f_3f_5$ & $f_3f_6$ \\ \hline
			$f_4$ & $f_1f_4$ & $f_2f_4f_6$ & $f_3f_4$ & $f_5$ & $f_4f_5$ & $f_4f_6$ \\ \hline
			$f_5$ & $f_1f_5$ & $f_2f_5$ & $f_3f_5$ & $f_4f_5$ & $f_6$ & $f_5f_6$ \\ \hline
			$f_6$ & $f_1f_6$ & $f_2f_6$ & $f_3f_6$ & $f_4f_6$ & $f_5f_6$ & $1$ \\ \hline
		\end{tabular}
	\end{center}
\end{table}

\begin{table}[H]
	\begin{center}
		\caption{Candidate subgroups and the resulting twisted groups $G^\omega$}
		\begin{tabular}{|c|c|c|c|c|}
			Subgroup & Structure of \(N\) & \(G\)-inv & \(G^\omega\cong G\) & Number of \(G^\omega\)  \\ \hline
			
			$[ f_3, f_6 ]$ & $C_2\times C_2$ & $\bigcirc$ & $\bigcirc$ & 2 \\ \hline 
			
		\end{tabular}
	\end{center}
\end{table}

\subsubsection{\((C_{16}:C_2):C_2\)}No group isocategorical to this group exists.
\begin{table}[H]
	\begin{center}
		\caption{Group table for generators}
		\begin{tabular}{c||c|c|c|c|c|c|}
			& $f_1$ & $f_2$ & $f_3$ & $f_4$ & $f_5$ & $f_6$ \\ \hline\hline
			$f_1$ & $f_4$ & $f_1f_2$ & $f_1f_3$ & $f_1f_4$ & $f_1f_5$ & $f_1f_6$ \\ \hline
			$f_2$ & $f_1f_2f_3$ & $1$ & $f_2f_3$ & $f_2f_4$ & $f_2f_5$ & $f_2f_6$ \\ \hline
			$f_3$ & $f_1f_3$ & $f_2f_3f_6$ & $f_6$ & $f_3f_4$ & $f_3f_5$ & $f_3f_6$ \\ \hline
			$f_4$ & $f_1f_4$ & $f_2f_4f_6$ & $f_3f_4$ & $f_5$ & $f_4f_5$ & $f_4f_6$ \\ \hline
			$f_5$ & $f_1f_5$ & $f_2f_5$ & $f_3f_5$ & $f_4f_5$ & $f_6$ & $f_5f_6$ \\ \hline
			$f_6$ & $f_1f_6$ & $f_2f_6$ & $f_3f_6$ & $f_4f_6$ & $f_5f_6$ & $1$ \\ \hline
		\end{tabular}
	\end{center}
\end{table}

\begin{table}[H]
	\begin{center}
		\caption{Candidate subgroups and the resulting twisted groups $G^\omega$}
		\begin{tabular}{|c|c|c|c|c|}
			Subgroup & Structure of \(N\) & \(G\)-inv & \(G^\omega\cong G\) & Number of \(G^\omega\)  \\ \hline
			
			$[ f_3f_5, f_6 ]$ & $C_2\times C_2$ & $\bigcirc$ & $\bigcirc$ & 3 \\ \hline 
			
		\end{tabular}
	\end{center}
\end{table}

\subsubsection{\(C_{16}\times C_2\times C_2\)}No group isocategorical to this group exists.
As the group is abelian, no non-isomorphic groups exist that are isocategorical to it.

\subsubsection{\(C_2\times (C_{16}:C_2)\)}No group isocategorical to this group exists.
\begin{table}[H]
	\begin{center}
		\caption{Group table for generators}
		\begin{tabular}{c||c|c|c|c|c|c|}
			& $f_1$ & $f_2$ & $f_3$ & $f_4$ & $f_5$ & $f_6$ \\ \hline\hline
			$f_1$ & $f_4$ & $f_1f_2$ & $f_1f_3$ & $f_1f_4$ & $f_1f_5$ & $f_1f_6$ \\ \hline
			$f_2$ & $f_1f_2f_6$ & $1$ & $f_2f_3$ & $f_2f_4$ & $f_2f_5$ & $f_2f_6$ \\ \hline
			$f_3$ & $f_1f_3$ & $f_2f_3$ & $1$ & $f_3f_4$ & $f_3f_5$ & $f_3f_6$ \\ \hline
			$f_4$ & $f_1f_4$ & $f_2f_4$ & $f_3f_4$ & $f_5$ & $f_4f_5$ & $f_4f_6$ \\ \hline
			$f_5$ & $f_1f_5$ & $f_2f_5$ & $f_3f_5$ & $f_4f_5$ & $f_6$ & $f_5f_6$ \\ \hline
			$f_6$ & $f_1f_6$ & $f_2f_6$ & $f_3f_6$ & $f_4f_6$ & $f_5f_6$ & $1$ \\ \hline
		\end{tabular}
	\end{center}
\end{table}

\begin{table}[H]
	\begin{center}
		\caption{Candidate subgroups and the resulting twisted groups $G^\omega$}
		\begin{tabular}{|c|c|c|c|c|}
			Subgroup & Structure of \(N\) & \(G\)-inv & \(G^\omega\cong G\) & Number of \(G^\omega\)  \\ \hline
			
			$[ f_3, f_6 ]$ & $C_2\times C_2$ & $\bigcirc$ & $\bigcirc$ & 5 \\ \hline 
			$[ f_2, f_6 ]$ & $C_2\times C_2$ & $\bigcirc$ & $\bigcirc$ & 5 \\ \hline 
			
		\end{tabular}
	\end{center}
\end{table}

\subsubsection{\((C_{16}\times C_2):C_2\)}No group isocategorical to this group exists.
\begin{table}[H]
	\begin{center}
		\caption{Group table for generators}
		\begin{tabular}{c||c|c|c|c|c|c|}
			& $f_1$ & $f_2$ & $f_3$ & $f_4$ & $f_5$ & $f_6$ \\ \hline\hline
			$f_1$ & $f_4$ & $f_1f_2$ & $f_1f_3$ & $f_1f_4$ & $f_1f_5$ & $f_1f_6$ \\ \hline
			$f_2$ & $f_1f_2$ & $1$ & $f_2f_3$ & $f_2f_4$ & $f_2f_5$ & $f_2f_6$ \\ \hline
			$f_3$ & $f_1f_3$ & $f_2f_3f_6$ & $1$ & $f_3f_4$ & $f_3f_5$ & $f_3f_6$ \\ \hline
			$f_4$ & $f_1f_4$ & $f_2f_4$ & $f_3f_4$ & $f_5$ & $f_4f_5$ & $f_4f_6$ \\ \hline
			$f_5$ & $f_1f_5$ & $f_2f_5$ & $f_3f_5$ & $f_4f_5$ & $f_6$ & $f_5f_6$ \\ \hline
			$f_6$ & $f_1f_6$ & $f_2f_6$ & $f_3f_6$ & $f_4f_6$ & $f_5f_6$ & $1$ \\ \hline
		\end{tabular}
	\end{center}
\end{table}

\begin{table}[H]
	\begin{center}
		\caption{Candidate subgroups and the resulting twisted groups $G^\omega$}
		\begin{tabular}{|c|c|c|c|c|}
			Subgroup & Structure of \(N\) & \(G\)-inv & \(G^\omega\cong G\) & Number of \(G^\omega\)  \\ \hline
			
			$[ f_3, f_6 ]$ & $C_2\times C_2$ & $\bigcirc$ & $\bigcirc$ & 6 \\ \hline 
			$[ f_2, f_6 ]$ & $C_2\times C_2$ & $\bigcirc$ & $\bigcirc$ & 6 \\ \hline 
			$[ f_2f_3f_5, f_6 ]$ & $C_2\times C_2$ & $\bigcirc$ & $\bigcirc$ & 6 \\ \hline 
			
		\end{tabular}
	\end{center}
\end{table}

\subsection{Order list \([1,7,8,48,0,0,0]\)}
\subsubsection{\(C_4 . (C_4 \times C_4)\)}No group isocategorical to this group exists.
\begin{table}[H]
	\begin{center}
		\caption{Group table for generators}
		\begin{tabular}{c||c|c|c|c|c|c|}
			& $f_1$ & $f_2$ & $f_3$ & $f_4$ & $f_5$ & $f_6$ \\ \hline\hline
			$f_1$ & $f_4$ & $f_1f_2$ & $f_1f_3$ & $f_1f_4$ & $f_1f_5$ & $f_1f_6$ \\ \hline
			$f_2$ & $f_1f_2f_3$ & $f_5$ & $f_2f_3$ & $f_2f_4$ & $f_2f_5$ & $f_2f_6$ \\ \hline
			$f_3$ & $f_1f_3$ & $f_2f_3$ & $f_6$ & $f_3f_4$ & $f_3f_5$ & $f_3f_6$ \\ \hline
			$f_4$ & $f_1f_4$ & $f_2f_4f_6$ & $f_3f_4$ & $f_6$ & $f_4f_5$ & $f_4f_6$ \\ \hline
			$f_5$ & $f_1f_5f_6$ & $f_2f_5$ & $f_3f_5$ & $f_4f_5$ & $f_6$ & $f_5f_6$ \\ \hline
			$f_6$ & $f_1f_6$ & $f_2f_6$ & $f_3f_6$ & $f_4f_6$ & $f_5f_6$ & $1$ \\ \hline
		\end{tabular}
	\end{center}
\end{table}

\begin{table}[H]
	\begin{center}
		\caption{Candidate subgroups and the resulting twisted groups $G^\omega$}
		\begin{tabular}{|c|c|c|c|c|}
			Subgroup & Structure of \(N\) & \(G\)-inv & \(G^\omega\cong G\) & Number of \(G^\omega\)  \\ \hline
			
			$[ f_4f_5, f_6 ]$ & $C_2\times C_2$ & $\bigcirc$ & $\bigcirc$ & 1 \\ \hline 
			$[ f_3f_5, f_6 ]$ & $C_2\times C_2$ & $\bigcirc$ & $\bigcirc$ & 1 \\ \hline 
			$[ f_3f_4, f_6 ]$ & $C_2\times C_2$ & $\bigcirc$ & $\bigcirc$ & 1 \\ \hline 
			
		\end{tabular}
	\end{center}
\end{table}

\subsubsection{\((C_4 \times C_2) . D_8 = C_4 . (C_4 \times C_4)\)}No group isocategorical to this group exists.
\begin{table}[H]
	\begin{center}
		\caption{Group table for generators}
		\begin{tabular}{c||c|c|c|c|c|c|}
			& $f_1$ & $f_2$ & $f_3$ & $f_4$ & $f_5$ & $f_6$ \\ \hline\hline
			$f_1$ & $f_4$ & $f_1f_2$ & $f_1f_3$ & $f_1f_4$ & $f_1f_5$ & $f_1f_6$ \\ \hline
			$f_2$ & $f_1f_2f_3$ & $f_5$ & $f_2f_3$ & $f_2f_4$ & $f_2f_5$ & $f_2f_6$ \\ \hline
			$f_3$ & $f_1f_3f_6$ & $f_2f_3$ & $f_6$ & $f_3f_4$ & $f_3f_5$ & $f_3f_6$ \\ \hline
			$f_4$ & $f_1f_4$ & $f_2f_4$ & $f_3f_4$ & $f_6$ & $f_4f_5$ & $f_4f_6$ \\ \hline
			$f_5$ & $f_1f_5f_6$ & $f_2f_5$ & $f_3f_5$ & $f_4f_5$ & $f_6$ & $f_5f_6$ \\ \hline
			$f_6$ & $f_1f_6$ & $f_2f_6$ & $f_3f_6$ & $f_4f_6$ & $f_5f_6$ & $1$ \\ \hline
		\end{tabular}
	\end{center}
\end{table}

\begin{table}[H]
	\begin{center}
		\caption{Candidate subgroups and the resulting twisted groups $G^\omega$}
		\begin{tabular}{|c|c|c|c|c|}
			Subgroup & Structure of \(N\) & \(G\)-inv & \(G^\omega\cong G\) & Number of \(G^\omega\)  \\ \hline
			
			$[ f_4f_5, f_6 ]$ & $C_2\times C_2$ & $\bigcirc$ & $\bigcirc$ & 2 \\ \hline 
			$[ f_3f_5, f_6 ]$ & $C_2\times C_2$ & $\bigcirc$ & $\bigcirc$ & 2 \\ \hline 
			$[ f_3f_4, f_6 ]$ & $C_2\times C_2$ & $\bigcirc$ & $\bigcirc$ & 2 \\ \hline   			
			
		\end{tabular}
	\end{center}
\end{table}

\subsubsection{\(C_2 \times (C_4 . D_8 = C_4 . (C_4 \times C_2))\)}No group isocategorical to this group exists.
\begin{table}[H]
	\begin{center}
		\caption{Group table for generators}
		\begin{tabular}{c||c|c|c|c|c|c|}
			& $f_1$ & $f_2$ & $f_3$ & $f_4$ & $f_5$ & $f_6$ \\ \hline\hline
			$f_1$ & $f_5$ & $f_1f_2$ & $f_1f_3$ & $f_1f_4$ & $f_1f_5$ & $f_1f_6$ \\ \hline
			$f_2$ & $f_1f_2f_4$ & $f_4$ & $f_2f_3$ & $f_2f_4$ & $f_2f_5$ & $f_2f_6$ \\ \hline
			$f_3$ & $f_1f_3$ & $f_2f_3$ & $1$ & $f_3f_4$ & $f_3f_5$ & $f_3f_6$ \\ \hline
			$f_4$ & $f_1f_4f_6$ & $f_2f_4$ & $f_3f_4$ & $f_6$ & $f_4f_5$ & $f_4f_6$ \\ \hline
			$f_5$ & $f_1f_5$ & $f_2f_5$ & $f_3f_5$ & $f_4f_5$ & $f_6$ & $f_5f_6$ \\ \hline
			$f_6$ & $f_1f_6$ & $f_2f_6$ & $f_3f_6$ & $f_4f_6$ & $f_5f_6$ & $1$ \\ \hline
		\end{tabular}
	\end{center}
\end{table}

\begin{table}[H]
	\begin{center}
		\caption{Candidate subgroups and the resulting twisted groups $G^\omega$}
		\begin{tabular}{|c|c|c|c|c|}
			Subgroup & Structure of \(N\) & \(G\)-inv & \(G^\omega\cong G\) & Number of \(G^\omega\)  \\ \hline
			
			$[ f_4f_5, f_6 ]$ & $C_2\times C_2$ & $\bigcirc$ & $\bigcirc$ & 3 \\ \hline 
			$[ f_3, f_6 ]$ & $C_2\times C_2$ & $\bigcirc$ & $\bigcirc$ & 3 \\ \hline 
			$[ f_3f_4f_5, f_6 ]$ & $C_2\times C_2$ & $\bigcirc$ & $\bigcirc$ & 3 \\ \hline 
			
		\end{tabular}
	\end{center}
\end{table}

\subsubsection{\((C_4 . D_8 = C_4 . (C_4 \times C_2)) : C_2\)}No group isocategorical to this group exists.
\begin{table}[H]
	\begin{center}
		\caption{Group table for generators}
		\begin{tabular}{c||c|c|c|c|c|c|}
			& $f_1$ & $f_2$ & $f_3$ & $f_4$ & $f_5$ & $f_6$ \\ \hline\hline
			$f_1$ & $f_5$ & $f_1f_2$ & $f_1f_3$ & $f_1f_4$ & $f_1f_5$ & $f_1f_6$ \\ \hline
			$f_2$ & $f_1f_2f_4$ & $f_4$ & $f_2f_3$ & $f_2f_4$ & $f_2f_5$ & $f_2f_6$ \\ \hline
			$f_3$ & $f_1f_3$ & $f_2f_3f_6$ & $1$ & $f_3f_4$ & $f_3f_5$ & $f_3f_6$ \\ \hline
			$f_4$ & $f_1f_4f_6$ & $f_2f_4$ & $f_3f_4$ & $f_6$ & $f_4f_5$ & $f_4f_6$ \\ \hline
			$f_5$ & $f_1f_5$ & $f_2f_5$ & $f_3f_5$ & $f_4f_5$ & $f_6$ & $f_5f_6$ \\ \hline
			$f_6$ & $f_1f_6$ & $f_2f_6$ & $f_3f_6$ & $f_4f_6$ & $f_5f_6$ & $1$ \\ \hline
		\end{tabular}
	\end{center}
\end{table}

\begin{table}[H]
	\begin{center}
		\caption{Candidate subgroups and the resulting twisted groups $G^\omega$}
		\begin{tabular}{|c|c|c|c|c|}
			Subgroup & Structure of \(N\) & \(G\)-inv & \(G^\omega\cong G\) & Number of \(G^\omega\)  \\ \hline
			
			$[ f_4f_5, f_6 ]$ & $C_2\times C_2$ & $\bigcirc$ & $\bigcirc$ & 4 \\ \hline 
			$[ f_3, f_6 ]$ & $C_2\times C_2$ & $\bigcirc$ & $\bigcirc$ & 4 \\ \hline 
			$[ f_3f_4f_5, f_6 ]$ & $C_2\times C_2$ & $\bigcirc$ & $\bigcirc$ & 4 \\ \hline
		\end{tabular}
	\end{center}
\end{table}

\subsection{Order list \([1,7,24,32,0,0,0]\)}
\subsubsection{\((C_4 \times C_2) : C_8\)}No group isocategorical to this group exists.
\begin{table}[H]
	\begin{center}
		\caption{Group table for generators}
		\begin{tabular}{c||c|c|c|c|c|c|}
			& $f_1$ & $f_2$ & $f_3$ & $f_4$ & $f_5$ & $f_6$ \\ \hline\hline
			$f_1$ & $f_4$ & $f_1f_2$ & $f_1f_3$ & $f_1f_4$ & $f_1f_5$ & $f_1f_6$ \\ \hline
			$f_2$ & $f_1f_2f_3$ & $f_5$ & $f_2f_3$ & $f_2f_4$ & $f_2f_5$ & $f_2f_6$ \\ \hline
			$f_3$ & $f_1f_3f_5$ & $f_2f_3$ & $1$ & $f_3f_4$ & $f_3f_5$ & $f_3f_6$ \\ \hline
			$f_4$ & $f_1f_4$ & $f_2f_4f_5$ & $f_3f_4$ & $f_6$ & $f_4f_5$ & $f_4f_6$ \\ \hline
			$f_5$ & $f_1f_5$ & $f_2f_5$ & $f_3f_5$ & $f_4f_5$ & $1$ & $f_5f_6$ \\ \hline
			$f_6$ & $f_1f_6$ & $f_2f_6$ & $f_3f_6$ & $f_4f_6$ & $f_5f_6$ & $1$ \\ \hline
		\end{tabular}
	\end{center}
\end{table}

\begin{table}[H]
	\begin{center}
		\caption{Candidate subgroups and the resulting twisted groups $G^\omega$}
		\begin{tabular}{|c|c|c|c|c|}
			Subgroup & Structure of \(N\) & \(G\)-inv & \(G^\omega\cong G\) & Number of \(G^\omega\)  \\ \hline
			
			$[ f_5, f_6 ]$ & $C_2\times C_2$ & $\bigcirc$ & $\bigcirc$ & 1 \\ \hline 
			$[ f_3, f_5 ]$ & $C_2\times C_2$ & $\bigcirc$ & $\bigcirc$ & 1 \\ \hline 
			$[ f_3f_6, f_5 ]$ & $C_2\times C_2$ & $\bigcirc$ & $\bigcirc$ & 1 \\ \hline 
		\end{tabular}
	\end{center}
\end{table}

\subsubsection{\((C_8 \times C_2) : C_4\)}No group isocategorical to this group exists.
\begin{table}[H]
	\begin{center}
		\caption{Group table for generators}
		%
	\end{center}
\end{table}

\subsubsection{\((C_8 : C_2) : C_4\)}No group isocategorical to this group exists.
\begin{table}[H]
	\begin{center}
		\caption{Group table for generators}
		%
	\end{center}
\end{table}

\subsubsection{\((C_8 \times C_2) : C_4\)}No group isocategorical to this group exists.
\begin{table}[H]
	\begin{center}
		\caption{Group table for generators}
		%
	\end{center}
\end{table}

\subsubsection{\(C_8 \times C_4 \times C_2\)}No group isocategorical to this group exists.
\begin{table}[H]
	\begin{center}
		\caption{Group table for generators}
		%
	\end{center}
\end{table}

As the group is abelian, no non-isomorphic groups exist that are isocategorical to it.

\subsubsection{\(C_2 \times (C_8 : C_4)\)}No group isocategorical to this group exists.
\begin{table}[H]
	\begin{center}
		\caption{Group table for generators}
		%
	\end{center}
\end{table}

\subsubsection{\(C_4 \times (C_8 : C_2)\)}No group isocategorical to this group exists.
\begin{table}[H]
	\begin{center}
		\caption{Group table for generators}
		%
	\end{center}
\end{table}

\subsubsection{\((C_8 \times C_4) : C_2\)}No group isocategorical to this group exists.
\begin{table}[H]
	\begin{center}
		\caption{Group table for generators}
		%
	\end{center}
\end{table}

\subsubsection{\(C_2 \times (C_2 . ((C_4 \times C_2) : C_2) = (C_2 \times C_2) . (C_4 \times C_2))\)}No group isocategorical to this group exists.
\begin{table}[H]
	\begin{center}
		\caption{Group table for generators}
		%
	\end{center}
\end{table}

\subsubsection{\(C_2 \times (C_4 : C_8)\)}No group isocategorical to this group exists.
\begin{table}[H]
	\begin{center}
		\caption{Group table for generators}
		%
	\end{center}
\end{table}

\subsubsection{\((C_4 : C_8) : C_2\)}No group isocategorical to this group exists.
\begin{table}[H]
	\begin{center}
		\caption{Group table for generators}
		%
	\end{center}
\end{table}

\subsubsection{\((C_4 : C_8) : C_2\)}No group isocategorical to this group exists.
\begin{table}[H]
	\begin{center}
		\caption{Group table for generators}
		%
	\end{center}
\end{table}

\subsubsection{\((C_8 \times C_4) : C_2\)}No group isocategorical to this group exists.
\begin{table}[H]
	\begin{center}
		\caption{Group table for generators}
		%
	\end{center}
\end{table}

\subsubsection{\((C_8 : C_4) : C_2\)}No group isocategorical to this group exists.
\begin{table}[H]
	\begin{center}
		\caption{Group table for generators}
		%
	\end{center}
\end{table}

\subsubsection{\((C_8 \times C_4) : C_2\)}No group isocategorical to this group exists.
\begin{table}[H]
	\begin{center}
		\caption{Group table for generators}
		%
	\end{center}
\end{table}

\subsubsection{\((C_2 . ((C_4 \times C_2) : C_2) = (C_2 \times C_2) . (C_4 \times C_2)) : C_2\)}No group isocategorical to this group exists.
\begin{table}[H]
	\begin{center}
		\caption{Group table for generators}
		%
	\end{center}
\end{table}

\subsection{Order list \([1,7,40,16,0,0,0]\)}
\subsubsection{\((C_4 : C_4) : C_4\)}No group isocategorical to this group exists.
\begin{table}[H]
	\begin{center}
		\caption{Group table for generators}
		%
	\end{center}
\end{table}

\subsubsection{\((C_4 \times C_4) : C_4\)}No group isocategorical to this group exists.
\begin{table}[H]
	\begin{center}
		\caption{Group table for generators}
		%
	\end{center}
\end{table}

\subsubsection{\((C_4 : C_4) : C_4\)}No group isocategorical to this group exists.
\begin{table}[H]
	\begin{center}
		\caption{Group table for generators}
		%
	\end{center}
\end{table}

\subsubsection{\((C_4 : C_4) : C_4\)}No group isocategorical to this group exists.
\begin{table}[H]
	\begin{center}
		\caption{Group table for generators}
		%
	\end{center}
\end{table}

\subsubsection{\(C_2 \times (Q_8 : C_4)\)}No group isocategorical to this group exists.
\begin{table}[H]
	\begin{center}
		\caption{Group table for generators}
		%
	\end{center}
\end{table}

\subsubsection{\((Q_8 : C_4) : C_2\)}No group isocategorical to this group exists.
\begin{table}[H]
	\begin{center}
		\caption{Group table for generators}
		%
	\end{center}
\end{table}

\subsubsection{\(C2 \times (C_8 : C_4)\)}No group isocategorical to this group exists.
\begin{table}[H]
	\begin{center}
		\caption{Group table for generators}
		%
	\end{center}
\end{table}

\subsubsection{\(C_2 \times (C_8 : C_4)\)}No group isocategorical to this group exists.
\begin{table}[H]
	\begin{center}
		\caption{Group table for generators}
		%
	\end{center}
\end{table}

\subsubsection{\((C_8 : C_4) : C_2\)}No group isocategorical to this group exists.
\begin{table}[H]
	\begin{center}
		\caption{Group table for generators}
		%
	\end{center}
\end{table}

\subsubsection{\((C_8 : C_4) : C_2\)}No group isocategorical to this group exists.
\begin{table}[H]
	\begin{center}
		\caption{Group table for generators}
		%
	\end{center}
\end{table}

\subsubsection{\((C_2 \times Q_{16}) : C_2\)}No group isocategorical to this group exists.
\begin{table}[H]
	\begin{center}
		\caption{Group table for generators}
		%
	\end{center}
\end{table}

\subsubsection{\((Q_8 : C_4) : C_2\)}No group isocategorical to this group exists.
\begin{table}[H]
	\begin{center}
		\caption{Group table for generators}
		%
	\end{center}
\end{table}

\subsubsection{\((Q_8 : C_4) : C_2\)}No group isocategorical to this group exists.
\begin{table}[H]
	\begin{center}
		\caption{Group table for generators}
		%
	\end{center}
\end{table}

\subsubsection{\((Q_8 : C_4) : C_2\)}No group isocategorical to this group exists.
\begin{table}[H]
	\begin{center}
		\caption{Group table for generators}
		%
	\end{center}
\end{table}

\subsubsection{\((Q_8 : C_4) : C_2\)}No group isocategorical to this group exists.
\begin{table}[H]
	\begin{center}
		\caption{Group table for generators}
		%
	\end{center}
\end{table}

\subsubsection{\((Q_8 : C_4) : C_2\)}No group isocategorical to this group exists.
\begin{table}[H]
	\begin{center}
		\caption{Group table for generators}
		%
	\end{center}
\end{table}

\subsubsection{\(C_2 \times C_2 \times Q_{16}\)}No group isocategorical to this group exists.
\begin{table}[H]
	\begin{center}
		\caption{Group table for generators}
		%
	\end{center}
\end{table}

\subsection{Order list \([1,7,56,0,0,0,0]\)}
\subsubsection{\(C_4 \times C_4 \times C_4\)}
As the group is abelian, no non-isomorphic groups exist that are isocategorical to it.
\subsubsection{\((C_4 \times C_4) : C_4\)}No group isocategorical to this group exists.
\begin{table}[H]
	\begin{center}
		\caption{Group table for generators}
		%
	\end{center}
\end{table}
\subsubsection{\(C_4 \times (C_4 : C_4)\)}No group isocategorical to this group exists.
\begin{table}[H]
	\begin{center}
		\caption{Group table for generators}
		%
	\end{center}
\end{table}
\subsubsection{\((C_4 \times C_4) : C_4\)}No group isocategorical to this group exists.
\begin{table}[H]
	\begin{center}
		\caption{Group table for generators}
		%
	\end{center}
\end{table}
\subsubsection{\((C_4 \times C_4) : C_4\)}No group isocategorical to this group exists.
\begin{table}[H]
	\begin{center}
		\caption{Group table for generators}
		%
	\end{center}
\end{table}
\subsubsection{\((C_4 \times C_4) : C_4\)}No group isocategorical to this group exists.
\begin{table}[H]
	\begin{center}
		\caption{Group table for generators}
		%
	\end{center}
\end{table}
\subsubsection{\((C_4 : C_4) : C_4\)}No group isocategorical to this group exists.
\begin{table}[H]
	\begin{center}
		\caption{Group table for generators}
		%
	\end{center}
\end{table}
\subsubsection{\((C_4 : C_4) : C_4\)}No group isocategorical to this group exists.
\begin{table}[H]
	\begin{center}
		\caption{Group table for generators}
		%
	\end{center}
\end{table}
\subsubsection{\((C_2 \times Q_8) : C_4\)}No group isocategorical to this group exists.
\begin{table}[H]
	\begin{center}
		\caption{Group table for generators}
		%
	\end{center}
\end{table}
\subsubsection{\((C_4 \times C_2) : Q_8\)}No group isocategorical to this group exists.
\begin{table}[H]
	\begin{center}
		\caption{Group table for generators}
		%
	\end{center}
\end{table}
\subsubsection{\((C_2 \times C_2 \times C_2) . (C_2 \times C_2 \times C_2)\)}No group isocategorical to this group exists.
\begin{table}[H]
	\begin{center}
		\caption{Group table for generators}
		%
	\end{center}
\end{table}
\subsubsection{\((C_2 \times C_2 \times C_2) . (C_2 \times C_2 \times C_2)\)}No group isocategorical to this group exists.
\begin{table}[H]
	\begin{center}
		\caption{Group table for generators}
		%
	\end{center}
\end{table}
\subsubsection{\((C_2 \times C_2 \times C_2) . (C_2 \times C_2 \times C_2)\)}No group isocategorical to this group exists.
\begin{table}[H]
	\begin{center}
		\caption{Group table for generators}
		%
	\end{center}
\end{table}
\subsubsection{\(C_2 \times C_4 \times Q_8\)}No group isocategorical to this group exists.
\begin{table}[H]
	\begin{center}
		\caption{Group table for generators}
		%
	\end{center}
\end{table}
\subsubsection{\((C_4 \times Q_8) : C_2\)}No group isocategorical to this group exists.
\begin{table}[H]
	\begin{center}
		\caption{Group table for generators}
		%
	\end{center}
\end{table}
\subsubsection{\(C_2 \times ((C_2 \times C_2) . (C_2 \times C_2 \times C_2))\)}No group isocategorical to this group exists.
\begin{table}[H]
	\begin{center}
		\caption{Group table for generators}
		%
	\end{center}
\end{table}
\subsubsection{\(C_2 \times (C_4 : Q_8)\)}No group isocategorical to this group exists.
\begin{table}[H]
	\begin{center}
		\caption{Group table for generators}
		%
	\end{center}
\end{table}
\subsubsection{\((C_4 \times Q_8) : C_2\)}No group isocategorical to this group exists.
\begin{table}[H]
	\begin{center}
		\caption{Group table for generators}
		%
	\end{center}
\end{table}
\subsubsection{\((C_4 \times Q_8) : C_2\)}No group isocategorical to this group exists.
\begin{table}[H]
	\begin{center}
		\caption{Group table for generators}
		%
	\end{center}
\end{table}
\subsubsection{\((C_4 : Q_8) : C_2\)}No group isocategorical to this group exists.
\begin{table}[H]
	\begin{center}
		\caption{Group table for generators}
		%
	\end{center}
\end{table}

\subsection{Order list \([1,11,12,24,16,0,0]\)}
\subsubsection{\((C_{16}\times C_2):C_2\)}No group isocategorical to this group exists.
Since only one group belongs to this order list, there are no non-isomorphic groups that are monoidally equivalent to it.

\subsection{Order list \([1,11,20,32,0,0,0]\)}
\subsubsection{\((C_8 \times C_4) : C_2\)}No group isocategorical to this group exists.
\begin{table}[H]
	\begin{center}
		\caption{Group table for generators}
		%
	\end{center}
\end{table}
\subsubsection{\((C_8 : C_4) : C_2\)}No group isocategorical to this group exists.
\begin{table}[H]
	\begin{center}
		\caption{Group table for generators}
		%
	\end{center}
\end{table}
\subsubsection{\((C_2 . ((C_4 \times C_2) : C_2) = (C_2 \times C_2) . (C_4 \times C_2)) : C_2\)}No group isocategorical to this group exists.
\begin{table}[H]
	\begin{center}
		\caption{Group table for generators}
		%
	\end{center}
\end{table}
\subsubsection{\(C_8 \times D_8\)}No group isocategorical to this group exists.
\begin{table}[H]
	\begin{center}
		\caption{Group table for generators}
		%
	\end{center}
\end{table}
\subsubsection{\((C_2 \times (C_8 : C_2)) : C_2\)}No group isocategorical to this group exists.
\begin{table}[H]
	\begin{center}
		\caption{Group table for generators}
		%
	\end{center}
\end{table}
\subsubsection{\((C_2 \times (C_8 : C_2)) : C_2\)}No group isocategorical to this group exists.
\begin{table}[H]
	\begin{center}
		\caption{Group table for generators}
		%
	\end{center}
\end{table}
\subsubsection{\(((C_8 \times C_2) : C_2) : C_2\)}No group isocategorical to this group exists.
\begin{table}[H]
	\begin{center}
		\caption{Group table for generators}
		%
	\end{center}
\end{table}
\subsubsection{\(((C_8 \times C_2) : C_2) : C_2\)}No group isocategorical to this group exists.
\begin{table}[H]
	\begin{center}
		\caption{Group table for generators}
		%
	\end{center}
\end{table}

\subsection{Order list \([1,11,28,8,16,0,0]\)}
\subsubsection{\((C_{16} : C_2) : C_2\)}No group isocategorical to this group exists.
\begin{table}[H]
	\begin{center}
		\caption{Group table for generators}
		%
	\end{center}
\end{table}
\subsubsection{\(QD_{32} : C_2\)}No group isocategorical to this group exists.
\begin{table}[H]
	\begin{center}
		\caption{Group table for generators}
		%
	\end{center}
\end{table}

\subsection{Order list \([1,11,36,16,0,0,0]\)}
\subsubsection{\((C_4 \times C_2 \times C_2) : C_4\)}No group isocategorical to this group exists.
\begin{table}[H]
	\begin{center}
		\caption{Group table for generators}
		%
	\end{center}
\end{table}
\subsubsection{\(C_4 \times QD_{16}\)}
\begin{table}[H]
	\begin{center}
		\caption{Group table for generators}
		%
	\end{center}
\end{table}
\subsubsection{\((C_4 \times Q_8) : C_2\)}No group isocategorical to this group exists.
\begin{table}[H]
	\begin{center}
		\caption{Group table for generators}
		%
	\end{center}
\end{table}
\subsubsection{\((C_4 : Q_8) : C_2\)}No group isocategorical to this group exists.
\begin{table}[H]
	\begin{center}
		\caption{Group table for generators}
		%
	\end{center}
\end{table}
\subsubsection{\((Q_8 : C_4) : C_2\)}No group isocategorical to this group exists.
\begin{table}[H]
	\begin{center}
		\caption{Group table for generators}
		%
	\end{center}
\end{table}
\subsubsection{\((C_4 \times Q_8) : C_2\)}No group isocategorical to this group exists.
\begin{table}[H]
	\begin{center}
		\caption{Group table for generators}
		%
	\end{center}
\end{table}
\subsubsection{\((C_8 : C_4) : C_2\)}No group isocategorical to this group exists.
\begin{table}[H]
	\begin{center}
		\caption{Group table for generators}
		%
	\end{center}
\end{table}
\subsubsection{\((C_8 : C_4) : C_2\)}No group isocategorical to this group exists.
\begin{table}[H]
	\begin{center}
		\caption{Group table for generators}
		%
	\end{center}
\end{table}
\subsubsection{\((C_8 : C_4) : C_2\)}No group isocategorical to this group exists.
\begin{table}[H]
	\begin{center}
		\caption{Group table for generators}
		%
	\end{center}
\end{table}
\subsubsection{\((C_8 \times C_4) : C_2\)}No group isocategorical to this group exists.
\begin{table}[H]
	\begin{center}
		\caption{Group table for generators}
		%
	\end{center}
\end{table}
\subsubsection{\((C_8 : C_4) : C_2\)}No group isocategorical to this group exists.
\begin{table}[H]
	\begin{center}
		\caption{Group table for generators}
		%
	\end{center}
\end{table}
\subsubsection{\((C_4 : Q_8) : C_2\)}No group isocategorical to this group exists.
\begin{table}[H]
	\begin{center}
		\caption{Group table for generators}
		%
	\end{center}
\end{table}

\subsection{Order list \([1,11,52,0,0,0,0]\)}
\subsubsection{\((C_2 \times Q_8) : C_4\)}No group isocategorical to this group exists.
\begin{table}[H]
	\begin{center}
		\caption{Group table for generators}
		%
	\end{center}
\end{table}
\subsubsection{\(Q_8 \times D_8\)}No group isocategorical to this group exists.
\begin{table}[H]
	\begin{center}
		\caption{Group table for generators}
		%
	\end{center}
\end{table}
\subsubsection{\(((C_4 \times C_4) : C_2) : C_2\)}No group isocategorical to this group exists.
\begin{table}[H]
	\begin{center}
		\caption{Group table for generators}
		%
	\end{center}
\end{table}
\subsubsection{\((C_2 \times (C_4 : C_4)) : C_2\)}No group isocategorical to this group exists.
\begin{table}[H]
	\begin{center}
		\caption{Group table for generators}
		%
	\end{center}
\end{table}
\subsubsection{\(((C_4 \times C_4) : C_2) : C_2\)}No group isocategorical to this group exists.
\begin{table}[H]
	\begin{center}
		\caption{Group table for generators}
		%
	\end{center}
\end{table}
\subsubsection{\((C_4 : Q_8) : C_2\)}No group isocategorical to this group exists.
\begin{table}[H]
	\begin{center}
		\caption{Group table for generators}
		%
	\end{center}
\end{table}

\subsection{Order list \([1,15,16,32,0,0,0]\)}
\subsubsection{\(((C_8 \times C_2) : C_2) : C_2\)}No group isocategorical to this group exists.
\begin{table}[H]
	\begin{center}
		\caption{Group table for generators}
		%
	\end{center}
\end{table}
\subsubsection{\(C_2 \times ((C_8 \times C_2) : C_2)\)}No group isocategorical to this group exists.
\begin{table}[H]
	\begin{center}
		\caption{Group table for generators}
		%
	\end{center}
\end{table}
\subsubsection{\((C_2 \times (C_8 : C_2)) : C_2\)}No group isocategorical to this group exists.
\begin{table}[H]
	\begin{center}
		\caption{Group table for generators}
		%
	\end{center}
\end{table}
\subsubsection{\((C_8 \times C_2 \times C_2) : C_2\)}No group isocategorical to this group exists.
\begin{table}[H]
	\begin{center}
		\caption{Group table for generators}
		%
	\end{center}
\end{table}
\subsubsection{\((C_2 \times (C_8 : C_2)) : C_2\)}No group isocategorical to this group exists.
\begin{table}[H]
	\begin{center}
		\caption{Group table for generators}
		%
	\end{center}
\end{table}
\subsubsection{\(((C_8 : C_2) : C_2) : C_2\)}No group isocategorical to this group exists.
\begin{table}[H]
	\begin{center}
		\caption{Group table for generators}
		%
	\end{center}
\end{table}
\subsubsection{\(C_8 \times C_2 \times C_2 \times C_2\)}No group isocategorical to this group exists.
As the group is abelian, no non-isomorphic groups exist that are isocategorical to it.
\subsubsection{\(C_2 \times C_2 \times (C_8 : C_2)\)}No group isocategorical to this group exists.
\begin{table}[H]
	\begin{center}
		\caption{Group table for generators}
		%
	\end{center}
\end{table}
\subsubsection{\(C_2 \times ((C_8 \times C_2) : C_2)\)}No group isocategorical to this group exists.
\begin{table}[H]
	\begin{center}
		\caption{Group table for generators}
		%
	\end{center}
\end{table}
\subsubsection{\((C_2 \times (C_8 : C_2)) : C_2\)}No group isocategorical to this group exists.
\begin{table}[H]
	\begin{center}
		\caption{Group table for generators}
		%
	\end{center}
\end{table}

\subsection{Order list \([1,15,32,16,0,0,0]\)}
\subsubsection{\(((C_8 \times C_2) : C_2) : C_2\)}No group isocategorical to this group exists.
\begin{table}[H]
	\begin{center}
		\caption{Group table for generators}
		%
	\end{center}
\end{table}

\subsubsection{\(C_2 \times ((C_2 \times Q_8) : C_2)\)}No group isocategorical to this group exists.
\begin{table}[H]
	\begin{center}
		\caption{Group table for generators}
		%
	\end{center}
\end{table}
\subsubsection{\((C_2 \times (C_4 : C_4)) : C_2\)}No group isocategorical to this group exists.
\begin{table}[H]
	\begin{center}
		\caption{Group table for generators}
		%
	\end{center}
\end{table}
\subsubsection{\(C_2 \times ((C_4 \times C_4) : C_2)\)}No group isocategorical to this group exists.
\begin{table}[H]
	\begin{center}
		\caption{Group table for generators}
		%
	\end{center}
\end{table}
\subsubsection{\((C_2 \times (C_8 : C_2)) : C_2\)}No group isocategorical to this group exists.
\begin{table}[H]
	\begin{center}
		\caption{Group table for generators}
		%
	\end{center}
\end{table}
\subsubsection{\((C_2 \times QD_{16}) : C_2\)}No group isocategorical to this group exists.
\begin{table}[H]
	\begin{center}
		\caption{Group table for generators}
		%
	\end{center}
\end{table}
\subsubsection{\(((C_2 \times Q_8) : C_2) : C_2\)}No group isocategorical to this group exists.
\begin{table}[H]
	\begin{center}
		\caption{Group table for generators}
		%
	\end{center}
\end{table}
\subsubsection{\(((C_8 \times C_2) : C_2) : C_2\)}No group isocategorical to this group exists.
\begin{table}[H]
	\begin{center}
		\caption{Group table for generators}
		%
	\end{center}
\end{table}
\subsubsection{\(((C_8 \times C_2) : C_2) : C_2\)}No group isocategorical to this group exists.
\begin{table}[H]
	\begin{center}
		\caption{Group table for generators}
		%
	\end{center}
\end{table}
\subsubsection{\(((C_8 \times C_2) : C_2) : C_2\)}No group isocategorical to this group exists.
\begin{table}[H]
	\begin{center}
		\caption{Group table for generators}
		%
	\end{center}
\end{table}
\subsubsection{\(((C_8 \times C_2) : C_2) : C_2\)}No group isocategorical to this group exists.
\begin{table}[H]
	\begin{center}
		\caption{Group table for generators}
		%
	\end{center}
\end{table}
\subsubsection{\(((C_8 \times C_2) : C_2) : C_2\)}No group isocategorical to this group exists.
\begin{table}[H]
	\begin{center}
		\caption{Group table for generators}
		%
	\end{center}
\end{table}
\subsubsection{\(C_2 \times ((C_2 \times Q_8) : C_2)\)}No group isocategorical to this group exists.
\begin{table}[H]
	\begin{center}
		\caption{Group table for generators}
		%
	\end{center}
\end{table}
\subsubsection{\((C_2 \times Q_{16}) : C_2\)}No group isocategorical to this group exists.
\begin{table}[H]
	\begin{center}
		\caption{Group table for generators}
		%
	\end{center}
\end{table}

\subsection{Order list \([1,15,48,0,0,0,0]\)}
\subsubsection{\(((C_4 \times C_2) : C_2) : C_4\)}No group isocategorical to this group exists.
\begin{table}[H]
	\begin{center}
		\caption{Group table for generators}
		%
	\end{center}
\end{table}
\subsubsection{\(C_2 \times ((C_4 \times C_2) : C_4)\)}No group isocategorical to this group exists.
\begin{table}[H]
	\begin{center}
		\caption{Group table for generators}
		%
	\end{center}
\end{table}
\subsubsection{\(C_4 \times ((C_4 \times C_2) : C_2)\)}No group isocategorical to this group exists.
\begin{table}[H]
	\begin{center}
		\caption{Group table for generators}
		%
	\end{center}
\end{table}
\subsubsection{\(((C_4 \times C_2) : C_4) : C_2\)}No group isocategorical to this group exists.
\begin{table}[H]
	\begin{center}
		\caption{Group table for generators}
		%
	\end{center}
\end{table}
\subsubsection{\(((C_4 \times C_2) : C_4) : C_2\)}No group isocategorical to this group exists.
\begin{table}[H]
	\begin{center}
		\caption{Group table for generators}
		%
	\end{center}
\end{table}
\subsubsection{\((C_2 \times (C_4 : C_4)) : C_2\)}No group isocategorical to this group exists.
\begin{table}[H]
	\begin{center}
		\caption{Group table for generators}
		%
	\end{center}
\end{table}
\subsubsection{\(((C_4 \times C_2) : C_4) : C_2\)}No group isocategorical to this group exists.
\begin{table}[H]
	\begin{center}
		\caption{Group table for generators}
		%
	\end{center}
\end{table}
\subsubsection{\(((C_4 \times C_2) : C_4) : C_2\)}No group isocategorical to this group exists.
\begin{table}[H]
	\begin{center}
		\caption{Group table for generators}
		%
	\end{center}
\end{table}
\subsubsection{\((C_2 \times (C_4 : C_4)) : C_2\)}No group isocategorical to this group exists.
\begin{table}[H]
	\begin{center}
		\caption{Group table for generators}
		%
	\end{center}
\end{table}
\subsubsection{\((C_2 \times (C_4 : C_4)) : C_2\)}No group isocategorical to this group exists.
\begin{table}[H]
	\begin{center}
		\caption{Group table for generators}
		%
	\end{center}
\end{table}
\subsubsection{\(((C_4 \times C_2) : C_4) : C_2\)}No group isocategorical to this group exists.
\begin{table}[H]
	\begin{center}
		\caption{Group table for generators}
		%
	\end{center}
\end{table}
\subsubsection{\(((C_4 \times C_4) : C_2) : C_2\)}No group isocategorical to this group exists.
\begin{table}[H]
	\begin{center}
		\caption{Group table for generators}
		%
	\end{center}
\end{table}
\subsubsection{\(C_4 \times C_4 \times C_2 \times C_2\)}No group isocategorical to this group exists.
As the group is abelian, no non-isomorphic groups exist that are isocategorical to it.

\subsubsection{\(C_2 \times C_2 \times (C_4 : C_4)\)}No group isocategorical to this group exists.
\begin{table}[H]
	\begin{center}
		\caption{Group table for generators}
		%
	\end{center}
\end{table}
\subsubsection{\(C_2 \times ((C_4 \times C_4) : C_2)\)}No group isocategorical to this group exists.
\begin{table}[H]
	\begin{center}
		\caption{Group table for generators}
		%
	\end{center}
\end{table}
\subsubsection{\(C_4 \times ((C_4 \times C_2) : C_2)\)}No group isocategorical to this group exists.
\begin{table}[H]
	\begin{center}
		\caption{Group table for generators}
		%
	\end{center}
\end{table}
\subsubsection{\(((C_4 \times C_4) : C_2) : C_2\)}No group isocategorical to this group exists.
\begin{table}[H]
	\begin{center}
		\caption{Group table for generators}
		%
	\end{center}
\end{table}
\subsubsection{\(C_2 \times ((C_2 \times Q_8) : C_2)\)}No group isocategorical to this group exists.
\begin{table}[H]
	\begin{center}
		\caption{Group table for generators}
		%
	\end{center}
\end{table}
\subsubsection{\(C_2 \times ((C_4 \times C_4) : C_2)\)}No group isocategorical to this group exists.
\begin{table}[H]
	\begin{center}
		\caption{Group table for generators}
		%
	\end{center}
\end{table}
\subsubsection{\(((C_4 \times C_4) : C_2) : C_2\)}No group isocategorical to this group exists.
\begin{table}[H]
	\begin{center}
		\caption{Group table for generators}
		%
	\end{center}
\end{table}
\subsubsection{\(((C_4 \times C_4) : C_2) : C_2\)}No group isocategorical to this group exists.
\begin{table}[H]
	\begin{center}
		\caption{Group table for generators}
		%
	\end{center}
\end{table}
\subsubsection{\((C_4 \times D_8) : C_2\)}No group isocategorical to this group exists.
\begin{table}[H]
	\begin{center}
		\caption{Group table for generators}
		%
	\end{center}
\end{table}
\subsubsection{\((C_4 \times D_8) : C_2\)}No group isocategorical to this group exists.
\begin{table}[H]
	\begin{center}
		\caption{Group table for generators}
		%
	\end{center}
\end{table}
\subsubsection{\(((C_2 \times Q_8) : C_2) : C_2\)}No group isocategorical to this group exists.
\begin{table}[H]
	\begin{center}
		\caption{Group table for generators}
		%
	\end{center}
\end{table}
\subsubsection{\(C_2 \times C_2 \times C_2 \times Q_8\)}No group isocategorical to this group exists.
\begin{table}[H]
	\begin{center}
		\caption{Group table for generators}
		%
	\end{center}
\end{table}

\subsection{Order list \([1,17,18,4,8,16,0]\)}
\subsubsection{\(QD_{64}\)}
Since only one group belongs to this order list, there are no non-isomorphic groups that are monoidally equivalent to it.

\subsection{Order list \([1,19,4,24,16,0,0]\)}
\subsubsection{\((C_{16}:C_2):C_2\)}No group isocategorical to this group exists.
Since only one group belongs to this order list, there are no non-isomorphic groups that are monoidally equivalent to it.

\subsection{Order list \([1,19,12,32,0,0,0]\)}
\subsubsection{\((C_4:C_8):C_2\)}No group isocategorical to this group exists.
Since only one group belongs to this order list, there are no non-isomorphic groups that are monoidally equivalent to it.

\subsection{Order list \([1,19,20,8,16,0,0]\)}
\subsubsection{\((C_{16} \times C_2) : C_2\)}No group isocategorical to this group exists.
\begin{table}[H]
	\begin{center}
		\caption{Group table for generators}

	\end{center}
\end{table}

\begin{table}[H]
	\begin{center}
		\caption{Candidate subgroups and the resulting twisted groups $G^\omega$}
		%
	\end{center}
\end{table}
\subsubsection{\(C_2 \times QD_{32}\)}
\begin{table}[H]
	\begin{center}
		\caption{Group table for generators}
		%
	\end{center}
\end{table}

\begin{table}[H]
	\begin{center}
		\caption{Candidate subgroups and the resulting twisted groups $G^\omega$}
		%
	\end{center}
\end{table}
\subsubsection{\(((C_8 \times C_2) : C_2) : C_2\)}No group isocategorical to this group exists.
\begin{table}[H]
	\begin{center}
		\caption{Group table for generators}
		%
	\end{center}
\end{table}

\begin{table}[H]
	\begin{center}
		\caption{Candidate subgroups and the resulting twisted groups $G^\omega$}
		%
	\end{center}
\end{table}

\subsection{Order list \([1,19,28,16,0,0,0]\)}
\subsubsection{\(((C_8 : C_2) : C_2) : C_2\)}No group isocategorical to this group exists.
\begin{table}[H]
	\begin{center}
		\caption{Group table for generators}
		%
	\end{center}
\end{table}
\subsubsection{\(C_4 \times D_{16}\)}
\begin{table}[H]
	\begin{center}
		\caption{Group table for generators}
		%
	\end{center}
\end{table}
\subsubsection{\((C_4 \times D_8) : C_2\)}No group isocategorical to this group exists.
\begin{table}[H]
	\begin{center}
		\caption{Group table for generators}
		%
	\end{center}
\end{table}
\subsubsection{\(((C_4 \times C_4) : C_2) : C_2\)}No group isocategorical to this group exists.
\begin{table}[H]
	\begin{center}
		\caption{Group table for generators}
		%
	\end{center}
\end{table}
\subsubsection{\(((C_4 \times C_4) : C_2) : C_2\)}No group isocategorical to this group exists.
\begin{table}[H]
	\begin{center}
		\caption{Group table for generators}
		%
	\end{center}
\end{table}
\subsubsection{\((C_2 \times QD_{16}) : C_2\)}No group isocategorical to this group exists.
\begin{table}[H]
	\begin{center}
		\caption{Group table for generators}
		%
	\end{center}
\end{table}
\subsubsection{\(((C_8 \times C_2) : C_2) : C_2\)}No group isocategorical to this group exists.
\begin{table}[H]
	\begin{center}
		\caption{Group table for generators}
		%
	\end{center}
\end{table}
\subsubsection{\(((C_8 \times C_2) : C_2) : C_2\)}No group isocategorical to this group exists.
\begin{table}[H]
	\begin{center}
		\caption{Group table for generators}
		%
	\end{center}
\end{table}
\subsubsection{\(((C_8 \times C_2) : C_2) : C_2\)}No group isocategorical to this group exists.
\begin{table}[H]
	\begin{center}
		\caption{Group table for generators}
		%
	\end{center}
\end{table}
\subsubsection{\((C_2 \times QD_{16}) : C_2\)}No group isocategorical to this group exists.
\begin{table}[H]
	\begin{center}
		\caption{Group table for generators}
		%
	\end{center}
\end{table}
\subsubsection{\(((C_4 \times C_4) : C_2) : C_2\)}No group isocategorical to this group exists.
\begin{table}[H]
	\begin{center}
		\caption{Group table for generators}
		%
	\end{center}
\end{table}

\subsection{Order list \([1,19,44,0,0,0,0]\)}
\subsubsection{\(((C_2 \times C_2 \times C_2) : C_4) : C_2\)}No group isocategorical to this group exists.
\begin{table}[H]
	\begin{center}
		\caption{Group table for generators}
		%
	\end{center}
\end{table}

\subsubsection{\(((C_4 \times C_2 \times C_2) : C_2) : C_2\)}No group isocategorical to this group exists.
\begin{table}[H]
	\begin{center}
		\caption{Group table for generators}
		%
	\end{center}
\end{table}

\subsubsection{\((C_2 \times ((C_4 \times C_2) : C_2)) : C_2\)}No group isocategorical to this group exists.
\begin{table}[H]
	\begin{center}
		\caption{Group table for generators}
		%
	\end{center}
\end{table}

\subsubsection{\((C_2 \times ((C_4 \times C_2) : C_2)) : C_2\)}No group isocategorical to this group exists.
\begin{table}[H]
	\begin{center}
		\caption{Group table for generators}
		%
	\end{center}
\end{table} 

\subsubsection{\(((C_4 \times C_4) : C_2) : C_2\)}\label{G1}
It is monoidally equivalent to the seventh group in subsubsection \([1,19,44,0,0,0,0]\).
\begin{table}[H]
	\begin{center}
		\caption{Group table for generators}
		%
	\end{center}
\end{table}

\subsubsection{\((C_4 \times C_4) : C_2) : C_2\)}No group isocategorical to this group exists.
\begin{table}[H]
	\begin{center}
		\caption{Group table for generators}
		%
	\end{center}
\end{table}

\subsubsection{\(((C_4 \times C_4) : C_2) : C_2\)}\label{G2}
It is monoidally equivalent to the fifth group in subsubsection \([1,19,44,0,0,0,0]\).
\begin{table}[H]
	\begin{center}
		\caption{Group table for generators}
		%
	\end{center}
\end{table}

\subsubsection{\((C_4 \times D_8) : C_2\)}No group isocategorical to this group exists.
\begin{table}[H]
	\begin{center}
		\caption{Group table for generators}
		%
	\end{center}
\end{table}

\subsubsection{\(((C_4 \times C_2 \times C_2) : C_2) : C_2\)}No group isocategorical to this group exists.
\begin{table}[H]
	\begin{center}
		\caption{Group table for generators}
		%
	\end{center}
\end{table}

\subsection{Order list \([1,23,8,32,0,0,0]\)}
\subsubsection{\(C_2 \times ((C_8 : C_2) : C_2)\)}No group isocategorical to this group exists.
\begin{table}[H]
	\begin{center}
		\caption{Group table for generators}
		%
	\end{center}
\end{table}
\subsubsection{\((C_8 : (C_2 \times C_2)) : C_2\)}No group isocategorical to this group exists.
\begin{table}[H]
	\begin{center}
		\caption{Group table for generators}
		%
	\end{center}
\end{table}

\subsection{Order list \([1,23,24,16,0,0,0]\)}
\subsubsection{\(C_2 \times ((C_8 \times C_2) : C_2)\)}No group isocategorical to this group exists.
\begin{table}[H]
	\begin{center}
		\caption{Group table for generators}
		%
	\end{center}
\end{table}
\subsubsection{\(((C_4 \times C_4) : C_2) : C_2\)}No group isocategorical to this group exists.
\begin{table}[H]
	\begin{center}
		\caption{Group table for generators}
		%
	\end{center}
\end{table}
\subsubsection{\(((C_4 \times C_2 \times C_2) : C_2) : C_2\)}No group isocategorical to this group exists.
\begin{table}[H]
	\begin{center}
		\caption{Group table for generators}
		%
	\end{center}
\end{table}
\subsubsection{\(((C_2 \times Q_8) : C_2) : C_2\)}No group isocategorical to this group exists.
\begin{table}[H]
	\begin{center}
		\caption{Group table for generators}
		%
	\end{center}
\end{table}
\subsubsection{\(((C_4 \times C_2 \times C_2) : C_2) : C_2\)}No group isocategorical to this group exists.
\begin{table}[H]
	\begin{center}
		\caption{Group table for generators}
		%
	\end{center}
\end{table}
\subsubsection{\(((C_4 \times C_2 \times C_2) : C_2) : C_2\)}No group isocategorical to this group exists.
\begin{table}[H]
	\begin{center}
		\caption{Group table for generators}
		%
	\end{center}
\end{table}
\subsubsection{\(C_2 \times C_2 \times QD_{16}\)}
\begin{table}[H]
	\begin{center}
		\caption{Group table for generators}
		%
	\end{center}
\end{table}
\subsubsection{\(C_2 \times ((C_8 \times C_2) : C_2)\)}No group isocategorical to this group exists.
\begin{table}[H]
	\begin{center}
		\caption{Group table for generators}
		%
	\end{center}
\end{table}
\subsubsection{\(((C_8 \times C_2) : C_2) : C_2\)}No group isocategorical to this group exists.
\begin{table}[H]
	\begin{center}
		\caption{Group table for generators}
		%
	\end{center}
\end{table}
\subsubsection{\((C_2 \times QD_{16}) : C_2\)}No group isocategorical to this group exists.
\begin{table}[H]
	\begin{center}
		\caption{Group table for generators}
		%
	\end{center}
\end{table}

\subsection{Order list \([1,23,40,0,0,0,0]\)}
\subsubsection{\((C_2 \times ((C_4 \times C_2) : C_2)) : C_2\)}No group isocategorical to this group exists.
\begin{table}[H]
	\begin{center}
		\caption{Group table for generators}
		%
	\end{center}
\end{table}
\subsubsection{\((C_2 \times (C_4 : C_4)) : C_2\)}No group isocategorical to this group exists.
\begin{table}[H]
	\begin{center}
		\caption{Group table for generators}
		%
	\end{center}
\end{table}
\subsubsection{\((C_2 \times ((C_4 \times C_2) : C_2)) : C_2\)}No group isocategorical to this group exists.
\begin{table}[H]
	\begin{center}
		\caption{Group table for generators}
		%
	\end{center}
\end{table}
\subsubsection{\(C_2 \times ((C_2 \times C_2 \times C_2) : C_4)
	\)}
\begin{table}[H]
	\begin{center}
		\caption{Group table for generators}
		%
	\end{center}
\end{table}
\subsubsection{\(C_2 \times C_4 \times D_8\)}No group isocategorical to this group exists.
\begin{table}[H]
	\begin{center}
		\caption{Group table for generators}
		%
	\end{center}
\end{table}
\subsubsection{\((C_2 \times ((C_4 \times C_2) : C_2)) : C_2\)}No group isocategorical to this group exists.
\begin{table}[H]
	\begin{center}
		\caption{Group table for generators}
		%
	\end{center}
\end{table}
\subsubsection{\(C_2 \times ((C_4 \times C_2 \times C_2) : C_2)\)}No group isocategorical to this group exists.
\begin{table}[H]
	\begin{center}
		\caption{Group table for generators}
		%
	\end{center}
\end{table}
\subsubsection{\((C_2 \times ((C_4 \times C_2) : C_2)) : C_2\)}No group isocategorical to this group exists.
\begin{table}[H]
	\begin{center}
		\caption{Group table for generators}
		%
	\end{center}
\end{table}
\subsubsection{\(C_2 \times ((C_4 \times C_4) : C_2)\)}No group isocategorical to this group exists.
\begin{table}[H]
	\begin{center}
		\caption{Group table for generators}
		%
	\end{center}
\end{table}
\subsubsection{\((C_2 \times ((C_4 \times C_2) : C_2)) : C_2\)}No group isocategorical to this group exists.
\begin{table}[H]
	\begin{center}
		\caption{Group table for generators}
		%
	\end{center}
\end{table}
\subsubsection{\((C_2 \times ((C_4 \times C_2) : C_2)) : C_2\)}No group isocategorical to this group exists.
\begin{table}[H]
	\begin{center}
		\caption{Group table for generators}
		%
	\end{center}
\end{table}
\subsubsection{\((C_4 \times D_8) : C_2\)}No group isocategorical to this group exists.
\begin{table}[H]
	\begin{center}
		\caption{Group table for generators}
		%
	\end{center}
\end{table}
\subsubsection{\((C_4 \times D_8) : C_2\)}No group isocategorical to this group exists.
\begin{table}[H]
	\begin{center}
		\caption{Group table for generators}
		%
	\end{center}
\end{table}
\subsubsection{\(C_2 \times ((C_2 \times Q_8) : C_2)\)}No group isocategorical to this group exists.
\begin{table}[H]
	\begin{center}
		\caption{Group table for generators}
		%
	\end{center}
\end{table}

\subsection{Order list \([1,27,12,8,16,0,0]\)}
\subsubsection{\(((C_8\times C_2):C_2):C_2\)}No group isocategorical to this group exists.
Since only one group belongs to this order list, there are no non-isomorphic groups that are monoidally equivalent to it.

\subsection{Order list \([1,27,20,16,0,0,0]\)}
\subsubsection{\(((C_4 \times C_4) : C_2) : C_2\)}No group isocategorical to this group exists.
\begin{table}[H]
	\begin{center}
		\caption{Group table for generators}
		%
	\end{center}
\end{table}
\subsubsection{\((C_2 \times D_{16}) : C_2\)}No group isocategorical to this group exists.
\begin{table}[H]
	\begin{center}
		\caption{Group table for generators}
		%
	\end{center}
\end{table}
\subsubsection{\(((C_4 \times C_4) : C_2) : C_2\)}No group isocategorical to this group exists.
\begin{table}[H]
	\begin{center}
		\caption{Group table for generators}
		%
	\end{center}
\end{table}

\subsection{Order list \([1,27,36,0,0,0,0]\)}
\subsubsection{\(((C_2 \times C_2 \times C_2 \times C_2) : C_2) : C_2\)}No group isocategorical to this group exists.
\begin{table}[H]
	\begin{center}
		\caption{Group table for generators}
		%
	\end{center}
\end{table}
\subsubsection{\((C_2 \times ((C_4 \times C_2) : C_2)) : C_2\)}No group isocategorical to this group exists.
\begin{table}[H]
	\begin{center}
		\caption{Group table for generators}
		%
	\end{center}
\end{table}
\subsubsection{\((C_2 \times ((C_4 \times C_2) : C_2)) : C_2\)}No group isocategorical to this group exists.
\begin{table}[H]
	\begin{center}
		\caption{Group table for generators}
		%
	\end{center}
\end{table}
\subsubsection{\(((C_4 \times C_2 \times C_2) : C_2) : C_2\)}No group isocategorical to this group exists.
\begin{table}[H]
	\begin{center}
		\caption{Group table for generators}
		%
	\end{center}
\end{table}
\subsubsection{\(((C_4 \times C_4) : C_2) : C_2\)}No group isocategorical to this group exists.
\begin{table}[H]
	\begin{center}
		\caption{Group table for generators}
		%
	\end{center}
\end{table}

\subsection{Order list \([1,31,16,16,0,0,0]\)}
\subsubsection{\(((C_4 \times C_2 \times C_2) : C_2) : C_2\)}No group isocategorical to this group exists.
\begin{table}[H]
	\begin{center}
		\caption{Group table for generators}
		%
	\end{center}
\end{table}
\subsubsection{\(C_2 \times (C_8 : (C_2 \times C_2))\)}No group isocategorical to this group exists.
\begin{table}[H]
	\begin{center}
		\caption{Group table for generators}
		%
	\end{center}
\end{table}
\subsubsection{\((C_2 \times D_{16}) : C_2\)}No group isocategorical to this group exists.
\begin{table}[H]
	\begin{center}
		\caption{Group table for generators}
		%
	\end{center}
\end{table}

\subsection{Order list \([1,31,32,0,0,0,0]\)}
\subsubsection{\((C_2 \times ((C_4 \times C_2) : C_2)) : C_2\)}No group isocategorical to this group exists.
\begin{table}[H]
	\begin{center}
		\caption{Group table for generators}
		%
	\end{center}
\end{table}
\subsubsection{\((C_2 \times C_2 \times D_8) : C_2\)}No group isocategorical to this group exists.
\begin{table}[H]
	\begin{center}
		\caption{Group table for generators}
		%
	\end{center}
\end{table}
\subsubsection{\(C_2 \times C_2 \times ((C_4 \times C_2) : C_2)\)}No group isocategorical to this group exists.
\begin{table}[H]
	\begin{center}
		\caption{Group table for generators}
		%

\subsubsection{\(C_2 \times ((C_4 \times C_2 \times C_2) : C_2)\)}No group isocategorical to this group exists.
\begin{table}[H]
	\begin{center}
		\caption{Group table for generators}
		%
	\end{center}
\end{table}
\subsubsection{\((C_2 \times C_2 \times D_8) : C_2\)}\label{G3}
It is monoidally equivalent to the sixth group in subsubsection \([1,31,32,0,0,0,0]\).
\begin{table}[H]
	\begin{center}
		\caption{Group table for generators}
		%
	\end{center}
\end{table}
\subsubsection{\((C_2 \times ((C_4 \times C_2) : C_2)) : C_2\)}\label{G4}
It is monoidally equivalent to the fifth group in subsubsection \([1,31,32,0,0,0,0]\).
\begin{table}[H]
	\begin{center}
		\caption{Group table for generators}
		%
	\end{center}
\end{table}
\subsubsection{\(C_4\times C_2\times C_2\times C_2\times C_2\)}No group isocategorical to this group exists.
As the group is abelian, no non-isomorphic groups exist that are isocategorical to it.

\subsubsection{\(C_2 \times C_2 \times ((C_4 \times C_2) : C_2)\)}No group isocategorical to this group exists.
\begin{table}[H]
	\begin{center}
		\caption{Group table for generators}
		%
	\end{center}
\end{table}
\subsubsection{\((C_2 \times ((C_4 \times C_2) : C_2)) : C_2\)}No group isocategorical to this group exists.
\begin{table}[H]
	\begin{center}
		\caption{Group table for generators}
		%
	\end{center}
\end{table}

\subsection{Order list \([1,33,2,4,8,16,0]\)}
\subsubsection{\(D_{64}\)}
Since only one group belongs to this order list, there are no non-isomorphic groups that are monoidally equivalent to it.

\subsection{Order list \([1,35,4,8,16,0,0]\)}
\subsubsection{\(C_2\times D_{32}\)}
Since only one group belongs to this order list, there are no non-isomorphic groups that are monoidally equivalent to it.

\subsection{Order list \([1,35,12,16,0,0,0]\)}
\subsubsection{\(((C_4\times C_4):C_2):C_2\)}
Since only one group belongs to this order list, there are no non-isomorphic groups that are monoidally equivalent to it.

\subsection{Order list \([1,35,28,0,0,0,0]\)}
\subsubsection{\(D_8\times D_8\)}
Since only one group belongs to this order list, there are no non-isomorphic groups that are monoidally equivalent to it.

\subsection{Order list \([1,39,8,16,0,0,0]\)}
\subsubsection{\(C_2\times C_2\times D_{16}\)}
Since only one group belongs to this order list, there are no non-isomorphic groups that are monoidally equivalent to it.

\subsection{Order list \([1,39,24,0,0,0,0]\)}
\subsubsection{\(C_2 \times ((C_2 \times C_2 \times C_2 \times C_2) : C_2)\)}No group isocategorical to this group exists.
\begin{table}[H]
	\begin{center}
		\caption{Group table for generators}
		\begin{tabular}{c||c|c|c|c|c|c|}
			& $f_1$ & $f_2$ & $f_3$ & $f_4$ & $f_5$ & $f_6$ \\ \hline\hline
			$f_1$ & $1$ & $f_1f_2$ & $f_1f_3$ & $f_1f_4$ & $f_1f_5$ & $f_1f_6$ \\ \hline
			$f_2$ & $f_1f_2f_5$ & $1$ & $f_2f_3$ & $f_2f_4$ & $f_2f_5$ & $f_2f_6$ \\ \hline
			$f_3$ & $f_1f_3f_6$ & $f_2f_3$ & $1$ & $f_3f_4$ & $f_3f_5$ & $f_3f_6$ \\ \hline
			$f_4$ & $f_1f_4$ & $f_2f_4$ & $f_3f_4$ & $1$ & $f_4f_5$ & $f_4f_6$ \\ \hline
			$f_5$ & $f_1f_5$ & $f_2f_5$ & $f_3f_5$ & $f_4f_5$ & $1$ & $f_5f_6$ \\ \hline
			$f_6$ & $f_1f_6$ & $f_2f_6$ & $f_3f_6$ & $f_4f_6$ & $f_5f_6$ & $1$ \\ \hline
		\end{tabular}
	\end{center}
\end{table}

\begin{table}[H]
	\begin{center}
		\caption{Candidate subgroups and the resulting twisted groups $G^\omega$}
		\begin{tabular}{|c|c|c|c|c|}
			Subgroup & Structure of \(N\) & \(G\)-inv & \(G^\omega\cong G\) & Number of \(G^\omega\)  \\ \hline
			
			$[ f_2, f_5 ]$ & $C_2\times C_2$ & $\bigcirc$ & $\bigcirc$ & 1 \\ \hline 
			$[ f_2f_3, f_5f_6 ]$ & $C_2\times C_2$ & $\bigcirc$ & $\bigcirc$ & 1 \\ \hline 
			$[ f_3, f_6 ]$ & $C_2\times C_2$ & $\bigcirc$ & $\bigcirc$ & 1 \\ \hline 
			$[ f_4f_6, f_5f_6 ]$ & $C_2\times C_2$ & $\bigcirc$ & $\bigcirc$ & 1 \\ \hline 
			$[ f_4, f_5f_6 ]$ & $C_2\times C_2$ & $\bigcirc$ & $\bigcirc$ & 1 \\ \hline 
			$[ f_4f_6, f_5 ]$ & $C_2\times C_2$ & $\bigcirc$ & $\bigcirc$ & 1 \\ \hline 
			$[ f_4, f_5 ]$ & $C_2\times C_2$ & $\bigcirc$ & $\bigcirc$ & 1 \\ \hline 
			$[ f_4f_5, f_6 ]$ & $C_2\times C_2$ & $\bigcirc$ & $\bigcirc$ & 1 \\ \hline 
			$[ f_4, f_6 ]$ & $C_2\times C_2$ & $\bigcirc$ & $\bigcirc$ & 1 \\ \hline 
			$[ f_3f_5, f_6 ]$ & $C_2\times C_2$ & $\bigcirc$ & $\bigcirc$ & 1 \\ \hline 
			$[ f_2f_4f_6, f_5 ]$ & $C_2\times C_2$ & $\bigcirc$ & $\bigcirc$ & 1 \\ \hline 
			$[ f_2f_4, f_5 ]$ & $C_2\times C_2$ & $\bigcirc$ & $\bigcirc$ & 1 \\ \hline 
			$[ f_2f_3f_6, f_5f_6 ]$ & $C_2\times C_2$ & $\bigcirc$ & $\bigcirc$ & 1 \\ \hline 
			$[ f_2f_6, f_5 ]$ & $C_2\times C_2$ & $\bigcirc$ & $\bigcirc$ & 1 \\ \hline 
			$[ f_3f_4f_5, f_6 ]$ & $C_2\times C_2$ & $\bigcirc$ & $\bigcirc$ & 1 \\ \hline 
			$[ f_3f_4, f_6 ]$ & $C_2\times C_2$ & $\bigcirc$ & $\bigcirc$ & 1 \\ \hline 
			$[ f_2f_3f_4f_6, f_5f_6 ]$ & $C_2\times C_2$ & $\bigcirc$ & $\bigcirc$ & 1 \\ \hline 
			$[ f_2f_3f_4, f_5f_6 ]$ & $C_2\times C_2$ & $\bigcirc$ & $\bigcirc$ & 1 \\ \hline 
			$[ f_5, f_6 ]$ & $C_2\times C_2$ & $\bigcirc$ & $\bigcirc$ & 1 \\ \hline 
			$[ f_5, f_6, f_2f_4, f_3 ]$ & $C_2\times C_2\times C_2\times C_2$ & 
			\begin{tabular}{c}
				1,2,3,4 \\
				5,6,11,12
			\end{tabular}
			& $\bigcirc$  &1   \\ \hline 
			$[ f_5, f_6, f_3, f_4 ]$ & $C_2\times C_2\times C_2\times C_2$ & 1,2,3,4  & $\bigcirc$  &1   \\ \hline 
			$[ f_5, f_6, f_2f_3, f_4 ]$ & $C_2\times C_2\times C_2\times C_2$ & 9,10,11,12  & $\bigcirc$  &1   \\ \hline 
			$[ f_5, f_6, f_1, f_4 ]$ & $C_2\times C_2\times C_2\times C_2$ & $\times$  & -  &-   \\ \hline 
			$[ f_5, f_6, f_2, f_4 ]$ & $C_2\times C_2\times C_2\times C_2$ & 5,6,7,8  & $\bigcirc$  & 1  \\ \hline 
			$[ f_5, f_6, f_2, f_3f_4 ]$ & $C_2\times C_2\times C_2\times C_2$ & 
			\begin{tabular}{c}
				1,2,3,4 \\
				5,6,11,12
			\end{tabular}
			& $\bigcirc$  &1   \\ \hline 
			$[ f_5, f_6, f_2, f_3 ]$ & $C_2\times C_2\times C_2\times C_2$ & 
			\begin{tabular}{c}
				1,2,3,4 \\
				5,6,11,12
			\end{tabular}
			& $\bigcirc$  &1   \\ \hline 
			$[ f_5, f_6, f_2f_4, f_3f_4 ]$ & $C_2\times C_2\times C_2\times C_2$ & 
			\begin{tabular}{c}
				1,2,3,4 \\
				5,6,11,12
			\end{tabular}
			& $\bigcirc$  &1   \\ \hline
		\end{tabular}
	\end{center}
\end{table}
\subsubsection{\(C_2 \times ((C_4 \times C_4) : C_2)\)}No group isocategorical to this group exists.
\begin{table}[H]
	\begin{center}
		\caption{Group table for generators}
		\begin{tabular}{c||c|c|c|c|c|c|}
			& $f_1$ & $f_2$ & $f_3$ & $f_4$ & $f_5$ & $f_6$ \\ \hline\hline
			$f_1$ & $1$ & $f_1f_2$ & $f_1f_3$ & $f_1f_4$ & $f_1f_5$ & $f_1f_6$ \\ \hline
			$f_2$ & $f_1f_2f_5$ & $f_5$ & $f_2f_3$ & $f_2f_4$ & $f_2f_5$ & $f_2f_6$ \\ \hline
			$f_3$ & $f_1f_3f_6$ & $f_2f_3$ & $f_6$ & $f_3f_4$ & $f_3f_5$ & $f_3f_6$ \\ \hline
			$f_4$ & $f_1f_4$ & $f_2f_4$ & $f_3f_4$ & $1$ & $f_4f_5$ & $f_4f_6$ \\ \hline
			$f_5$ & $f_1f_5$ & $f_2f_5$ & $f_3f_5$ & $f_4f_5$ & $1$ & $f_5f_6$ \\ \hline
			$f_6$ & $f_1f_6$ & $f_2f_6$ & $f_3f_6$ & $f_4f_6$ & $f_5f_6$ & $1$ \\ \hline
		\end{tabular}
	\end{center}
\end{table}

\begin{table}[H]
	\begin{center}
		\caption{Candidate subgroups and the resulting twisted groups $G^\omega$}
		\begin{tabular}{|c|c|c|c|c|}
			Subgroup & Structure of \(N\) & \(G\)-inv & \(G^\omega\cong G\) & Number of \(G^\omega\)  \\ \hline
			
			$[ f_4f_5, f_6 ]$ & $C_2\times C_2$ & $\bigcirc$ & $\bigcirc$ & 2 \\ \hline 
			$[ f_4f_6, f_5f_6 ]$ & $C_2\times C_2$ & $\bigcirc$ & $\bigcirc$ & 2 \\ \hline 
			$[ f_4, f_5f_6 ]$ & $C_2\times C_2$ & $\bigcirc$ & $\bigcirc$ & 2 \\ \hline 
			$[ f_4f_6, f_5 ]$ & $C_2\times C_2$ & $\bigcirc$ & $\bigcirc$ & 2 \\ \hline 
			$[ f_4, f_5 ]$ & $C_2\times C_2$ & $\bigcirc$ & $\bigcirc$ & 2 \\ \hline 
			$[ f_4, f_6 ]$ & $C_2\times C_2$ & $\bigcirc$ & $\bigcirc$ & 2 \\ \hline 
			$[ f_5, f_6 ]$ & $C_2\times C_2$ & $\bigcirc$ & $\bigcirc$ & 2 \\ \hline 
			$[ f_2f_4, f_3 ]$ & $C_4\times C_4$ & $\bigcirc$ & $\bigcirc$ & 2 \\ \hline 
			$[ f_5, f_6, f_1f_2f_3, f_4 ]$ & $C_2\times C_2\times C_2\times C_2$ &  $\times$ & -  &-   \\ \hline 
			$[ f_2, f_3f_4 ]$ & $C_4\times C_4$ & $\bigcirc$ & $\bigcirc$ & 2 \\ \hline 
			$[ f_5, f_6, f_1, f_4 ]$ & $C_2\times C_2\times C_2\times C_2$ & $\times$  &-   &-   \\ \hline 
			$[ f_5, f_6, f_1f_3, f_4 ]$ & $C_2\times C_2\times C_2\times C_2$ & $\times$  &-   &-   \\ \hline 
			$[ f_2, f_3 ]$ & $C_4\times C_4$ & $\bigcirc$ & $\bigcirc$ & 2 \\ \hline 
			$[ f_5, f_6, f_1f_2, f_4 ]$ & $C_2\times C_2\times C_2\times C_2$ &$\times$   &-   &-   \\ \hline 
			$[ f_2f_4, f_3f_4 ]$ & $C_4\times C_4$ & $\bigcirc$ & $\bigcirc$ & 2 \\ \hline 
		\end{tabular}
	\end{center}
\end{table}
\subsubsection{\(C_2 \times ((C_2 \times C_2 \times C_2) : (C_2 \times C_2))\)}No group isocategorical to this group exists.
\begin{table}[H]
	\begin{center}
		\caption{Group table for generators}
		\begin{tabular}{c||c|c|c|c|c|c|}
			& $f_1$ & $f_2$ & $f_3$ & $f_4$ & $f_5$ & $f_6$ \\ \hline\hline
			$f_1$ & $1$ & $f_1f_2$ & $f_1f_3$ & $f_1f_4$ & $f_1f_5$ & $f_1f_6$ \\ \hline
			$f_2$ & $f_1f_2f_6$ & $1$ & $f_2f_3$ & $f_2f_4$ & $f_2f_5$ & $f_2f_6$ \\ \hline
			$f_3$ & $f_1f_3$ & $f_2f_3f_6$ & $1$ & $f_3f_4$ & $f_3f_5$ & $f_3f_6$ \\ \hline
			$f_4$ & $f_1f_4f_6$ & $f_2f_4$ & $f_3f_4$ & $1$ & $f_4f_5$ & $f_4f_6$ \\ \hline
			$f_5$ & $f_1f_5$ & $f_2f_5$ & $f_3f_5$ & $f_4f_5$ & $1$ & $f_5f_6$ \\ \hline
			$f_6$ & $f_1f_6$ & $f_2f_6$ & $f_3f_6$ & $f_4f_6$ & $f_5f_6$ & $1$ \\ \hline
		\end{tabular}
	\end{center}
\end{table}

\begin{table}[H]
	\begin{center}
		\caption{Candidate subgroups and the resulting twisted groups $G^\omega$}
		\begin{tabular}{|c|c|c|c|c|}
			Subgroup & Structure of \(N\) & \(G\)-inv & \(G^\omega\cong G\) & Number of \(G^\omega\)  \\ \hline
			
			$[ f_3f_4, f_6 ]$ & $C_2\times C_2$ & $\bigcirc$ & $\bigcirc$ & 3 \\ \hline 
			$[ f_2f_4, f_6 ]$ & $C_2\times C_2$ & $\bigcirc$ & $\bigcirc$ & 3 \\ \hline 
			$[ f_5, f_6 ]$ & $C_2\times C_2$ & $\bigcirc$ & $\bigcirc$ & 3 \\ \hline 
			$[ f_4, f_6 ]$ & $C_2\times C_2$ & $\bigcirc$ & $\bigcirc$ & 3 \\ \hline 
			$[ f_1f_5, f_6 ]$ & $C_2\times C_2$ & $\bigcirc$ & $\bigcirc$ & 3 \\ \hline 
			$[ f_2f_5, f_6 ]$ & $C_2\times C_2$ & $\bigcirc$ & $\bigcirc$ & 3 \\ \hline 
			$[ f_3f_5, f_6 ]$ & $C_2\times C_2$ & $\bigcirc$ & $\bigcirc$ & 3 \\ \hline 
			$[ f_1f_3, f_6 ]$ & $C_2\times C_2$ & $\bigcirc$ & $\bigcirc$ & 3 \\ \hline 
			$[ f_1f_3f_5, f_6 ]$ & $C_2\times C_2$ & $\bigcirc$ & $\bigcirc$ & 3 \\ \hline 
			$[ f_1f_2f_4, f_6 ]$ & $C_2\times C_2$ & $\bigcirc$ & $\bigcirc$ & 3 \\ \hline 
			$[ f_3, f_6 ]$ & $C_2\times C_2$ & $\bigcirc$ & $\bigcirc$ & 3 \\ \hline 
			$[ f_4f_5, f_6 ]$ & $C_2\times C_2$ & $\bigcirc$ & $\bigcirc$ & 3 \\ \hline 
			$[ f_2f_4f_5, f_6 ]$ & $C_2\times C_2$ & $\bigcirc$ & $\bigcirc$ & 3 \\ \hline 
			$[ f_2, f_6 ]$ & $C_2\times C_2$ & $\bigcirc$ & $\bigcirc$ & 3 \\ \hline 
			$[ f_1f_2f_4f_5, f_6 ]$ & $C_2\times C_2$ & $\bigcirc$ & $\bigcirc$ & 3 \\ \hline 
			$[ f_1, f_6 ]$ & $C_2\times C_2$ & $\bigcirc$ & $\bigcirc$ & 3 \\ \hline 
			$[ f_3f_4f_5, f_6 ]$ & $C_2\times C_2$ & $\bigcirc$ & $\bigcirc$ & 3 \\ \hline 
			$[ f_1f_2f_3, f_6 ]$ & $C_2\times C_2$ & $\bigcirc$ & $\bigcirc$ & 3 \\ \hline 
			$[ f_1f_2f_3f_5, f_6 ]$ & $C_2\times C_2$ & $\bigcirc$ & $\bigcirc$ & 3 \\ \hline 
			$[ f_6, f_3, f_4, f_5 ]$ & $C_2\times C_2\times C_2\times C_2$ & $\times$  & -  &-   \\ \hline 
			$[ f_6, f_2, f_4, f_5 ]$ & $C_2\times C_2\times C_2\times C_2$ & $\times$  & -  &-   \\ \hline 
			$[ f_6, f_1, f_3, f_5 ]$ & $C_2\times C_2\times C_2\times C_2$ & $\times$ & -  &-   \\ \hline 
			$[ f_6, f_2, f_1f_3, f_5 ]$ & $C_2\times C_2\times C_2\times C_2$ & $\times$  &  - & -  \\ \hline 
			$[ f_6, f_1, f_2f_4, f_5 ]$ & $C_2\times C_2\times C_2\times C_2$ & $\times$ &  - &  - \\ \hline 
			$[ f_6, f_1f_2f_3, f_1f_2f_4, f_5 ]$ & $C_2\times C_2\times C_2\times C_2$ &  $\times$ & -  &-   \\ \hline 
		\end{tabular}
	\end{center}
\end{table}

\subsection{Order list \([1,47,16,0,0,0,0]\)}
\subsubsection{\(C_2\times C_2\times C_2\times D_8\)}
Since only one group belongs to this order list, there are no non-isomorphic groups that are monoidally equivalent to it.

\subsection{Order list \([1,63,0,0,0,0,0]\)}
\subsubsection{\(C_2\times C_2\times C_2\times C_2\times C_2\times C_2\)}
As the group is abelian, no non-isomorphic groups exist that are isocategorical to it.

\section{conclusion}
Based on the computational results presented in the previous sections, we establish the following theorem regarding the classification of isocategorical groups of order 64.
\begin{thm}\label{classification}
	There exist exactly two pairs of non-isomorphic groups of order 64 that are monoidally equivalent.
\end{thm}

The two pairs stated in Theorem \ref{classification} are described in Theorem \ref{pair1} and Theorem \ref{pair2}, respectively.

\begin{thm}\label{pair1}
	Let $G_1$ and $G_2$ be the groups generated by the elements defined in the following group tables.
	While $G_1$ and $G_2$ are not isomorphic as groups, they are monoidally equivalent.
	The order list for both groups is given by the list:
	$$
	\mathrm{List}(G_1)=\mathrm{List}(G_2)=\{1, 19, 44, 0, 0, 0, 0\}.
	$$
	Furthermore, $G_1$ is isomorphic to the Izumi-Kosaki group $\Gizumi$.
	\begin{table}[H]
		\begin{center}
			\caption{Group table for the generators of \(G_1\)}
			\begin{tabular}{c||c|c|c|c|c|c|}
				& $f_1$ & $f_2$ & $f_3$ & $f_4$ & $f_5$ & $f_6$ \\ \hline\hline
				$f_1$ & $f_5$ & $f_1f_2$ & $f_1f_3$ & $f_1f_4$ & $f_1f_5$ & $f_1f_6$ \\ \hline
				$f_2$ & $f_1f_2f_5$ & $f_6$ & $f_2f_3$ & $f_2f_4$ & $f_2f_5$ & $f_2f_6$ \\ \hline
				$f_3$ & $f_1f_3f_6$ & $f_2f_3$ & $1$ & $f_3f_4$ & $f_3f_5$ & $f_3f_6$ \\ \hline
				$f_4$ & $f_1f_4$ & $f_2f_4f_5$ & $f_3f_4$ & $1$ & $f_4f_5$ & $f_4f_6$ \\ \hline
				$f_5$ & $f_1f_5$ & $f_2f_5$ & $f_3f_5$ & $f_4f_5$ & $1$ & $f_5f_6$ \\ \hline
				$f_6$ & $f_1f_6$ & $f_2f_6$ & $f_3f_6$ & $f_4f_6$ & $f_5f_6$ & $1$ \\ \hline
			\end{tabular}
		\end{center}
	\end{table}
	\begin{table}[H]
		\begin{center}
			\caption{Group table for the generators of \(G_2\)}
			\begin{tabular}{c||c|c|c|c|c|c|}
				& $f_1$ & $f_2$ & $f_3$ & $f_4$ & $f_5$ & $f_6$ \\ \hline\hline
				$f_1$ & $1$ & $f_1f_2$ & $f_1f_3$ & $f_1f_4$ & $f_1f_5$ & $f_1f_6$ \\ \hline
				$f_2$ & $f_1f_2f_5$ & $f_6$ & $f_2f_3$ & $f_2f_4$ & $f_2f_5$ & $f_2f_6$ \\ \hline
				$f_3$ & $f_1f_3f_6$ & $f_2f_3$ & $f_5$ & $f_3f_4$ & $f_3f_5$ & $f_3f_6$ \\ \hline
				$f_4$ & $f_1f_4$ & $f_2f_4f_5$ & $f_3f_4$ & $f_5$ & $f_4f_5$ & $f_4f_6$ \\ \hline
				$f_5$ & $f_1f_5$ & $f_2f_5$ & $f_3f_5$ & $f_4f_5$ & $1$ & $f_5f_6$ \\ \hline
				$f_6$ & $f_1f_6$ & $f_2f_6$ & $f_3f_6$ & $f_4f_6$ & $f_5f_6$ & $1$ \\ \hline
			\end{tabular}
		\end{center}
	\end{table}
\end{thm}
\begin{proof}
	See subsubsections \ref{G1} and \ref{G2}.
\end{proof}

\begin{thm}\label{pair2}
	Let $G_3$ and $G_4$ be the groups generated by the elements defined in the following group tables.
	While $G_3$ and $G_4$ are not isomorphic as groups, they are monoidally equivalent.
	The order list for both groups is given by the list:
	$$
	\mathrm{List}(G_3)=\mathrm{List}(G_4)=\{1, 31, 32, 0, 0, 0, 0\}.
	$$
	\begin{table}[H]
		\begin{center}
			\caption{Group table for the generators of \(G_3\)}
			\begin{tabular}{c||c|c|c|c|c|c|}
				& $f_1$ & $f_2$ & $f_3$ & $f_4$ & $f_5$ & $f_6$ \\ \hline\hline
				$f_1$ & $1$ & $f_1f_2$ & $f_1f_3$ & $f_1f_4$ & $f_1f_5$ & $f_1f_6$ \\ \hline
				$f_2$ & $f_1f_2f_5$ & $1$ & $f_2f_3$ & $f_2f_4$ & $f_2f_5$ & $f_2f_6$ \\ \hline
				$f_3$ & $f_1f_3f_6$ & $f_2f_3f_5$ & $1$ & $f_3f_4$ & $f_3f_5$ & $f_3f_6$ \\ \hline
				$f_4$ & $f_1f_4f_5$ & $f_2f_4$ & $f_3f_4$ & $1$ & $f_4f_5$ & $f_4f_6$ \\ \hline
				$f_5$ & $f_1f_5$ & $f_2f_5$ & $f_3f_5$ & $f_4f_5$ & $1$ & $f_5f_6$ \\ \hline
				$f_6$ & $f_1f_6$ & $f_2f_6$ & $f_3f_6$ & $f_4f_6$ & $f_5f_6$ & $1$ \\ \hline
			\end{tabular}
		\end{center}
	\end{table}
	\begin{table}[H]
		\begin{center}
			\caption{Group table for the generators of \(G_4\)}
			\begin{tabular}{c||c|c|c|c|c|c|}
				& $f_1$ & $f_2$ & $f_3$ & $f_4$ & $f_5$ & $f_6$ \\ \hline\hline
				$f_1$ & $1$ & $f_1f_2$ & $f_1f_3$ & $f_1f_4$ & $f_1f_5$ & $f_1f_6$ \\ \hline
				$f_2$ & $f_1f_2f_5$ & $f_5$ & $f_2f_3$ & $f_2f_4$ & $f_2f_5$ & $f_2f_6$ \\ \hline
				$f_3$ & $f_1f_3f_6$ & $f_2f_3f_5$ & $1$ & $f_3f_4$ & $f_3f_5$ & $f_3f_6$ \\ \hline
				$f_4$ & $f_1f_4f_5$ & $f_2f_4$ & $f_3f_4$ & $1$ & $f_4f_5$ & $f_4f_6$ \\ \hline
				$f_5$ & $f_1f_5$ & $f_2f_5$ & $f_3f_5$ & $f_4f_5$ & $1$ & $f_5f_6$ \\ \hline
				$f_6$ & $f_1f_6$ & $f_2f_6$ & $f_3f_6$ & $f_4f_6$ & $f_5f_6$ & $1$ \\ \hline
			\end{tabular}
		\end{center}
	\end{table}
\end{thm}
\begin{proof}
	See subsubsections \ref{G3} and \ref{G4}.
\end{proof}


\end{document}